\documentclass[11pt]{amsart}

\usepackage{pinlabel,latexsym,amssymb,titletoc}

\usepackage[all]{xy}
\SelectTips{cm}{}

\theoremstyle{plain}
\newtheorem{theorem}{Theorem}[section]
\newtheorem{lemma}[theorem]{Lemma}
\newtheorem{proposition}[theorem]{Proposition}
\newtheorem{corollary}[theorem]{Corollary}
\newtheorem{definition}[theorem]{Definition}

\theoremstyle{definition}
\newtheorem{notation}[theorem]{Notation}
\newtheorem{example}[theorem]{Example}
\newtheorem{exercice}[theorem]{Exercice}
\newtheorem{remark}[theorem]{Remark}

\numberwithin{equation}{section}
\numberwithin{figure}{section}

\def \N{\mathbb{N}}
\def \Z{\mathbb{Z}}
\def \Q{\mathbb{Q}}
\def \C{\mathbb{C}}
\def \R{\mathbb{R}}
\def \F{\mathbb{F}}

\def \Aut{\operatorname{Aut}}
\def \Coker{\operatorname{Coker}}
\def \Diff{\operatorname{Diff}}
\def \Eul{\operatorname{Eul}}
\def \Free{\operatorname{F}}
\def \Hom{\operatorname{Hom}}
\def \IAut{\operatorname{IAut}}
\def \Id{\operatorname{Id}}
\def \incl{\operatorname{incl}}
\def \int{\operatorname{int}}
\def \K{\operatorname{K}}
\def \Ker{\operatorname{Ker}}
\def \length{\operatorname{length}}
\def \Magnus{\operatorname{M}}
\def \Magnusab{\operatorname{M}_{\operatorname{ab}}}
\def \ord{\operatorname{ord}}
\def \Quad{\operatorname{Quad}}
\def \rank{\operatorname{rank}}
\def \sgn{\operatorname{sgn}}
\def \span{\operatorname{span}}
\def \Tors{\operatorname{Tors}}

\def \mcg{\mathcal{M}}

\def \ie{i$.$e$.$ }

\newcommand{\Thurston}[1]{\Vert #1 \Vert_T}
\newcommand{\torsion}[1]{\Vert #1 \Vert_\tau}
\newcommand{\Alexander}[1]{\Vert #1 \Vert_A}

\title[]{An introduction to the abelian Reidemeister torsion of three-dimensional manifolds}

\date{October 12, 2010}
 
\author[]{Gw\'ena\"el Massuyeau}

\address{IRMA\\
Universit\'e de Strasbourg \& CNRS\\
7 rue Ren\'e Descartes\\ 
67084 Strasbourg\\ 
France}

\email{massuyeau@math.unistra.fr}


\thanks{The author thanks Vladimir Turaev for comments on an earlier version of these notes.}


\begin{document}

\begin{abstract}
These notes accompany some lectures given at the autumn school ``\emph{Tresses in Pau}'' in October 2009.
The abelian Reidemeister torsion for $3$-manifolds, and its refinements by Turaev, are introduced.
Some applications, including relations between the Reidemeister torsion and other classical invariants, are surveyed.
\end{abstract}

\maketitle

\setcounter{tocdepth}{1}
{\small \tableofcontents} 

\section{Introduction}

The Reidemeister torsion and the Alexander polynomial are among the most classical invariants of $3$-manifolds. 
The first one was introduced in 1935 by Kurt Reidemeister \cite{Reidemeister} 
in order to classify lens spaces -- those closed $3$-manifolds which
are obtained by gluing two copies of the solid torus  along their boundary.
The second invariant is even  older since it dates back to 1928, 
when James Alexander defined it for knot and link complements \cite{Alexander}.
Those two invariants can be extended to higher dimensions and 
refer by their very definition to the maximal \emph{abelian} cover of the $3$-manifold.

In 1962, John Milnor interpreted in \cite{Milnor_duality} the Alexander polynomial of a link as a  kind of Reidemeister torsion 
(with coefficients in a field of rational fractions, instead of a cyclotomic field as was the case for lens spaces). 
This new viewpoint on the Alexander polynomial clarified its properties, among which its symmetry. 
This approach to study the Alexander polynomial  was systematized by Vladimir Turaev in \cite{Turaev_Alexander,Turaev_knot},
where he re-proved most of the known properties of the Alexander polynomial of $3$-manifolds and links
using general properties of the Reidemeister torsion. 
Vladimir Turaev also defined a kind of  ``maximal'' torsion, which is universal 
among Reidemeister torsions with abelian coefficients \cite{Turaev_Alexander_torsion}.
If $H$ denotes the first homology group of a compact $3$-manifold $M$, his invariant is an element
$$
\tau(M) \in Q(\Z[H])/\pm H
$$
of the ring of fractions of the group ring $\Z[H]$, up to some multiplicative indeterminacy. 
He also  explained how the sign ambiguity $\pm 1$ and the ambiguity in $H$ 
can be fixed by endowing $M$ with two extra structures which he introduced on this purpose
(a ``homological orientation'' \cite{Turaev_knot} and an ``Euler structure'' \cite{Turaev_Euler}).

These notes are aimed at introducing those invariants of $3$-manifolds by  Alexander, 
Reidemeister,  Milnor and  Turaev, and at presenting some of their properties and applications.
This will also give us the opportunity to present a few aspects of $3$-dimensional topology.
We have focussed on closed oriented $3$-manifolds to simplify the exposition,
although most of the material extends to compact $3$-manifolds with toroidal boundary (which include link complements).
Besides, we have restricted to \emph{abelian} Reidemeister torsions. 
In particular, geometrical aspects of Reidemeister torsions are not considered  (see \cite{Porti} for hyperbolic $3$-manifolds)
and we have not dealt with twisted Alexander polynomials (see \cite{FV} for a survey).
The theme of simple homotopy and Whitehead torsion is no treated neither
and, for this, we refer to Milnor's survey \cite{Milnor_Whitehead} and Cohen's book \cite{Cohen}.
For further reading on abelian Reidemeister torsions, 
we recommend Turaev's books \cite{Turaev_book,Turaev_book_dim_3},
which we used to prepare these notes, as well as Nicolaescu's book \cite{Nicolaescu}.

Throughout the notes, we use the following notation.

\begin{notation} 
An abelian group $G$ is denoted additively, except when it is thought of as a subset of the group ring $\C[G]$,
in which case multiplicative notation is used. 
Besides, (co)homology groups are taken with $\Z$ coefficients if no coefficients are  specified. 
\end{notation}

Moreover, as already mentioned, we are mainly interested in closed oriented connected topological $3$-manifolds.
So, without further mention, the term \emph{$3$-manifold} will refer to a  manifold of this kind.

\section{Three-dimensional manifolds}

We start by recalling a few general facts about $3$-manifolds.
We also introduce an important class of $3$-manifolds, namely the lens spaces.

\subsection{About triangulations and smooth structures}

One way to present a $3$-manifold $M$ is as the result of gluing finitely many tetrahedra, face by face, using affine homeomorphisms. 
For instance, $S^3$ can be identified with the boundary of a $4$-dimensional simplex, 
and so appears as the union of $5$ tetrahedra. In general, any $3$-manifold can be presented in that way:

\begin{theorem}[Triangulation in dimension $3$]
\label{th:triangulation}
For any $3$-manifold $M$, there exists a simplicial complex $K$ and a homeomorphism $\rho: |K| \to M$.
\end{theorem}

\noindent
The pair $(K,\rho)$ is called a \emph{triangulation} of the manifold $M$.
The reader is refered to \cite{RS} for an introduction to piecewise-linear topology  and its terminology.

Furthermore, the triangulation of a $3$-manifold is essentially unique.

\begin{theorem}[Hauptvermutung in dimension $3$]
\label{th:hauptvermutung}
For any two triangulations $(K_1,\rho_1)$ and $(K_2,\rho_2)$ of a $3$-manifold $M$,
the homeomorphism ${\rho_2}^{-1} \circ \rho_1:|K_1| \to |K_2|$ is homotopic to a piecewise-linear homeomorphism.
In other words, there exist subdivisions $K'_1 \leq K_1$ and $K'_2 \leq K_2$ 
and a simplicial isomorphism $f:K'_1 \to K'_2$ such that $|f|: |K_1| \to |K_2|$ is homotopic to ${\rho_2}^{-1} \circ \rho_1$.
\end{theorem}

Those two theorems, which are due to Moise \cite{Moise}, show that
piecewise-linear invariants give rise to topological invariants in dimension $3$.
As shown by Munkres \cite{Munkres} and Whitehead \cite{Whitehead_smooth},
these results also imply that every $3$-manifold has an essentially unique smooth structure.
The reader is refered to \cite[\S 3.10]{Thurston_book}.

\begin{theorem}[Smoothing in dimension $3$]
\label{th:smooth}
Any $3$-manifold $M$ has a smooth structure, and any homeomorphism
between smooth $3$-manifolds is homotopic to a diffeomorphism.
\end{theorem}

\subsection{Heegaard splittings}

A \emph{handlebody} of \emph{genus} $g$ is a compact oriented $3$-manifold  
(with boundary) obtained by attaching $g$ handles to a $3$-dimensional ball: 
\begin{center}
\labellist \small \hair 2pt
\pinlabel {$1$} [br] at 48 212
\pinlabel {$g$} [bl] at 383 210
\pinlabel {$\cdots$} at 215 193
\endlabellist
\centering
\includegraphics[scale=0.45]{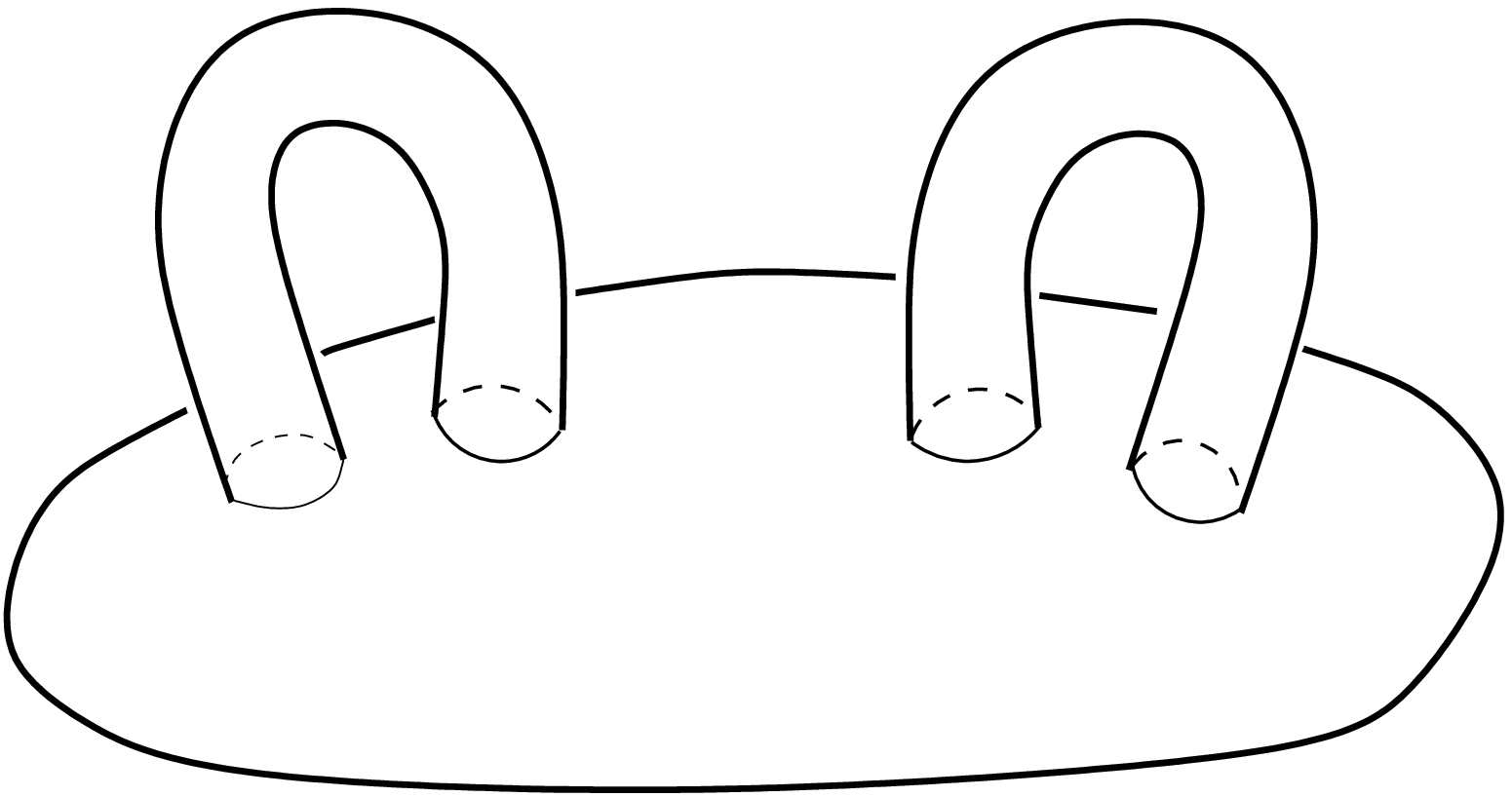}
\end{center}
(Those handles have  ``index $1$'' in the terminology of handle decompositions.)
It is easily checked that any two handlebodies are homeomorphic if, and only if, they have the same genus.
We fix  a ``standard'' handlebody of genus $g$, which we denote by $H_g$. We also set
$$
\Sigma_g := \partial H_g
$$
which is our ``standard'' closed oriented connected surface of genus $g$.

\begin{definition}
A \emph{Heegaard splitting} of \emph{genus} $g$ for a $3$-manifold $M$ is a decomposition $M=A \cup B$ 
into two handlebodies $A$ and $B$ of genus $g$ such that $A\cap B= \partial A = -\partial B$.
\end{definition}

\noindent
For example, the decomposition of the $3$-sphere
$$
S^3 = \{x=(x_0,x_1,x_2,x_3)\in \R^4 : x_0^2 + x_1^2 + x_2^2 + x_3^2  = 1\}
$$
into two hemispheres
$$
D^3_+ := \{x \in S^3: x_3\geq 0\} \quad \hbox{and} \quad
D^3_- := \{x  \in S^3: x_3\leq 0\}
$$
gives a genus $0$ Heegaard splitting of $S^3$.

\begin{proposition}
\label{prop:Heegaard}
Any $3$-manifold has a Heegaard splitting.
\end{proposition}

\begin{proof} 
Here is the classical argument.
Let $M$ be a $3$-manifold and let $K$ be a triangulation of $M$: $|K|=M$. 
Then, we notice in $M$ two embedded graphs: 
one is the $1$-skeleton $a$ of $K$ and the other one is the $1$-skeleton $b$ of the cell decomposition dual to $K$.
We fatten the graph $a$ to a $3$-dimensional submanifold $A \subset M$: 
each vertex of $a$ is fattened to a ball, while each edge of $a$ is fattened 
to a handle connecting two such balls. By choosing a maximal subtree of the graph $a$, we see that $A$ is a handlebody. 
Similarly, we fatten the graph $b$ to a handlebody $B \subset M$.
Moreover, we can choose $A$ and $B$ fat enough so that $M =A \cup B$ and $A\cap B= \partial A = -\partial B$.
\end{proof}

Up to homeomorphism, any Heegaard splitting of genus $g$ can be written
in terms of the ``standard'' genus $g$ handlebody as
$$
M_f := H_g \cup_f (-H_g)
$$
where $f:\Sigma_g \to \Sigma_g$ is an orientation-preserving homeomorphism.
This viewpoint leads us to consider the \emph{mapping class group} of the surface $\Sigma_g$, \ie the group
$$
\mcg(\Sigma_g) := \operatorname{Homeo}^{+}(\Sigma_g)/ \cong
$$
of orientation-preserving homeomorphisms $\Sigma_g \to \Sigma_g$ up to isotopy. 

\begin{lemma}
The homeomorphism class of $M_f$ only depends on  $[f] \in \mcg(\Sigma_g)$.
\end{lemma}

\begin{proof}
For any orientation-preserving homeomorphisms $E:H_g \to H_g$ and $f:\Sigma_g \to \Sigma_g$, we have
\begin{equation}
\label{eq:extension}
M_{f \circ E|_{\Sigma_g}} \cong_+ M_f \cong_+ M_{E|_{\Sigma_g} \circ f}.
\end{equation}
If $e:\Sigma_g \to \Sigma_g$ is a homeomorphism isotopic to the identity, 
we can use a collar neighborhood of $\Sigma_g$ in $H_g$ to construct a homeomorphism $E:H_g \to H_g$
such that $E|_{\Sigma_g}=e$. So, we obtain $M_{f\circ e}\cong_+ M_f$.
\end{proof}

The previous lemma, together with the next proposition, shows that there is only one Heegaard splitting of genus $0$,
namely $S^3 \subset \R^4$  decomposed into two hemispheres.

\begin{proposition}[Alexander]
\label{prop:Alexander}
The group $\mcg(S^2)$ is trivial.
\end{proposition}

\begin{proof} We start with the following
\begin{quote}
\textbf{Claim.} {Let $n\geq 1$ be an integer.
Any homeomorphism $h:D^n \to D^n$ which restricts to the identity on the boundary,
is isotopic to the identity of $D^n$ relatively to $\partial D^n$.}
\end{quote}

\noindent
Indeed, for all $t\in [0,1]$, we  define a homeomorphism $h_t:D^n \to D^n$ by 
$$
h_t(x) := \left\{\begin{array}{ll}
t\cdot h(x/t) & \hbox{if } 0 \leq \Vert x \Vert \leq t,\\
x & \hbox{if } t \leq \Vert x \Vert \leq 1.
\end{array}\right.
$$
Here $D^n$ is seen as a subset of $\R^n$ and $\Vert x \Vert$ is the euclidian norm of $x\in \R^n$. Then, the map
$$
H: D^n \times [0,1] \longrightarrow D^n, \ (x,t) \longmapsto h_t(x)
$$
is an isotopy (relatively to $\partial D^n$) between $h$ and the identity. 
This way of proving the claim is known as the ``Alexander's trick''.

Now, let $f:S^2 \to S^2$ be an orientation-preserving homeomorphism.
We are asked to prove that $f$ is isotopic to the identity. 
We choose an oriented simple closed curve $\gamma \subset S^2$.
It follows from the Jordan--Sch\"onflies theorem that the curves $f(\gamma)$ and $\gamma$ are isotopic
and, so, they are ambiently isotopic.
Therefore we can assume that $f(\gamma)=\gamma$.
Since any orientation-preserving homeomorphism $S^1\to S^1$ is isotopic to the identity 
(as can be deduced from the claim for $n=1$), 
we can  assume that $f|_\gamma$ is the identity of $\gamma$.
Then, because $\gamma$ splits $S^2$ into two disks -- say $D$ and $D'$ -- 
it is enough to show that $f|_D:D \to D$ is isotopic to the identity of $D$ relatively to $\partial D$,
and similarly for $D'$. But, this is an application for $n=2$ of the above claim.
\end{proof}

\subsection{Example: lens spaces}

In contrast with the genus $0$ case, there are infinitely many $3$-manifolds
with a Heegaard splitting of genus $1$. Such manifolds are called \emph{lens spaces}.
In order to enumerate them, we need to determine the mapping class group of a $2$-dimensional torus.

\begin{proposition}
\label{prop:torus}
Let $(a,a^\sharp)$ be the basis of $H_1(S^1 \times S^1)$ defined by
$a := [S^1\times 1]$ and $a^\sharp:=[1 \times S^1]$. Then, the map
$$
\mcg(S^1\times S^1) \longrightarrow \operatorname{SL}(2;\Z)
$$ 
which sends an $[f]$ to the matrix of $f_*:H_1(S^1 \times S^1) \to H_1(S^1 \times S^1)$
relative to the basis $(a,a^\sharp)$, is a group isomorphism.
\end{proposition}

\begin{proof}
Here is the classical argument.
The fact that we have a group homomorphism $\mcg(S^1\times S^1) \to \operatorname{GL}(2;\Z)$ is clear.
For all $[f]\in \mcg(S^1\times S^1)$, 
the matrix of $f_*$ has determinant one for the following reason:
$f$ preserves the orientation, 
so that $f_*$ leaves invariant the intersection pairing  $H_1(S^1\times S^1) \times H_1(S^1\times S^1) \to \Z$.

The surjectivity of the homomorphism can be proved as follows. We realize $S^1\times S^1$ as $\R^2/ \Z^2$,
in such a way that the loop $S^1\times 1$ lifts to $[0,1]\times 0$ and $1 \times S^1$ lifts to $0 \times [0,1]$. 
Any matrix $T \in \operatorname{SL}(2;\Z)$ defines a linear homeomorphism $\R^2 \to \R^2$,
which leaves $\Z^2$ globally invariant and so induces an (orientation-preserving) homeomorphism 
$t:\R^2/\Z^2 \to \R^2/\Z^2$. It is easily checked that the matrix of $t_*$ is exactly $T$. 

To prove the injectivity, let $f:S^1\times S^1 \to S^1\times S^1$ be a homeomorphism 
whose corresponding matrix is trivial. Since $\pi_1(S^1 \times S^1)$ is abelian, this implies that $f$ acts trivially at the level of 
the fundamental group. The canonical projection $\R^2 \to \R^2/\Z^2$ gives the universal cover of $S^1\times S^1$.
Thus, $f$ can be lifted to a unique homeomorphism $\widetilde{f}: \R^2 \to \R^2$ such that $\widetilde{f}(0)=0$ and, 
by our assumption on $f$, $\widetilde{f}$ is $\Z^2$-equivariant. Therefore, the ``affine'' homotopy
$$
H: \R^2 \times [0,1] \longrightarrow \R^2, \ (x,t) \longmapsto t\cdot\widetilde{f}(x) +(1-t) \cdot x
$$
between the identity of $\R^2$ and $\widetilde{f}$, descends to a homotopy between $\Id_{S^1\times S^1}$ and $f$.
An old result of Baer asserts that two self-homeomorphisms of a closed surface are homotopic
if and only if they are isotopic \cite{Baer1,Baer2}, so we deduce that $[f] = 1\in \mcg(S^1 \times S^1)$. 
\end{proof}

In the sequel, we identify the standard handlebody $H_1$ with the solid torus $D^2 \times S^1$,
so that $\Sigma_1$ is identified with $S^1 \times S^1$. A  Heegaard splitting of genus $1$ 
$$
M_f:= \left(D^2 \times S^1\right) \cup_f \left(-D^2 \times S^1\right)
$$
is thus encoded by $4$ parameters $p,q,r,s \in \Z$ such that
$$
\hbox{matrix of $f_*$} = \left(\begin{array}{cc} q & s\\ p & r 
\end{array}\right) 
\quad \hbox{with} \ qr-ps=1.
$$
So, $3$-manifolds having a genus $1$ Heegaard splitting can be indexed by  quadruplets $(p,q,r,s)$ of that sort, 
but many repetitions then occur. Indeed, let $E_k:H_1 \to H_1$ (for $k\in \Z$) 
be the self-homeomorphism of $H_1 =  D^2 \times S^1 \subset \C \times \C$ 
defined by $E_k(z_1,z_2) := (z_1z_2^k,z_2)$. We have
$$
\hbox{matrix of $\left(E_k|_{S^1 \times S^1}\right)_*$ }
=\left(\begin{array}{cc} 1 & k \\ 0 & 1
\end{array}\right)
$$
so that 
$$
\hbox{matrix of $\left(f \circ E_k|_{S^1 \times S^1}\right)_*$}
=\left(\begin{array}{cc} q & kq + s \\ p & kp +r
\end{array}\right).
$$
We deduce from (\ref{eq:extension}) that the homeomorphism class of $M_f$ 
only depends on the pair of parameters $(p,q)$. Furthermore, we have
$$
\hbox{matrix of $\left(E_k|_{S^1 \times S^1} \circ f\right)_*$}
=\left(\begin{array}{cc} q + pk & s +rk \\ p & r \end{array}\right)
$$
so that (for a given $p$) only the residue class of $q$ modulo $p$ does matter. 
Finally, the self-homeomorphism $C: H_1 \to H_1$ of 
$H_1 =  D^2 \times S^1 \subset \C \times \C$ defined by $C(z_1,z_2) := (\overline{z_1},\overline{z_2})$ satisfies
$$
\hbox{matrix of $\left(C|_{S^1 \times S^1}\right)_*$ }
=\left(\begin{array}{cc} -1 & 0 \\ 0 & -1
\end{array}\right)
$$
so that the $3$-manifolds corresponding to the pairs $(p,q)$ and $(-p,-q)$ are orientation-preserving homeomorphic.
Thus, we can assume that $p\geq 0$.

\begin{definition}
Let $p\geq 0$ be an integer and let $q$ be an invertible element of $\Z_p$.
The \emph{lens space} of parameters $(p,q)$ is the $3$-manifold 
$$
L_{p,q} :=  \left(D^2 \times S^1\right) \cup_f \left(-D^2 \times S^1\right)
$$
where $f:S^1 \times S^1 \to S^1 \times S^1$ is an orientation-preserving homeomorphism
such that $f_*(a) = q \cdot a + p\cdot a^\sharp$ (in the notation of Proposition \ref{prop:torus}). 
\end{definition}

The cases $p=0$ and $p=1$ are special. For $p=0$, $q\in \Z_0=\Z$ must take the value $+1$ or $-1$
but, by the previous discussion, we have $L_{0,1} \cong_+ L_{0,-1}$.
For $p=1$, there is no choice for $q\in \Z_1 =\{0\}$.
Observe that $L_{0,1} \cong S^2 \times S^1$ (by decomposing $S^2$ into two hemispheres $D^2_+$ and $D^2_-$)
and that $L_{1,0} \cong S^3$ (by identifying $S^3$ with the boundary of $D^2 \times D^2$). 
Thus, the topological classification of lens spaces is interesting only for $p\geq 2$.

\begin{theorem}[Reidemeister \cite{Reidemeister}]
\label{th:lens}
For any integers $p,p'\geq 2$ and invertible residue classes $q \in \Z_p, q' \in \Z_{p'}$, we have
$$
L_{p,q} \cong_+ L_{p',q'} \Longleftrightarrow \left(p=p' \hbox{ and } q' = q^{\pm 1}\right).
$$
\end{theorem}

\noindent
The direction ``$\Leftarrow$'' is easily checked by exchanging the two solid tori 
in the Heegaard splitting of $L_{p,q}$. 
The direction ``$\Rightarrow$'' needs topological invariants 
and will be proved in \S \ref{sec:torsion_dim_3} by means of the Reidemeister torsion.

\begin{exercice}
The terminology ``lens spaces'' (which dates back to Seifert and Threlfall \cite{ST})
is justified by the following equivalent description. 
Draw the planar regular $p$-gone $G_p$ in the plane $\R^{2} \times \{0\}$, 
with cyclically-ordered vertices $(v_i)_{i \in \Z_p}$:

$$
\labellist \small \hair 2pt
\pinlabel {$v_0$} [l] at 153 67
\pinlabel {$v_1$} [bl] at 113 131
\pinlabel {$v_2$} [br] at 39 131
\pinlabel {$v_3$} [r] at 0 65
\pinlabel {$v_4$} [tr] at 38 0
\pinlabel {$v_5$} [tl] at 116 0
\pinlabel {$G_6$} at 77 66
\endlabellist
\includegraphics[scale=0.35]{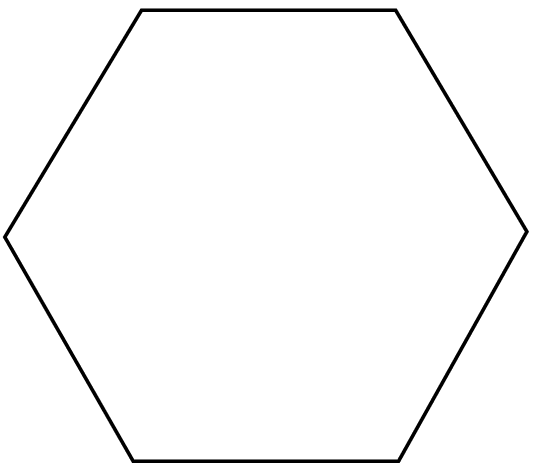}
$$
\vspace{0cm}

\noindent
Next, consider the polyhedron $B_p \subset \R^3$ whose boundary is
the bicone with base $G_p$ and with vertices $n:= (0,0,1)$ (the ``north pole'') and $s:=(0,0,-1)$ (the ``south pole''):

$$
\labellist \small \hair 2pt
\pinlabel {$n$} [b] at 75 123
\pinlabel {$s$} [t] at 76 0
\pinlabel {$B_6$}  at 136 10
\endlabellist
\includegraphics[scale=0.55]{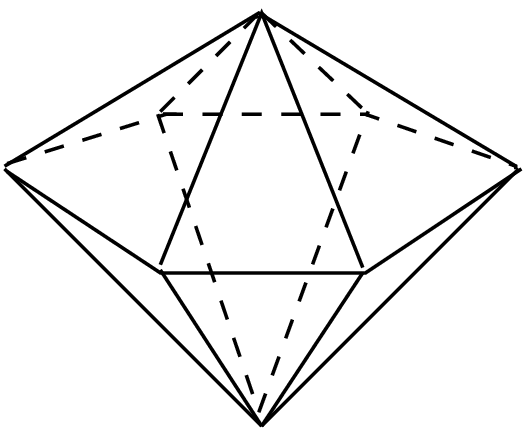}
$$
Let $\sim_q$ be the equivalence relation in $B_p$ such that $n\sim_q s$
and $v_{i} \sim_q v_{i+q}$ for all $i\in \Z_p$, 
and which identifies linearly the north face $(v_i,v_{i+1},n)$ of $B_p$ with its south face $(v_{i+q},v_{i+q+1},s)$.
Show that the quotient space $B_p/\!\sim_q$ is homeomorphic to $L_{p,q}$
and deduce that $L_{2,1}$ is homeomorphic to the projective space $\R P^3$.
\end{exercice}

\begin{exercice}
\label{ex:quotient_sphere}
Here is yet another description of lens spaces.
Consider the $3$-sphere $S^3= \{(z,z') \in \C^2 : |z|^2 + |z'|^2=1 \}$,
and let $\Z_p = \{\zeta \in \C : \zeta^p=1\}$ act on $S^3$ by 
$$
\forall \zeta \in \Z_p, \ \forall (z,z') \in S^3, \quad
\zeta \cdot (z,z') := \left(\zeta  z, \zeta^q z'\right).
$$
Show that the quotient space $S^3/\Z_p$ is homeomorphic to $L_{p,q}$.
\end{exercice}

\begin{exercice}
Check that, for any integer $p\geq 0$ and for any invertible $q\in \Z_p$, we have
$L_{p,-q} \cong_+ -L_{p,q}$. Deduce from Theorem \ref{th:lens}
that $L_{p,q}$ has an orientation-reversing self-homeomorphism
if, and only if, $q^2 = -1 \in \Z_p$. 
\end{exercice}

\section{The abelian Reidemeister torsion}

\label{sec:torsion_CW}

We aim at introducing the abelian Reidemeister torsion of $3$-manifolds.
This is a \emph{combinatorial} invariant in the sense that it is defined (and can be computed) 
from triangulations or, more generally, from cell decompositions.
Thus, we start by introducing the abelian Reidemeister torsion of CW-complexes.

\subsection{The torsion of a chain complex}

\label{subsec:torsion_chain_complex}

Let $\F$ be a commutative field, and let $C$ be a finite-dimensional chain complex over $\F$:
$$
C = \left( 0 \to C_m \stackrel{\partial_m}{\longrightarrow} C_{m-1} 
\stackrel{\partial_{m-1}}{\longrightarrow} \cdots \stackrel{\partial_{2}}{\longrightarrow} 
C_1 \stackrel{\partial_{1}}{\longrightarrow} C_0 \to 0\right).
$$
We assume that $C$ is acyclic and \emph{based}
in the sense that we are given a basis $c_i$ of $C_i$ for each $i=0,\dots,m$.

We denote by $B_i\subset C_i$ the image of $\partial_{i+1}$ and, for each $i$,
we choose a basis $b_i$ of $B_i$. The short exact sequence of $\F$-vector spaces
$$
0 \to B_{i} \longrightarrow C_i \stackrel{\partial_i}{\longrightarrow} B_{i-1} \to 0
$$
shows that we can obtain a new basis of $C_i$ by taking, first, the vectors of $b_i$
and, second, some lifts $\widetilde{b}_{i-1}$ of the vectors $b_{i-1}$. 
We denote by $b_i \widetilde{b}_{i-1}$ this new basis, and we compare it to $c_i$ by computing
$$
[b_i b_{i-1}/c_i] := \det
\left(\begin{array}{c}
\hbox{matrix expressing}\\
\hbox{$b_i \widetilde{b}_{i-1}$ in the basis $c_i$} 
\end{array}\right) \ \in \F\setminus \{0\}.
$$
This scalar does not depend on the choice of the lift $\widetilde{b}_{i-1}$ of $b_{i-1}$.

\begin{definition}
The \emph{torsion} of $C$ based by $c=(c_0,\dots,c_m)$ is
$$
\tau(C,c) := \prod_{i=0}^m \ [b_i b_{i-1}/c_i]^{(-1)^{i+1}} \ \in \F\setminus\{0\}.
$$
\end{definition}

\noindent
One easily checks that $\tau(C,c)$ does not depend on the choice of $b_0,\dots,b_m$.

\begin{remark}
\label{rem:non-commutative}
The torsion can also be defined for an acyclic free chain complex $C$ over an associative ring $\Lambda$.
Still, we must assume that the rank of free $\Lambda$-modules is well defined, 
\ie $\Lambda^r$ is not isomorphic to $\Lambda^s$ for $r\neq s$.
Then, the torsion of $C$ based by $c$ is defined, without taking determinants, as an element of 
$$
K_1(\Lambda)
:= \left( \hbox{the abelianization of $\hbox{GL}(\Lambda) = {\displaystyle \lim_{\longrightarrow}}\ \hbox{GL}(\Lambda;n)$}\right). 
$$
(Let us note that, when $\Lambda$ is a commutative field, 
the determinant provides an isomorphism between $K_1(\Lambda)$ and $\Lambda \setminus \{0\}$.)
This generalization is for instance needed in the definition of 
the Whitehead torsion of a homotopy equivalence  between finite CW-complexes,
which is the obstruction for it to be \emph{simple}, 
\ie to be homotopic to a finite sequence of elementary ``collapses'' and ``expansions''.
We refer to Milnor's survey  \cite{Milnor_Whitehead} or Cohen's book \cite{Cohen} 
for an introduction to this important subject.
\end{remark}

By its definition, the torsion of a chain complex $C$ can be seen as a multiplicative analogue of its Euler characteristic, namely
$$
\chi(C) := \sum_{i=0}^m (-1)^i \cdot\dim(C_i) \ \in \Z.
$$
Keeping in mind this analogy, let us state some of the most important properties of the torsion.
We refer to Milnor's survey \cite{Milnor_Whitehead} or to Turaev's book \cite{Turaev_book} for proofs.

Firstly, the Euler characteristic is additive in the sense that 
$\chi(C_1 \oplus C_2) = \chi(C_1) + \chi(C_2).$
Similarly, the torsion is multiplicative.

\begin{proposition}[Multiplicativity]
\label{prop:multiplicativity}
Let $C_1,C_2$ be some finite-dimensional acyclic based chain complexes over $\F$.
If their direct sum $C_1 \oplus C_2$ is based in the usual way, then we have
$$
\tau(C_1 \oplus C_2) = \pm \tau(C_1) \cdot \tau(C_2).
$$
\end{proposition}

Secondly, the Euler characteristic behaves well with respect to duality in the sense that 
$\chi(C^*) = (-1)^m \cdot \chi(C).$  
Here, $C^*$ is the dual chain complex 
$$
C^* = \left( 0 \to C^*_m \stackrel{\partial_m^*}{\longrightarrow} C^*_{m-1} 
\stackrel{\partial^*_{m-1}}{\longrightarrow} \cdots \stackrel{\partial^*_{2}}{\longrightarrow} 
C^*_1 \stackrel{\partial^*_{1}}{\longrightarrow} C^*_0 \to 0\right)
$$
defined by
$C^*_{i} := \Hom(C_{m-i},\F) $ and $\partial_i^* := (-1)^{i} \cdot \Hom(\partial_{m-i+1} ,\F)$.
The torsion enjoys a similar property.

\begin{proposition}[Duality]
\label{prop:duality}
Let $C$ be a finite-dimensional acyclic based chain complex over $\F$.
If the dual chain complex $C^*$ is equipped with the dual basis, then we have
$$
\tau(C^*) = \pm \tau(C)^{(-1)^{m+1}}.
$$
\end{proposition}

Finally, the Euler characteristic can be computed homologically by the classical formula $\chi(C) = \chi( H_*(C))$. 
If $\F=Q(R)$ is the field of fractions of a domain $R$, and if $C = Q(R) \otimes_R D$
is the localization of a chain complex $D$ over $R$, this formula takes the form
$$
\chi\big(Q(R) \otimes_R D\big) = \sum_{i=0}^m (-1)^i \cdot \rank H_i(D).
$$
There is a multiplicative analogue of this identity for the torsion, 
where ranks of $R$-modules are replaced by their orders.

\begin{theorem}[Homological computation]
\label{th:homological_computation}
Let $R$ be a noetherian unique factorization domain, and let $D$ be a finitely generated free chain complex over $R$.
We assume that $D$  is based and that $\rank H_i(D)=0$ for all $i$. 
Then, we have 
$$
\tau\big( Q(R) \otimes_R D \big) = 
\prod_{i=0}^m \left( \ord H_i(D) \right)^{(-1)^{i+1}} \in (Q(R) \setminus \{0\})/R^\times
$$
where $R^\times$ is the multiplicative group of invertible elements of $R$. 
\end{theorem}

\noindent
(The definition of the order of a finitely generated module
over a unique factorization domain is recalled in \S \ref{subsec:order}.)
This theorem is due to Milnor for a principal ideal domain  $R$ \cite{Milnor_cyclic},
and to Turaev in the general case \cite{Turaev_knot}.\\

The torsion can also be defined for a non-acyclic chain complex $C$.
In this case, $C$ should also be equipped with a basis $h_i$ of $H_i(C)$ for each $i=0,\dots,m$.
Let $Z_i\subset C_i$ be the kernel of $\partial_i$. The short exact sequence
$$
0 \to B_{i} \longrightarrow Z_i \longrightarrow H_i(C) \to 0
$$
shows that a basis of $Z_i$ is obtained by taking, first, a basis $b_i$ of $B_i$ and,
second, some lifts $\widetilde{h}_{i}$ of the vectors $h_i$. 
The resulting basis is denoted by $b_i \widetilde{h}_i$ and, according to the short exact sequence
$$
0 \to Z_{i} \longrightarrow C_i \stackrel{\partial_i}{\longrightarrow} B_{i-1} \to 0,
$$
it can be juxtaposed to some lift $\widetilde{b}_{i-1}$ of $b_{i-1}$ 
to get a new basis $b_i \widetilde{h}_i \widetilde{b}_{i-1}$ of $c_i$. We set
$$
[b_i h_i b_{i-1}/c_i] := \det
\left(\begin{array}{c}
\hbox{matrix expressing}\\
\hbox{$b_i \widetilde{h}_i \widetilde{b}_{i-1}$ in the basis $c_i$} 
\end{array}\right) \ \in \F\setminus \{0\}.
$$

\begin{definition}
The \emph{torsion} of $C$ based by $c=(c_0,\dots,c_m)$ and homologically based by $h=(h_0,\dots,h_m)$ is
$$
\tau(C,c,h) := (-1)^{N(C)} \cdot \prod_{i=0}^m \ [b_i h_i b_{i-1}/c_i]^{(-1)^{i+1}} \ \in \F\setminus\{0\}
$$
where $N(C)$ is the mod $2$ integer
$$
N(C) := \sum_{i=0}^m \left(\sum_{j=0}^i \dim(C_j)\right)  \cdot \left(\sum_{j=0}^i \dim H_j(C)\right).
$$
\end{definition}

\noindent
The sign $(-1)^{N(C)}$ is here for technical convenience and follows the convention of \cite{Turaev_knot}.

\subsection{Abelian Reidemeister torsions of a CW-complex}

\label{subsec:torsion_CW}

Let $X$ be a finite connected CW-complex, whose first homology group is denoted by $H:=H_1(X)$.
Assume that
$$
\varphi:\Z[H] \longrightarrow \F
$$ 
is a ring homomorphism with values in a commutative field $\F$.
We consider the cellular chain complex of $X$ with \emph{$\varphi$-twisted} coefficients
$$
C^\varphi(X) := \F \otimes_{\Z[H]} C(\widehat{X})
$$
where $\widehat{X}$ denotes the maximal abelian covering space of $X$.
This is a finite-dimensional chain complex over $\F$ whose homology
$$
H_*^\varphi(X) := H_*\left(C^\varphi(X)\right)
$$
may be trivial, or may be not. 

Let $E$ be the set of cells of $X$. For each $e \in E$, we choose a lift $\widehat{e}$ to $\widehat{X}$,
and we denote by $\widehat{E}$ the set of the lifted cells.
We also put a total ordering on the finite set $E$,
and we choose an orientation for each cell $e \in E$:
this double choice ($o$rdering+$o$rientation) is denoted by $oo$. 
The choice of $\widehat{E}$ combined to $oo$ induces a basis  $\widehat{E}_{oo}$ of $C(\widehat{X})$,
which defines itself a basis  $1 \otimes \widehat{E}_{oo}$ of  $C^\varphi(X)$.

\begin{definition}
\label{def:Reidemeister}
The \emph{Reidemeister torsion}  with $\varphi$-twisted coefficients of the CW-complex $X$ is 
$$
\tau^\varphi(X) := 
\tau\left(C^\varphi(X), 1 \otimes \widehat{E}_{oo} \right)
\ \in \F/\pm \varphi(H),
$$
with the convention that $\tau^\varphi(X):=0$ if $H_*^\varphi(X) \neq 0$.
\end{definition}

Because of the choices that we were forced to make, 
the quantity $\tau^\varphi(X)$ has two kinds of indeterminacy: 
a sign $\pm 1$ and the image by $\varphi$ of an element of $H$.
Those two ambiguities are resolved by Turaev \cite{Turaev_knot,Turaev_Euler}.
The $\varphi(H)$ indeterminacy is the most interesting and it will be discussed in \S \ref{sec:structures}.
To kill the sign indeterminacy, Turaev defines in \cite{Turaev_knot}
a  \emph{homological orientation} $\omega$ of $X$ to be an orientation of the $\R$-vector space $H_*(X;\R)$.
Then, the quantity
$$
\tau^\varphi(X,\omega) := 
\sgn \tau\big(C(X;\R),  oo, w \big) 
\cdot \tau\big(C^\varphi(X), 1\otimes \widehat{E}_{oo}\big) \ \in \F/\varphi(H),
$$
where $w$ is any basis of $H_*(X;\R)$ representing $\omega$,
does not depend on the choice of $oo$ and  only depends on  $\omega$.
Here again, we set $\tau^\varphi(X,\omega):=0$ if $H_*^\varphi(X) \neq 0$.
This invariant $\tau^\varphi(X,\omega)$ is sometimes called the \emph{sign-refined} Reidemeister torsion of $(X,\omega)$. 
We have $\tau^\varphi(X,-\omega)=-\tau^\varphi(X,\omega)$.

The algebraic properties of the torsion of chain complexes 
(some of those have been recalled in \S \ref{subsec:torsion_chain_complex}) have topological implications.
For instance, Proposition \ref{prop:multiplicativity} translates into a kind of Mayer--Vietoris theorem
for the Reidemeister torsion of CW-complexes. Also, Theorem \ref{th:homological_computation}
allows one to compute the Reidemeister torsion of a CW-complex by homological means 
when $\F$ is the field of fractions of a domain. 

\begin{remark}
Here, we have restricted ourselves to the \emph{abelian} version of the Reidemeister torsion,
in the sense that coefficients are taken in a commutative field $\F$.
Nonetheless, the same construction applies to any ring homomorphism 
$$
\varphi:\Z[\pi_1(X)]\longrightarrow \Lambda
$$
with values in a ring $\Lambda$ for which the rank of free modules is well-defined.
In this situation, we need the torsion of $\Lambda$-complexes evoked in Remark \ref{rem:non-commutative}
and we work with the universal cover  of $X$ instead of its maximal abelian cover. The Reidemeister torsion is then an element
$$
\tau^\varphi(X) \in K_1(\Lambda)/\pm \varphi(\pi_1(X)) \cup \{0\}.
$$ 
\end{remark}

\subsection{Turaev's maximal abelian torsion of a CW-complex}

\label{subsec:maximal}

Let $X$ be a finite connected  CW-complex,
and let $Q(\Z[H_1(X)])$ be  the ring of fractions of $\Z[H_1(X)]$. 
The canonical injection  $\Z[H_1(X)] \to Q(\Z[H_1(X)])$ can \emph{not} play
the role of ``universal coefficients'' for  abelian Reidemeister torsions of $X$,
since the ring $Q(\Z[H_1(X)])$ is not a field (unless $H_1(X)$ is torsion-free).
Nonetheless, we have the following statement.

\begin{proposition}
\label{prop:splitting}
Let $H$ be a finitely generated abelian group. 
Then, the ring $Q(\Z[H])$ splits in a unique way as a direct sum of finitely many commutative fields.
\end{proposition}

\noindent
We apply this to $H:= H_1(X)$ to write  $Q(\Z[H_1(X)])$ as a  direct sum of commutative fields:
$$
Q(\Z[H]) = \F_1 \oplus \cdots \oplus \F_n.
$$
We denote by $\varphi_i: Q(\Z[H]) \to \F_i$ the corresponding projections.

\begin{definition}[Turaev \cite{Turaev_Alexander_torsion}]
The \emph{maximal abelian torsion} of $X$ is 
$$
\tau(X) :=\tau^{\varphi_1}(X) + \cdots + \tau^{\varphi_n}(X) \ \in Q(\Z[H])/\pm H.
$$
\end{definition}

\noindent
Recall that the construction of $\tau^{\varphi_i}(X)$ needs some choices (Definition \ref{def:Reidemeister}).
Here, we do the same choices for all $i=1,\dots,n$ so that the indeterminacy is ``global'' (in $\pm H$) instead of ``local'' 
(one for each component $i=1,\dots,n$).
A sign-refined version $\tau(X,\omega) \in Q(\Z[H])/H$ 
of  $\tau(X)$ is also defined in the obvious way,
for any homological orientation $\omega$ of $X$.

The invariant $\tau(X)$ determines the abelian Reidemeister torsion $\tau^\varphi(X)$ 
for any  ring homomorphism $\varphi:\Z[H] \to \F$, which justifies the name given to $\tau(X)$.
Here is the precise statement, where it is assumed that $\tau^\varphi(X) \neq 0$.
We consider the subring 
$$
Q_\varphi(\Z[H]) := \{x\in Q(\Z[H]): \exists y \in \Z[H], \varphi(y)\neq 0, xy \in \Z[H]  \}
$$
of $Q(\Z[H])$. It is easily checked that $\varphi$ extends in a unique way to a ring homomorphism 
$\widetilde{\varphi}: Q_\varphi(\Z[H]) \to \F$. Then, it can be proved \cite[\S 13]{Turaev_book} that
$$
\tau(X) \in Q_\varphi(\Z[H])/\pm H
\quad \hbox{and} \quad
\tau^\varphi(X) = \widetilde{\varphi}\left(\tau(X)\right).
$$

\begin{proof}[Proof of Proposition \ref{prop:splitting}]
We follow Turaev \cite[\S 12]{Turaev_book}.
The unicity of the splitting means that, for two decompositions of $Q(\Z[H])$
as a direct sum of fields
$$
\F_1 \oplus \cdots \oplus \F_n = Q(\Z[H]) = \F'_1 \oplus \cdots \oplus \F'_{n'}
$$
we must have $n=n'$ and $\F_j = \F'_{\alpha(j)}$ for some permutation $\alpha \in S_n$.
This is an instance of the following general fact, which is easily  proved: 
If a ring splits as a direct sum of finitely many domains, then the splitting is unique in the above sense.

To prove the existence of the splitting, let us assume first that $H$ is finite. 
Each character $\sigma: H \to \C^*$ of $H$ extends to a ring homomorphism $\sigma:\Q[H] \to \C$ by linearity.
The subgroup $\sigma(H)$ of $\C^*$ is finite and, so, is cyclic:
thus, $\sigma\left(\Q[H]\right)$ is the cyclotomic field $\Q\left(e^{2i\pi/m_\sigma}\right)$ 
where $m_\sigma$ is the order of $\sigma(H)$.
Two characters $\sigma$ and $\sigma'$ of $H$ are declared
to be \emph{equivalent} if $m_\sigma=m_{\sigma'}$ and if 
$\sigma,\sigma': \Q[H] \to \Q\left(e^{2i\pi/m_\sigma}\right)$ 
differ by a Galois automorphism of $\Q\left(e^{2i\pi/m_\sigma}\right)$ over $\Q$. 
Let $\sigma_1,\dots, \sigma_n$ be some representatives for the equivalence classes of $\Hom(H,\C^*)$.
Then, the ring homomorphism
\begin{equation}
\label{eq:decomposition}
(\sigma_1,\dots,\sigma_n): \Q[H] \longrightarrow 
\bigoplus_{j=1}^n \Q(e^{2i\pi/m_{\sigma_j}})
\end{equation}
is injective since any non-trivial element of the ring $\Q[H]$ can be detected by a character.
(Indeed, by Maschke's theorem, the $\C$-algebra $\C[H]$ is semi-simple so that,
according to the Artin--Wedderburn theorem, it is isomorphic to $\hbox{End}(V_1) \oplus \cdots \oplus \hbox{End}(V_r)$
for some $\C$-vector spaces $V_1,\dots, V_r$. But, since the $\C$-algebra $\C[H]$ is commutative,
each $V_i$ should be one-dimensional. 
Thus, the $\C$-algebra $\C[H]$ is isomorphic to a direct sum of $r$ copies of $\C$ 
and each corresponding projection $\C[H] \to \C$ restricts to a character $H \to \C^*$.)
Therefore, the ring homomorphism (\ref{eq:decomposition}) is bijective
since its source and target have the same dimension over $\Q$:
\begin{eqnarray*}
\sum_{j=1}^n \dim \Q(e^{2i\pi/m_{\sigma_j}}) &=&
\sum_{j=1}^n \left|\operatorname{Gal}\left(\left.\Q(e^{2i\pi/m_{\sigma_j}})\right/ \Q \right)\right|\\
&=& |\Hom(H,\C^*)| = |H| = \dim \Q[H].
\end{eqnarray*}
(Here, we have used the fact that the extension $\Q(e^{2i\pi/m})$ of $\Q$ is galoisian.) So, when $H$ is finite, we get
\begin{eqnarray*}
Q\left(\Z[H]\right) = 
Q\left(\Q[H]\right) &\simeq &Q\left(\bigoplus_{j=1}^n \Q(e^{2i\pi/m_{\sigma_j }})\right)\\
&=& \bigoplus_{j=1}^n Q\left(\Q(e^{2i\pi/m_{\sigma_j}})\right)
= \bigoplus_{j=1}^n \Q(e^{2i\pi/m_{\sigma_j}})
\end{eqnarray*}
which also shows that $Q\left(\Z[H]\right) = \Q[H]$ in this case.

In general, we set $G:=H/\Tors H$ and we fix a splitting $H \simeq \Tors H \oplus G$.
Then, we have
$$
\Q[H] \simeq \left(\Q[\Tors H]\right) [G] \simeq
\bigoplus_{j=1}^n \Q(e^{2i\pi/m_{\sigma_j}})[G]
$$
and $Q\left(\Z[H]\right) = Q\left(\Q[H]\right)$ can be written
as a direct sum of finitely many commutative fields:
$$
Q\left(\Z[H]\right)  \simeq 
Q\left(\bigoplus_{j=1}^n \Q(e^{2i\pi/m_{\sigma_j}})[G]\right)=
\bigoplus_{j=1}^n Q\left(\Q(e^{2i\pi/m_{\sigma_j}})[G]\right).\\[-0.5cm]
$$
\end{proof}

\section{The Alexander polynomial}

\label{sec:Alexander}

We introduce the Alexander polynomial of a topological space $X$ of finite type\footnote{
A topological space is of \emph{finite type} if it has the homotopy type of a finite CW-complex.}
and, following Milnor and Turaev, 
we explain how to obtain it as a kind of Reidemeister torsion. 
In a few words, ``the Alexander polynomial of $X$ is the order of the first homology group
of the free maximal abelian cover of $X$.'' Thus, we start by recalling what the order of a module is.

\subsection{The order of a module}

\label{subsec:order}

Let $R$ be a unique factorization domain, whose multiplicative group of invertible elements is denoted by $R^\times$.
Let also $M$ be a finitely generated $R$-module.

We choose a presentation of $M$ with, say, $n$ generators and $m$ relations,
and we denote by $A$ the corresponding $m\times n$ matrix:
$$
R^m \stackrel{\cdot A}{\longrightarrow} R^n \longrightarrow M \longrightarrow 0.
$$
Here, $m$ may be infinite. Besides, we can assume that $m\geq n$ with no loss of generality.

\begin{definition}
For any integer $k\geq 0$, the \emph{$k$-th elementary ideal} of $M$ is  
$$
E_k(M) := \big\langle  (n-k)\hbox{-sized minors of } A \big\rangle_{\hbox{\footnotesize ideal}} \ \subset R
$$
with the convention that $E_k(M) :=  R  \ \hbox{ if } k\geq n.$
The \emph{$k$-th order} of $M$ is
$$
\Delta_k(M) := \gcd E_k(M) \ \in R/R^\times.
$$
The \emph{order} of $M$ is $\Delta_0(M)$ and is denoted by $\ord(M)$.
\end{definition}

\noindent
It is easily checked that the elementary ideals and, a fortiori, their greatest common divisors,
do not depend on the choice of the presentation matrix $A$. 
We have the following inclusions of ideals:
$$
E_0(M) \subset E_1(M) \subset \cdots \subset E_{n-1}(M) \subset E_n(M)= E_{n+1}(M) = \cdots = R,
$$
hence the following divisibility relations:
$$
1=\! \cdots\! = \Delta_{n+1}(M) =\Delta_n(M)\ 
|\ \Delta_{n-1}(M)\ | \cdots |\ \Delta_1(M)\ |\ \Delta_0(M) = \ord (M).
$$

\begin{example}
\label{ex:PID}
Let $R$ be a principal ideal domain. Then, $M$ can be decomposed as a direct sum of cyclic modules:
$M = R/n_1R\oplus \cdots \oplus R/n_kR.$
The order of $M$ is represented by the product $n_1\cdots n_k \in R$. 
\end{example}

\begin{exercice} Show that $E_0(M)$ is contained in the \emph{annihilator} of $M$
$$
\hbox{Ann}(M) := \left\{ r\in R : \forall m\in M, r \cdot m=0\right\},
$$
and prove that the converse is not true in general. 
\end{exercice}

\begin{exercice}
\label{ex:multiplicativity}
Show that $\ord(M_1 \oplus M_2) = \ord(M_1) \cdot \ord(M_2)$ 
for any two finitely generated $R$-modules $M_1$ and $M_2$.
\end{exercice}

\subsection{The Alexander polynomial of a  topological space}

Let $X$ be a connected topological space of finite type.
Then, the free abelian group 
$$
G := H_1(X)/ \Tors H_1(X)
$$
is finitely generated. We observe that 
$$
\Z[G] \simeq \Z[t_1^{\pm},\dots,t_b^{\pm}]
$$ 
where $b=\beta_1(X)$ is the first Betti number of $X$, so that the ring $\Z[G]$ has essentially 
the same properties as a polynomial ring with integer coefficients.
In particular, $\Z[G]$ is a unique factorization domain (by Gauss's theorem, since $\Z$ is so)
and it is noetherian (by Hilbert's basis theorem, since $\Z$ is so).
Moreover, we have $\Z[G]^\times = \pm G$.

We are interested in the \emph{maximal free abelian} cover  $\overline{X} \to X$
whose group of covering automorphisms is identified with $G$.
More precisely, we are interested in the homology of $\overline{X}$ as a $\Z[G]$-module.
By our assumptions, the $\Z[G]$-module $H_i(\overline{X})$ is finitely generated for any $i\geq 0$ 
and, sometimes, it is called the \emph{$i$-th Alexander module} of $X$. 
Here, we are mainly interested in the first Alexander module 
which \emph{only} depends  on the fundamental group of $X$.
Indeed, the Hurewicz theorem gives a canonical isomorphism
$$
H_1(\overline{X}) \simeq   K(X)/[K(X),K(X)]
$$
where  $K(X)$ denotes the kernel of the canonical epimorphism $\pi_1(X) \to G$.

\begin{definition}
The \emph{Alexander polynomial} of $X$ is 
$$
\Delta(X) := \ord \left( K(X)/[K(X),K(X)]\right) \ \in \Z[G]/\pm G
$$
where $G \simeq \pi_1(X)/K(X)$ acts on $K(X)/[K(X),K(X)]$ by conjugation.
\end{definition}

Here is a recipe to compute the Alexander polynomial by means of Fox's ``free derivatives''.
(The basics of Fox's free differential calculus are recalled in the appendix.)

\begin{theorem}[Fox \cite{Fox_Alexander}]
\label{th:Alexander_computation}
Let $X$ be a connected topological space of finite type.
Consider a finite presentation of $\pi_1(X)$
\begin{equation}
\label{eq:group_presentation}
\pi_1(X) = \langle x_1,\dots,x_n| r_1 , \dots, r_m\rangle
\end{equation}
and the associated matrix
$$
A := \left(\begin{array}{ccc}
\frac{\partial r_1}{\partial x_1} & \cdots & \frac{\partial r_1}{\partial x_n} \\
\vdots & \ddots & \vdots \\
\frac{\partial r_m}{\partial x_1} & \cdots & \frac{\partial r_m}{\partial x_n}
\end{array}\right)
$$
where $\frac{\partial\ }{\partial x_1},\dots, \frac{\partial \ }{\partial x_n}$ 
denote the free derivatives with respect to $(x_1,\dots,x_n)$. 
Then, the Alexander polynomial of $X$ is 
$$
\Delta(X) = \gcd \big\{ \hbox{$(n-1)$-sized minors of the reduction of $A$ to $\Z[G]$}\big\}.
$$
\end{theorem}

\begin{proof}
We follow Turaev \cite[\S 16]{Turaev_book}.
Let $Y$ be the $2$-dimensional cellular \emph{realization} of the group presentation (\ref{eq:group_presentation}).
More explicitely, $Y$ has a unique $0$-cell, $n$ $1$-cells (in bijection with the generators $x_1\dots,x_n$) 
and $m$ $2$-cells (in bijection with the relations $r_1,\dots,r_m$ which are interpreted as attaching maps
for the $2$-cells).  Then, $\pi_1(Y)$ has the same presentation (\ref{eq:group_presentation}) as the group $\pi_1(X)$. 
Since $\Delta(X)$ only depends on $\pi_1(X)$, we have
\begin{equation}
\label{eq:Delta_Delta}
\Delta(X) = \Delta(Y) = \ord H_1(\overline{Y})
\end{equation}
where $\overline{Y}$ denotes the maximal free abelian cover of $Y$.
It follows from the topological interpretation of Fox's free derivatives (see \S \ref{subapp:topological_interpretation}) 
that $A$ reduced to $\Z[\pi_1(Y)]$ is the matrix of the boundary operator of the universal cover $\widetilde{Y}$ of $Y$
$$
\partial_2 : C_2(\widetilde{Y}) \longrightarrow C_1(\widetilde{Y}) 
$$
with respect to some appropriate basis (which are obtained by lifting the cells of $Y$).
Let $\overline{Y}^0$ be the $0$-skeleton of $\overline{Y}$. Because we have
$$
\Coker \left(\partial_2:C_2(\overline{Y}) \longrightarrow C_1(\overline{Y})\right)
= H_1(\overline{Y},\overline{Y}^0),
$$
that topological interpretation of the matrix $A$ leads to 
\begin{equation}
\label{eq:Delta_1}
\gcd \big\{ \hbox{$(n-1)$-sized minors of $A$ reduced to $\Z[G]$} \big\} 
= \Delta_1 \left(H_1(\overline{Y},\overline{Y}^0)\right).
\end{equation}
The exact sequence of $\Z[G]$-modules
$$
0 \to H_1(\overline{Y}) \longrightarrow H_1(\overline{Y},\overline{Y}^0)
\longrightarrow H_0(\overline{Y}^0) \longrightarrow H_0(\overline{Y}) \to 0
$$
shows that 
$$
\Tors H_1(\overline{Y}) = \Tors H_1(\overline{Y},\overline{Y}^0)
\quad \hbox{and}\quad   \rank H_1(\overline{Y}) = \rank H_1(\overline{Y},\overline{Y}^0) -1.
$$ 
The following statement is proved by Blanchfield in \cite{Blanchfield_intersection}.
See also \cite[\S 3.1]{Hillman}.

\begin{quote}\textbf{Fact.} 
{\it Let $M$ be a finitely generated module over a noetherian unique factorization domain.
Then, we have
$$
\Delta_i(M) = \left\{\begin{array}{ll}
0 & \hbox{ if } i < \rank(M)\\
\Delta_{i-\rank{M}}(\Tors M) & \hbox{ if } i \geq  \rank(M).
\end{array}\right.
$$}
\end{quote}

\noindent
We deduce that $\Delta_0(H_1(\overline{Y})) = \Delta_1( H_1(\overline{Y},\overline{Y}^0))$.
The conclusion then follows from equations (\ref{eq:Delta_Delta})  and (\ref{eq:Delta_1}). 
\end{proof}

\subsection{Alexander polynomial and Milnor torsion}

Let $X$ be a finite connected CW-complex,
with maximal free abelian cover $\overline{X}$. As before, we set
$$ 
G:= H_1(X)/ \Tors H_1(X).
$$
We consider the fraction
$$
A(X) := \prod_{i\geq 0} \left( \ord  H_i(\overline{X})\right)^{(-1)^{i+1}} 
\ \in Q(\Z[G])/\pm G
$$
with the convention that $A(X):=0$ if $\ord H_i(\overline{X}) =0$ for some $i\geq 0$.
Observe that the Alexander polynomial $\Delta(X)$ appears as a numerator of $A(X)$,
which is sometimes called the \emph{Alexander function} of $X$.

\begin{definition}
The \emph{Milnor torsion} of $X$ is the Reidemeister torsion
$$
\tau^\mu(X) \in Q(\Z[G])/\pm G
$$
where the coefficients $\mu: \Z[H_1(X)] \to Q(\Z[G])$
are induced by the canonical map $H_1(X) \to G$.
\end{definition}

The next result shows that the combinatorial invariant $\tau^\mu(X)$
is in fact a topological invariant, and is more precisely  a homotopy invariant.

\begin{theorem}[Milnor \cite{Milnor_duality,Milnor_cyclic}, Turaev \cite{Turaev_knot}]
\label{th:Milnor-Turaev}
For any finite connected CW-complex $X$, we have 
$$
\tau^\mu(X) = A(X) \  \in Q(\Z[G])/\pm G.
$$
\end{theorem}

\begin{proof}
Assume that $\tau^\mu(X)=0$. By our convention, this means that 
$$
0 \neq H^\mu_i(X) = H_i\left( Q(\Z[G])\otimes_{\Z[G]} C(\overline{X})\right)
= Q(\Z[G])\otimes_{\Z[G]} H_i\left( \overline{X}\right)
$$
for some $i$.
Here, the last identity follows from the universal coefficients theorem 
(which can be applied since $C(\overline{X})$ is $\Z[G]$-free)
and the fact that the field of fractions of $\Z[G]$ is $\Z[G]$-flat. 
So, the $\Z[G]$-module $H_i\left( \overline{X}\right)$ is not fully torsion
or, equivalently, it has order $0$. It follows that $A(X)=0$ by convention.
If $\tau^\mu(X)$ is not zero, then the identity $\tau^\mu(X) = A(X)$ 
is an application of Theorem \ref{th:homological_computation}.
\end{proof}

\section{The abelian Reidemeister torsion for three-dimensional manifolds}

\label{sec:torsion_dim_3}

We introduced abelian Reidemeister torsions of CW-complexes in \S \ref{sec:torsion_CW},
and we saw  in \S \ref{sec:Alexander} that the Alexander polynomial  fits into this framework.
In this section, we apply the theory of abelian Reidemeister torsions to $3$-manifolds and we compute them
for two important classes of $3$-manifolds: lens spaces (which are classified in this way) and surface bundles.

\subsection{Abelian Reidemeister torsions of a $3$-manifold}

\label{subsec:torsion_dim_3}

A theorem of Chapman asserts that the Reidemeister torsion of CW-complexes is invariant under homeomorphisms,
so that it defines a topological invariant of those topological spaces which admit cell decompositions \cite{Chapman}. 
We will not need this deep result here. 
It is not too difficult (although technical) to prove that the Reidemeister torsion is invariant under cellular subdivisions,
so that the Reidemeister torsion defines a piecewise-linear invariant of polyhedra \cite{Milnor_Whitehead}.
It follows from Theorem \ref{th:triangulation} and Theorem \ref{th:hauptvermutung} 
that the Reidemeister torsion induces a \emph{topological} invariant of $3$-manifolds.

In more details, let $M$ be a $3$-manifold and let $\varphi:\Z[H_1(M)] \to \F$
be a ring homomorphism with values in a commutative field $\F$.
The given orientation of $M$ induces an orientation $\omega_M$ of $H_*(M;\R)$,
namely the orientation defined by the basis
$$
\left([\star],b,b^\sharp,[M]\right).
$$
Here $[\star]\in H_0(M;\R)$ is the homology class of a point,
$b$ is a basis of $H_1(M;\R)$, $b^\sharp$ is the dual basis of $H_2(M;\R)$
with respect to the intersection pairing and $[M] \in H_3(M;\R)$ is the fundamental class.
Observe that  $\omega_{-M}= (-1)^{\beta_1(M)+1} \cdot \omega_M$.

\begin{definition}
The \emph{Reidemeister torsion}  with $\varphi$-twisted coefficients of the $3$-manifold $M$ is
$$
\tau^\varphi(M) := \tau^{\varphi \circ \rho_*}\left(K,\rho_*^{-1}(\omega_M) \right) \ \in \F/\varphi\left(H_1(M)\right)
$$
where $(K,\rho)$ is a triangulation of $M$ and the cellular homology of $K$ is identified with its singular homology.
\end{definition}

\noindent
Theorem \ref{th:triangulation} ensures the existence of the triangulation $(K,\rho)$ of $M$.
Theorem \ref{th:hauptvermutung} and the invariance of the Reidemeister torsion under subdivisions
imply that the above definition does not depend on the choice of $(K,\rho)$.
Thus, for any cell decomposition $X$ of $M$ (which can be subdivided to a triangulation), we also have
$$
\tau^\varphi(M) = \tau^\varphi(X,\omega_M) \ \in \F/\varphi\left(H_1(M)\right)
$$
where the cellular homology of $X$ and the singular homology of $M$ are identified.

We mentioned in \S \ref{subsec:torsion_CW} two properties of the Reidemeister torsion of CW-complexes,
which were inherited from algebraic properties of the torsion of chain complexes.
Here is a third important property,  which is specific to manifolds. 

\begin{theorem}[Franz \cite{Franz_duality}, Milnor \cite{Milnor_duality}]
\label{th:duality_manifolds}
Let $M$ be a $3$-manifold and let $\F$ be a commutative field with an involution $f\mapsto \overline{f}$.
Consider a ring homomorphism  $\varphi:\Z[H_1(M)] \to \F$ such that
$$
\forall h \in H_1(M), \quad \varphi(h^{-1}) = \overline{\varphi(h)} \ \in \F.
$$
Then,  we have the following symmetry property:
$$
\tau^\varphi(M) = \overline{\tau^\varphi(M)} \ \in \F/\varphi\left(H_1(M)\right).
$$
\end{theorem}

\begin{proof}[Sketch of the proof]
Let $K$ be a triangulation of $M$: $|K|=M$. It can be lifted to a triangulation $\widehat{K}$ 
of the maximal abelian cover $\widehat{M}$ of $M$. We fix a total ordering  on the set $E$ of simplices of $K$, 
we choose an orientation for each simplex $e \in E$ and we choose a lift $\widehat{e}$ of $e$ to $\widehat{K}$.
Thus, we obtain a basis $\widehat{E}_{oo}$ of $C(\widehat{K})$ and, by definition, we have
$$
\pm \tau^\varphi(M) = \tau( C^\varphi(K), 1 \otimes \widehat{E}_{oo}) \ \in \F/\pm \varphi\left(H_1(M)\right)
$$
with the convention that $\tau^\varphi(M) := 0$ if $H_*^\varphi(M) \neq 0$.

Besides,  we can consider the  cell decomposition $K^*$ dual to the triangulation $K$,
which can be lifted to a cell decomposition $\widehat{K^*}$ of $\widehat{M}$. Each cell $e^*$ of $K^*$ is dual
to a unique simplex $e$ of $K$, so that it has a preferred orientation and a preferred lift $\widehat{e^*}$
which are determined by the choices we did for $e$. Moreover, the total ordering on $E$ induces a total ordering
on the set $E^*$ of the cells of $K^*$. Thus, we obtain a basis $\widehat{E^*}_{oo}$ of $C(\widehat{K^*})$ and, by definition, we have
$$
\pm \tau^{\overline{\varphi}}(M) = \tau( C^{\overline{\varphi}}(K^*), 1 \otimes \widehat{E^*}_{oo}) \ \in \F/\pm \varphi\left(H_1(M)\right)
$$
with the convention that $\tau^{\overline{\varphi}}(M) := 0$ if $H_*^{\overline{\varphi}}(M) \neq 0$.
Here we have denoted by $\overline{\varphi}$ the ring homomorphism $\varphi$ composed with the involution of $\F$.

As in the proof of the Poincar\'e duality theorem (relating the $\varphi$-twisted cohomology of $K$ 
to the $\overline{\varphi}$-twisted homology of $K^*$), one can prove that the dual of the based chain complex 
$( C^\varphi(K), 1 \otimes \widehat{E}_{oo})$  is isomorphic to the based chain complex
$( C^{\overline{\varphi}}(K^*), 1 \otimes \widehat{E^*}_{oo})$. 
We conclude with the duality property for torsions of chain complexes (Proposition \ref{prop:duality}) 
that $\tau^{\overline{\varphi}}(M) = \pm \tau^\varphi(M)$.
We refer to \cite{Milnor_duality} or to \cite{Turaev_book} for the details of proof.
See \cite{Turaev_knot} for the computation of signs.
\end{proof}

Let us now compute the Reidemeister torsion of a $3$-manifold $M$  presented 
by a Heegaard splitting $M= A \cup B$, where $A$ and $B$ are genus $g$ handlebodies. 
We assume that the ring homomorphism $\varphi:\Z[H_1(M)] \to \F$ giving the coefficients
is non-trivial, \ie $\varphi(H_1(M))\neq 1$. 
(Otherwise, we would have $H^\varphi_0(M)\neq 0$ so that $\tau^\varphi(M) =0$ by our convention.)

We consider the $2g$ oriented simple closed curves $\alpha_1, \dots, \alpha_g, \alpha^\sharp_1,\dots, \alpha^\sharp_g$
on the genus $g$ surface $\partial A$ shown on Figure \ref{fig:4g_curves}, 
as well as the curves $\beta_1, \dots, \beta_g, \beta^\sharp_1,\dots, \beta^\sharp_g$ on $\partial B$.
If we focus on $A$, the $3$-manifold $M$ is determined by 
how the curves $\beta_1,\dots,\beta_g$ read in the surface $\partial A$ when $\partial B$ is identified to $\partial A$.
So, we are aiming at a formula expressing $\tau^\varphi(M)$ 
in terms of  the curves $\beta_1,\dots,\beta_g \subset \partial A=\partial B$.

\begin{figure}[h]
\labellist \small \hair 2pt
\pinlabel {$A$}  at 417 150
\pinlabel {$B$} at 417 390
\pinlabel {$\cdots$} at 218 196
\pinlabel {$\cdots$} at 218 440
\pinlabel {$\alpha_1^\sharp$} [t] at 107 70
\pinlabel {$\alpha_g^\sharp$} [t] at 322 67
\pinlabel {$\beta_1^\sharp$} [t] at 107 322
\pinlabel {$\beta_g^\sharp$} [t] at 322 322
\pinlabel {$\alpha_1$} [t] at 98 194
\pinlabel {$\alpha_g$} [t] at 337 191
\pinlabel {$\beta_1$} [t] at 98 443
\pinlabel {$\beta_g$} [t] at 336 442
\endlabellist
\centering
\includegraphics[scale=0.4]{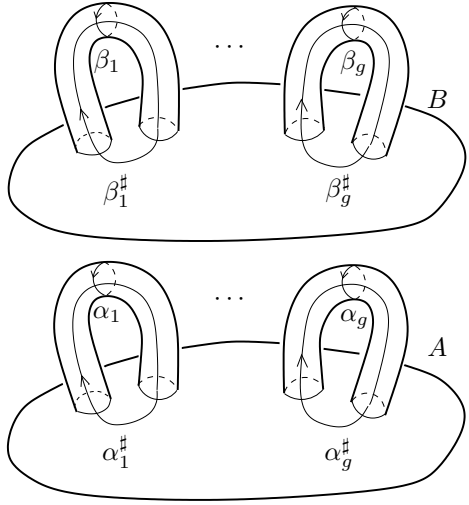}
\label{fig:4g_curves}
\caption{The curves $\alpha,\alpha^\sharp$ on $\partial A$
and the curves $\beta,\beta^\sharp$ on $\partial B$.}
\end{figure}

We choose a small disk $D \subset \partial A$ and a base point $\star \in \partial D$.
We base the oriented simple closed curves $\alpha,\alpha^\sharp$ at $\star$ 
to get a basis of the free group $\pi_1(\partial A\setminus D,\star)$.
We also denote by 
$$
\frac{\partial \ }{\partial \alpha_1}, \dots, \frac{\partial \ }{\partial \alpha_g}, \
\frac{\partial \ }{\partial \alpha_1^\sharp}, \dots, \frac{\partial \ }{\partial \alpha_g^\sharp}
$$
the Fox's free derivatives with respect to this basis. (See the appendix.) 
We also base the simple oriented closed curves $\beta$ at $\star$, 
so that they define elements of $\pi_1(\partial A \setminus D, \star)$.
Then, we can consider the following $g\times g$ matrix with coefficients in $\Z[H_1(M)]$:
$$
\mathcal{A} := \incl_* \left(\begin{array}{ccc}
\frac{\partial \beta_1}{\partial \alpha_1^\sharp} & \cdots & \frac{\partial \beta_1}{\partial \alpha_g^\sharp}\\
\vdots & \ddots & \vdots \\ 
\frac{\partial \beta_g}{\partial \alpha_1^\sharp} & \cdots & \frac{\partial \beta_g}{\partial \alpha_g^\sharp}
\end{array}\right)
$$
where $\incl_*: \pi_1(\partial A\setminus D,\star) \to H_1(M)$ is induced by the inclusion.

\begin{lemma}
\label{lem:torsion_Heegaard_splitting}
With the above notation and for any indices $i,j\in \{1,\dots,g\}$, we have
\begin{equation}
\label{eq:torsion_Heegaard_splitting}
\tau^\varphi(M) \cdot (\varphi(\alpha_i^\sharp)-1) \cdot (\varphi(\beta_j^\sharp)-1)
= (-1)^{i+j+g+1} \cdot \tau_0 \cdot \varphi(\mathcal{A}_{ij}) \ \in \F/\varphi(H_1(M))
\end{equation}
where $\tau_0$ is a certain explicit sign, 
and where $\mathcal{A}_{ij}$ is the $(i,j)$-th minor of the matrix $\mathcal{A}$.
\end{lemma}

\begin{proof}[Sketch of the proof]
The lemma is an application of a formula by Turaev which computes $\tau^\varphi(M)$
from a cell decomposition of $M$ with a single $0$-cell and a single $3$-cell \cite[\S II.1]{Turaev_book_dim_3}.
In more details, we consider the cell decomposition $X$  defined as follows by the Heegaard splitting.
There is only one $0$-cell $e^0$, the center of the ball to which
handles have been added to form $A$; the $1$-cells are $e^1_1,\dots,e^1_g$  
where $e^1_i$ is obtained from the core of the $i$-th handle of $A$,
which is bounded by two points, by adding the trace of those two points 
when the previous ball is ``squeezed'' to $e^0$; 
the $2$-cells are $e^2_1,\dots,e^2_g$ where $e^2_j$ is obtained from the co-core
of the $j$-th handle of $B$ (with boundary $\beta_j$) 
by adding the trace of $\beta_j \subset \partial A$ when $A$ is ``squeezed'' to
$e^0 \cup e^1_1\cup \cdots \cup e^1_g$; there is only one $3$-cell $e^3$, 
namely the complement in $M$ of the cells of smaller dimension. 
See Figure \ref{fig:cell_decomposition}.
The cell decomposition $X$ is related to the curves $(\alpha,\alpha^\sharp)$
and $(\beta, \beta^\sharp)$ in the following way:
each curve $\alpha_i^\sharp$ is isotopic to $e^1_i$ in $A$, 
and each curve $\beta_j$ is the boundary of $e^2_j \cap B$.

\begin{figure}[h]
\vspace{0.5cm}
\labellist \small \hair 2pt
\pinlabel {$A$}  at 417 150
\pinlabel {$B$} at 417 390
\pinlabel {$\cdots$} at 218 196
\pinlabel {$\cdots$} at 218 440
\pinlabel {$e^0$} [t] at 213 51
\pinlabel {$e^1_1$} [tr] at 120 62
\pinlabel {$e^1_g$} [t] at 306 65
\pinlabel {(part of) $e^2_1$} [b] at 101 485
\pinlabel {(part of) $e^2_g$} [b] at 334 483
\pinlabel {(part of) $e^3$} at 215 313
\endlabellist
\centering
\includegraphics[scale=0.4]{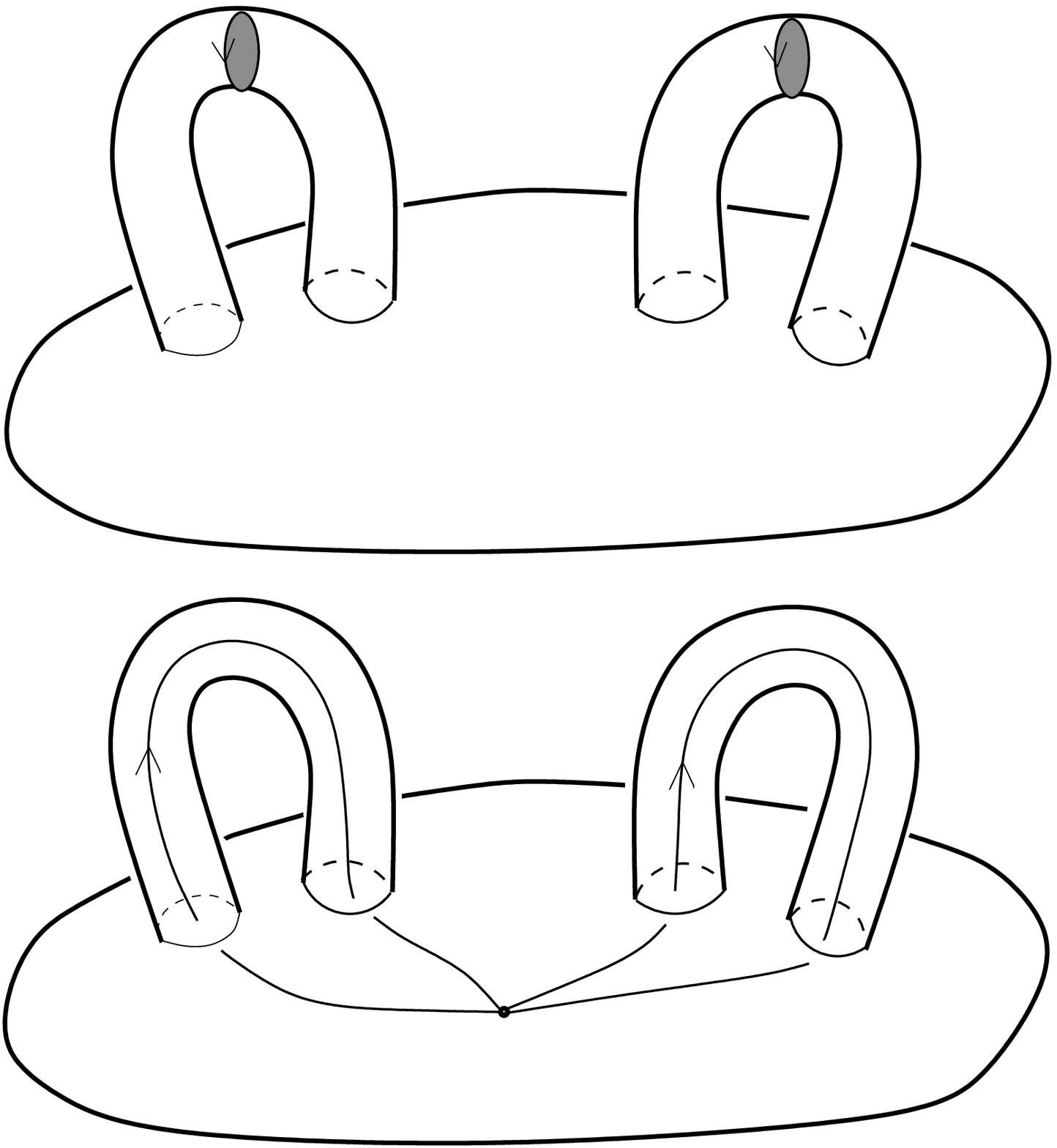}
\caption{The cell decomposition $X$ of $M$ induced by the Heegaard splitting.}
\label{fig:cell_decomposition}
\end{figure}

The cells of $X$ can be oriented as follows: $e^0$ is given the $+$ sign, 
the $1$-cell $e^1_i$ is oriented coherently with $\alpha_i^\sharp$,
the $2$-cell $e^2_j$ is oriented so that the intersection number $e^2_j \centerdot \beta_j^\sharp$ is $+1$
and $e^3$ inherits the orientation from $M$. 
Let $\widehat{X}$ be the maximal abelian cover of $X$, 
and choose some lifts $\widehat{e}^0,\widehat{e}^1_1,\dots,\widehat{e}^1_g,
\widehat{e}^2_1,\dots,\widehat{e}^2_g, \widehat{e}^3$ of the cells  to $\widehat{X}$.
The computation of $\tau^\varphi(M)$ in terms of the cell decomposition $X$ involves the sign
$$
\tau_0 := \sgn \tau\big(C(X;\R),  oo, \hbox{\footnotesize representative of } \omega_M \big)
$$
where $oo$ refers to the above choice of $o$rder and $o$rientations for the cells of $X$.
It also involves the cell chain complex of $\widehat{X}$.
For some appropriate choices of the lifts,
the boundary operators $\partial_3$ and $\partial_1$ are given by 
$$
\partial_3(\widehat{e}^3 )= \sum_{j=1}^g \left(\left[\beta_j^\sharp\right]-1\right)\cdot \widehat{e}^2_j
\quad \hbox{and} \quad 
\partial_1 ( \widehat{e}^1_i ) = \left(\left[\alpha_i^\sharp\right]-1\right) \cdot \widehat{e}^0,
$$
respectively. Thus, the main indeterminate is the boundary operator $\partial_2$ and we denote
$$
\mathcal{A}' := \left(\hbox{matrix of } \partial_2: C_2(\widehat{X}) \longrightarrow C_1(\widehat{X}) 
\hbox{ in the basis $\widehat{e}^2$ and $\widehat{e}^1$} \right).
$$
Then, a computation of $\tau^\varphi(X)$ from its definition gives
$$
\tau^\varphi(M) \cdot (\varphi(\alpha_i^\sharp)-1) \cdot (\varphi(\beta_j^\sharp)-1)
= (-1)^{i+j+g+1} \cdot \tau_0 \cdot \varphi(\mathcal{A}'_{ij}) \ \in \F/\varphi(H_1(M))
$$
See the proof of Theorem II.1.2 in \cite{Turaev_book_dim_3} for the details of  computation.
The topological interpretation of free differential calculus (see \S \ref{subapp:topological_interpretation})
shows that, for appropriate choices of the lifted cells, the matrix $\mathcal{A}'$ is equal to $\mathcal{A}$.
The conclusion follows.
\end{proof}

\subsection{The Alexander polynomial of a $3$-manifold}

Let $M$ be a $3$-manifold. We consider the free abelian group
$$
G := H_1(M)/ \Tors H_1(M)
$$
and the canonical ring homomorphism
$$
\mu: \Z[H_1(M)] \longrightarrow Q(\Z[G])
$$
with values in the field of fractions of the group ring $\Z[G]$.

\begin{theorem}[Milnor \cite{Milnor_duality}, Turaev \cite{Turaev_Alexander}]
\label{th:Milnor-Turaev_dim_3}
For any $3$-manifold $M$ with $\beta_1(M)\geq 2$, we have
$$
\pm \tau^\mu(M) = \Delta(M) \ \in \Z[G]/\pm G
$$
and for any $3$-manifold $M$ with $\beta_1(M)=1$, we have
$$
\pm \tau^\mu(M) = \frac{\Delta(M)}{(t-1)^2} \ \in Q(\Z[G])/\pm G
$$
where $t$ is a generator of $G \simeq \Z$.
\end{theorem}

\begin{proof}
This can be deduced from Theorem \ref{th:Milnor-Turaev}.
Alternatively, we can  use the following argument which is more direct.
The torsion $\tau^\mu(M)$ can be computed from a Heegaard splitting
as explained by Lemma  \ref{lem:torsion_Heegaard_splitting}. Using the same notations, we get
\begin{equation}
\label{eq:tau_Delta_1}
\tau^\mu(M) \cdot (\mu(\alpha_i^\sharp)-1) \cdot (\mu(\beta_j^\sharp)-1)
=  \mu(\mathcal{A}_{ij}) \ \in Q(\Z[G])/\pm G
\end{equation}
where $\mathcal{A}_{ij}$ is the $(i,j)$-th minor of
$$
\mathcal{A} = \left(\hbox{matrix of } \partial_2: C_2(\widehat{X}) \longrightarrow C_1(\widehat{X}) 
\hbox{ in the basis $\widehat{e}^2$ and $\widehat{e}^1$} \right).
$$
Let $Y$ be the $2$-skeleton of $X$ and, as before, 
we denote by $\overline{Y}$ the  maximal free abelian cover of $Y$.
The proof of Theorem \ref{th:Alexander_computation} tells us that
\begin{equation}
\label{eq:tau_Delta_2}
\Delta(M) = \gcd \{\hbox{$(g-1)$-sized minors of $\mathcal{D}$}\}
\end{equation}
where $\mathcal{D}$ is the matrix of $\partial_2 : C_2(\overline{Y}) \to C_1(\overline{Y})$.
Since $Y$ is the $2$-skeleton of $X$, $\mathcal{D}$ is also the reduction of $\mathcal{A}$ to $\Z[G]$, 
\ie $\mathcal{D} = \mu(\mathcal{A})$. We can conclude thanks to (\ref{eq:tau_Delta_1}) and (\ref{eq:tau_Delta_2}). 
\end{proof}

\begin{remark}
Theorem \ref{th:Milnor-Turaev_dim_3} allows one to study the Alexander polynomial of $3$-manifolds
using the technology of Reidemeister torsions \cite{Milnor_duality,Turaev_knot}. 
For example, the symmetry property of the Reidemeister torsion (Theorem \ref{th:duality_manifolds}) 
results in the same property for the Alexander polynomial: 
$$
\Delta(M) = \overline{\Delta(M)} \ \in \Z[G]/\pm G
$$
where the bar denotes the  ring endomorphism of $\Z[G]$ defined by $\overline{g}:=g^{-1}$ for all $g\in G$. 
\end{remark}

\subsection{Turaev's maximal abelian torsion of a $3$-manifold}

\label{subsec:maximal_dim_3}

The maximal abelian torsion of CW-complexes induces a topological invariant of $3$-manifolds.
As explained in \S \ref{subsec:torsion_dim_3}, this follows 
from Theorem \ref{th:triangulation} and Theorem \ref{th:hauptvermutung} using triangulations.
Let $M$ be a $3$-manifold  with first homology group $H := H_1(M)$.

\begin{definition}[Turaev \cite{Turaev_Alexander_torsion}]
\label{def:maximal_dim_3}
The \emph{maximal abelian torsion} of $M$ is
$$
\tau(M) := \tau(X,\omega_M) \ \in Q(\Z[H])/ H
$$
where $X$ is a cell decomposition of $M$ and where the identification
between the cellular homology of $X$ and the singular homology of $M$ is implicit.
\end{definition}

The target of the  maximal abelian torsion depends a lot on the first Betti number
of the $3$-manifold. Three cases have to be distinguished.\\

\noindent
\emph{Assume that $\beta_1(M)=0$.} 
Among the cyclotomic fields into which the field $Q(\Z[H])=\Q[H]$ splits 
(see the proof of Proposition \ref{prop:splitting}), there is $\Q$
which corresponds to the trivial character of $H$.
The corresponding projection is the \emph{augmentation} map
$$
\varepsilon: \Q[H] \longrightarrow \Q, \ \sum_{h\in H} q_h \cdot h \longmapsto \sum_{h\in H} q_h.
$$
Its kernel is called the \emph{augmentation ideal} of $\Q[H]$. 
The Reidemeister torsion with $\varepsilon$-twisted coefficients is zero since $H_*(M;\Q) \neq 0$.
We conclude that
$$
\tau(M) \in \Ker(\varepsilon) \subset \Q[H] = Q(\Z[H])
$$
for any $3$-manifold $M$ with $\beta_1(M)=0$.\\

\noindent
\emph{Assume that $\beta_1(M)\geq 2$.} 
Among the fields into which $Q(\Z[H])$ splits, there is the field of rational fractions $Q(\Z[G])$
where $G:=H/\Tors H$. Let $\mu:Q(\Z[H]) \to Q(\Z[G])$ be the corresponding projection so that,
by definition of $\tau(M)$, we have $\mu\left(\tau(M)\right) = \tau^\mu(M)$.
We know from Theorem \ref{th:Milnor-Turaev_dim_3} that $\tau^\mu(M)$ is the Alexander polynomial $\Delta(M)$,
so that it belongs to $\Z[G]$.
More generally, Turaev proves that 
$$
\tau(M) \in \Z[H]/H \subset Q(\Z[H])/H
$$ 
for any $3$-manifold $M$ with $\beta_1(M)\geq 2$.
See \cite[\S II]{Turaev_book_dim_3} for details.\\

\noindent
\emph{Assume that $\beta_1(M)=1$.} We proceed as in the previous case, and we denote by $t$ a generator of $G = H/\Tors H$.
We know from Theorem \ref{th:Milnor-Turaev_dim_3} that $\mu\left(\tau(M)\right)$ coincides with $\Delta(M)/(t-1)^2$,
so that it belongs to $\Z[t^\pm]/(t-1)^2$. 
More generally, Turaev ``extracts'' from $\tau(M)$ an \emph{integral part} denoted by
$$
[\tau](M) \in \Z[H]/H.
$$
In more details, we deduce from Theorem \ref{th:duality_manifolds} that $\tau(M)$ is symmetric
in the sense that $\overline{\tau(M)} = \tau(M) \in Q(\Z[H])/H$,
where the bar denotes the ring endomorphism of $Q(\Z[H])$ defined by $\overline{h}:= h^{-1}$ for all $h\in H$. 
In fact, we can find a representative of $\tau(M)$ which satisfies this symmetry property in $Q(\Z[H])$.
(This is a refinement of Theorem \ref{th:duality_manifolds} proved in \cite{Turaev_knot}.)
Then, $[\tau](M)$ is defined to be the mod $H$ class of
$$
(\hbox{\footnotesize a symmetric representative of }\tau(M)) -  \frac{\sum_{h\in \Tors H} h}{(t-1)\cdot(t^{-1}-1)}
\ \in \Z[H]
$$
where $t \in H$ projects to a generator of $G \simeq \Z$. We refer to \cite[\S II]{Turaev_book_dim_3} for further details.

\subsection{Example: torsion of lens spaces} 

\label{subsec:torsion_lens_spaces}

Recall that the lens space $L_{p,q}$ is defined by the Heegaard splitting
\begin{equation}
\label{eq:lens}
L_{p,q}= \left(D^2 \times S^1\right) \cup_f \left(-D^2 \times S^1\right)
\end{equation}
where $f: S^1 \times S^1 \to S^1 \times S^1$ is a homeomorphism
such that the matrix of $f_*$ in the basis $\left(a := [S^1 \times 1], a^\sharp:=[1 \times S^1]\right)$ 
of $H_1(S^1\times S^1)$ is
$$
\left(\begin{array}{cc} q & s\\ p & r 
\end{array}\right) 
\quad \hbox{with} \ qr-ps=1.
$$
Let $T$ be the \emph{preferred} generator of $H_1(L_{p,q})$ defined by the core $0\times S^1$ 
of the left-hand solid torus in the decomposition (\ref{eq:lens}).

\begin{lemma}
\label{lem:torsion_lens}
Let $\varphi: \Z[H_1(L_{p,q})]\to \F$ be a non-trivial ring homomorphism with values in a commutative field. 
Then, we have
$$
\tau^\varphi(L_{p,q}) = \varepsilon_p \cdot (\varphi(T)-1)^{-1} \cdot(\varphi(T)^r -1)^{-1}
\ \in \F/\varphi(H_1(L_{p,q}))
$$
where $\varepsilon_p$ is a sign which does not depend on $q\in \Z_p$ and $r:=q^{-1} \in \Z_p$.
\end{lemma}

\begin{proof}
We apply Lemma \ref{lem:torsion_Heegaard_splitting} with $i=j=g=1$.
We have $[\alpha_1^\sharp]=T$ by definition of $T$, and
$[\beta_1^\sharp]= s\cdot[\alpha_1]+r\cdot [\alpha_1^\sharp] = r\cdot T \in H_1(L_{p,q})$.
Just like the cellular chain complex for the cell decomposition induced by the Heegaard splitting,
the sign $\tau_0$ in (\ref{lem:torsion_Heegaard_splitting}) does not depend on $q$.
\end{proof}

The topological classification of lens spaces can now be completed
and, for this, we need a number-theoretic result.
This is for instance proved in \cite[\S I]{deRham} using elementary properties of Gauss sums
and a non-vanishing property of Dirichlet series.

\begin{lemma}[Franz \cite{Franz}]
\label{lem:Franz}
Let $p\geq 1$ be an integer and let $\Z_p^\times$ be the multiplicative group of invertible elements of $\Z_p$.
Let $a:\Z_p^\times \to \Z$ be a map such that
\begin{enumerate}
\item $\sum_{j\in \Z_p^\times} a(j) =0$,
\item $\forall j \in \Z_p^\times, \ a(j)=a(-j)$,
\item ${\displaystyle \prod_{j\in \Z_p^\times} (\zeta^j-1)^{a(j)}=1}$ for any $p$-th root of unity $\zeta \neq 1$.
\end{enumerate}
Then, $a(j)=0$ for all $j\in \Z_p^\times$.
\end{lemma}

\begin{proof}[Proof of Theorem \ref{th:lens}]
It remains to prove the necessary condition for two lens spaces $L_{p,q}$ and $L_{p',q'}$ 
to be homeomorphic with their orientations preserved. 
The interesting case is when $p\geq 3$, which we assume.

Let $f:L_{p,q} \to L_{p',q'}$ be an orientation-preserving homeomorphism.
Because $f$ induces an isomorphism in homology and because $H_1(L_{p,q})\simeq \Z_p$, we must have $p=p'$. 
Let $T \in H_1(L_{p,q})$ and $T' \in H_1(L_{p,q'})$ be the preferred generators defined by the Heegaard splittings.
There is a $k\in \Z_p^\times$  such that $T'=k\cdot f_*(T)$.
For any $p$-th root of unity $\zeta\neq 1$,
we consider the ring homomorphism $\varphi:\Z[H_1(L_{p,q})] \to \C$ defined by $\varphi(T):= \zeta$,
as well as the homomorphism $\varphi' := \varphi \circ f_*^{-1}: \Z[H_1(L_{p,q'})] \to \C$.
Since the Reidemeister torsion is a topological invariant, 
we have $\tau^{\varphi}(L_{p,q})=\tau^{\varphi'}(L_{p,q'})$ and we deduce from Lemma \ref{lem:torsion_lens} that
\begin{equation}
\label{eq:zeta}
\zeta^u \cdot (\zeta-1) \cdot (\zeta^r-1) = (\zeta^k-1) \cdot (\zeta^{kr'}-1)
\end{equation}
where $r:=q^{-1}$, $r':=(q')^{-1}$ and $u\in \Z_p$ is unknown. 
If we multiply (\ref{eq:zeta}) by its conjugate, we get the formula
\begin{equation}
\label{eq:zeta_double}
(\zeta-1) \cdot (\zeta^r-1) \cdot (\zeta^{-1}-1) \cdot (\zeta^{-r}-1)
= (\zeta^k-1) \cdot (\zeta^{kr'}-1) \cdot (\zeta^{-k}-1) \cdot (\zeta^{-kr'}-1).
\end{equation}
For all $j\in \Z_p^\times$, let $m(j)$ be the number of times $j$ appears in the sequence $(1,-1,r,-r)$
and let $m'(j)$ be the number of times it appears in the sequence $(k,-k,kr',-kr')$.
We have $m(j)=m(-j)$ and $\sum_{j\in \Z_p^\times} m(j) =4$, and similarly for $m'$. 
We deduce from this  and from (\ref{eq:zeta_double}) that the map
$a:=m-m'$ satisfies the hypothesis of Lemma \ref{lem:Franz}. Thus, we have $m=m'$ 
so that the two sequences $(1,-1,r,-r)$ and $(k,-k,kr',-kr')$ coincide up to some permutation, which we write
\begin{equation}
\label{eq:quadruplets}
\{1,-1,r,-r\} = \{k,-k,kr',-kr'\}.
\end{equation}

Moreover, let $P\in \C[X]$ be a polynomial such that the identity
$$
P(T) = T^u \cdot (T-1) \cdot (T^r-1) -(T^k-1) \cdot (T^{kr'}-1)
$$ 
holds in the group ring $\C[H]$ where, following our convention, the group $H:=H_1(L_{p,q})$ is written multiplicatively.
Since $T^p=1 \in \C[H]$, we can assume that $P$ has degree at most $p-1$.
But, since the identity (\ref{eq:zeta}) holds for any $p$-th root of unity $\zeta$ (including $\zeta=1$), 
this polynomial has at least $p$ roots. Therefore, $P$ is zero and
\begin{equation}
\label{eq:zeta_universal}
T^u \cdot (T-1) \cdot (T^r-1) = (T^k-1) \cdot (T^{kr'}-1) \ \in \C[H].
\end{equation}
Equation (\ref{eq:quadruplets}) leaves out $8$ possible cases
which must be analized separately, possibly using (\ref{eq:zeta_universal}).
For instance, if $k=-r$ and $kr'=-1$, then we have $rr'=1$ and we are done.
As another example, let us consider the case when $k=1$ and $kr'=-r$.
Then, (\ref{eq:zeta_universal}) gives
\begin{equation}
\label{eq:special_case}
T^{u+r} \cdot (T-1) \cdot (T^r-1) = - (T-1) \cdot (T^{r}-1).
\end{equation}
To pursue, we point out the remarkable identity
$$
\forall h\in H^{\times}, \
(h - 1) \cdot \frac{1}{p}\left(1 + 2h + \cdots + p\cdot h^{p-1}\right) = 1 - \frac{\Sigma_H}{p} \ \in \C[H]
$$
where, again, $H$ is written multiplicatively and $\Sigma_H := \sum_{h \in H} h$.
Using that identity for $h=T$ and $h=T^r$, we get
$$
T^{u+r}\cdot \left(1 - \frac{\Sigma_H}{p}\right) = -\left(1-\frac{\Sigma_H}{p}\right)
$$
or, equivalently,
$$
T^{u+r} + 1 = \frac{2}{p} \cdot \Sigma_H 
$$
which is impossible if $p>2$. So, the case $(k=1,kr'=-r)$ must be excluded.
The $6$ remaining cases are treated in a similar way.
\end{proof}

\begin{remark}
Bonahon proved  by purely topological methods that a lens space has, up to isotopy, 
only one Heegaard splitting of genus one \cite{Bonahon}.
The topological classification of lens spaces is easily deduced from this result, 
without using the Reidemeister torsion.
The proof of Theorem \ref{th:lens} that we have presented here, which dates back
to Reidemeister \cite{Reidemeister} and Franz \cite{Franz}, has the advantage to extend to higher dimensions.
(Lens spaces can be defined in any odd dimension by generalizing the description proposed in Exercice \ref{ex:quotient_sphere}: 
see \cite[\S V]{Cohen} for instance.)
\end{remark}

To conclude with lens spaces, we  compute the maximal abelian torsion of $L_{p,q}$.
We deduce from Lemma \ref{lem:torsion_lens} that
\begin{equation}
\label{eq:maximal_lens}
\tau(L_{p,q}) = \varepsilon_p \cdot (T-1)^{-1} \cdot (T^r-1)^{-1} \ \in \Q[H]/H 
\end{equation}
where $T \in H := H_1(L_{p,q})$ is the preferred generator, 
and where $(T-1)^{-1}$ and $(T^r-1)^{-1}$ denote inverses \emph{in} the subring $\Ker(\varepsilon) \subset \Q[H]$.
(Indeed,  $T-1$ and $T^r-1$ are invertible in $\Ker(\varepsilon)$ since they project
to non-zero elements in each cyclotomic field $\neq \Q$ into which $\Q[H]$ splits.)
The unit element of the ring $\Ker(\varepsilon)$ is $1-\Sigma_H/p$ where $\Sigma_H:= \sum_{h\in H} h$,
and we have the following identity:
$$
\forall h\in H^\times, 
\ (h-1) \cdot \left(\sum_{i=0}^{p-1}\frac{2i-p+1}{2p} \cdot h^i\right) = 1 - \frac{\Sigma_H}{p}
\ \in \Ker(\varepsilon) \subset \Q[H].
$$
Thus, we obtain the formula
$$
\tau(L_{p,q}) = \varepsilon_p \cdot  
\left(\sum_{i=0}^{p-1}\frac{2i-p+1}{2p} \cdot T^i\right) \cdot 
\left(\sum_{i=0}^{p-1}\frac{2i-p+1}{2p} \cdot T^{ri} \right) \ \in \Q[H]/H.
$$

\subsection{Example: torsion of surface bundles}

\label{subsec:surface_bundles}

Let $\Sigma_g$ be a closed connected oriented surface of genus $g$.
The \emph{mapping torus} of an  orientation-preserving homeomorphism $f:\Sigma_g \to \Sigma_g$  is the $3$-manifold
$$
T_f := \left(\Sigma_g \times [-1,1]\right)/\!\sim
$$ 
where $\sim$ is the equivalence relation generated by $(x,-1) \sim (f(x),1)$ for all $x\in \Sigma_g$.
The cartesian projection $\Sigma_g \times [-1,1] \to [-1,1]$ induces a projection $T_f \to S^1$,
which makes $T_f$ a bundle with base $S^1$ and fiber $\Sigma_g$. 
Conversely, any surface bundle over $S^1$ can be obtained in that way.

We wish to compute the maximal abelian torsion of $T_f$.
For this, we can assume that $g\geq 1$ since $T_f$ only depends on the isotopy class of $f$ (up to homeomorphism)
and we have seen that $\mcg(S^2)$ is trivial.
We restrict ourselves to the case when \emph{$f$ acts trivially in homology}.
Then, $H:=H_1(T_f)$ is free abelian (isomorphic to $H_1(\Sigma_g) \oplus \Z$)
and Theorem \ref{th:Milnor-Turaev_dim_3} says that $\tau(T_f) = \Delta({T_f})$.
Thus, we are reduced in this case to compute the Alexander polynomial of $T_f$.
(See Remark \ref{rem:other_proofs} for the general case.)

Thanks to an isotopy of $f$, we can further assume that $f$ is the identity on a $2$-disk $D$.
The bordered surface $\Sigma_g \setminus \operatorname{int}(D)$ is denoted by $\Sigma_{g,1}$,
and the restriction of the homeomorphism $f$ to $\Sigma_{g,1}$ is denoted by $f|$.
We pick a base point $\star \in \partial \Sigma_{g,1}$. Then, we have
$$
H = H_1(\Sigma_g) \oplus (\Z \cdot t)
$$
where $t$ is the homology class of the circle $(\star \times [-1,1])/\!\sim$.
The group $\pi_1(\Sigma_{g,1},\star)$ is freely generated 
by the based loops $\alpha_1,\dots,\alpha_g,\alpha_1^\sharp,\dots,\alpha_g^\sharp$ shown on Figure \ref{fig:surface}. 
We denote by $\frac{\partial \ }{\partial {\alpha_1}}, \dots, \frac{\partial \ }{\partial {\alpha_g}},
\frac{\partial \ }{\partial {\alpha_1^\sharp}}, \dots, \frac{\partial \ }{\partial {\alpha_g^\sharp}}$
the free derivatives with respect to this basis,
and we consider the ``Jacobian matrix'' of $f|_*:\pi_1(\Sigma_{g,1},\star) \to \pi_1(\Sigma_{g,1},\star)$ defined by
$$
J(f|_*) :=  \left(\begin{array}{cc}
\begin{array}{ccc}
\frac{\partial f|_*(\alpha_1)}{\partial \alpha_1} & \cdots &  \frac{\partial f|_*(\alpha_g)}{\partial \alpha_1} \\
\vdots & \ddots & \vdots \\
\frac{\partial f|_*(\alpha_1)}{\partial \alpha_g} & \cdots &  \frac{\partial f|_*(\alpha_g)}{\partial \alpha_g}  
\end{array} &
\begin{array}{ccc}
\frac{\partial f|_*(\alpha_1^\sharp)}{\partial \alpha_1} & \cdots &  \frac{\partial f|_*(\alpha_g^\sharp)}{\partial \alpha_1} \\
\vdots & \ddots & \vdots \\
\frac{\partial f|_*(\alpha_1^\sharp)}{\partial \alpha_g} & \cdots &  \frac{\partial f|_*(\alpha_g^\sharp)}{\partial \alpha_g}  
\end{array} \\
\begin{array}{ccc}
\frac{\partial f|_*(\alpha_1)}{\partial \alpha_1^\sharp} & \cdots &  \frac{\partial f|_*(\alpha_g)}{\partial \alpha_1^\sharp} \\
\vdots & \ddots & \vdots \\
\frac{\partial f|_*(\alpha_1)}{\partial \alpha_g^\sharp} & \cdots &  \frac{\partial f|_*(\alpha_g)}{\partial \alpha_g^\sharp}  
\end{array} &
\begin{array}{ccc}
\frac{\partial f|_*(\alpha_1^\sharp)}{\partial \alpha_1^\sharp} & \cdots &  \frac{\partial f|_*(\alpha_g^\sharp)}{\partial \alpha_1^\sharp} \\
\vdots & \ddots & \vdots \\
\frac{\partial f|_*(\alpha_1^\sharp)}{\partial \alpha_g^\sharp} & \cdots &  \frac{\partial f|_*(\alpha_g^\sharp)}{\partial \alpha_g^\sharp}  
\end{array}
\end{array}\right).
$$
It turns out that $\tau(T_f)$ is essentially
given by the characteristic polynomial with indeterminate $t$ of (a reduction of) the matrix $J(f|_*)$.

\begin{figure}[h]
\begin{center}
{\labellist \small \hair 0pt 
\pinlabel {\Large $\star$} at 567 3
\pinlabel {$\alpha_1$} [l] at 310 105
\pinlabel {$\alpha_g$} [l] at 568 105
\pinlabel {$\alpha_1^\sharp$} [b] at 216 70
\pinlabel {$\alpha_g^\sharp$} [b] at 472 65
\endlabellist}
\includegraphics[scale=0.5]{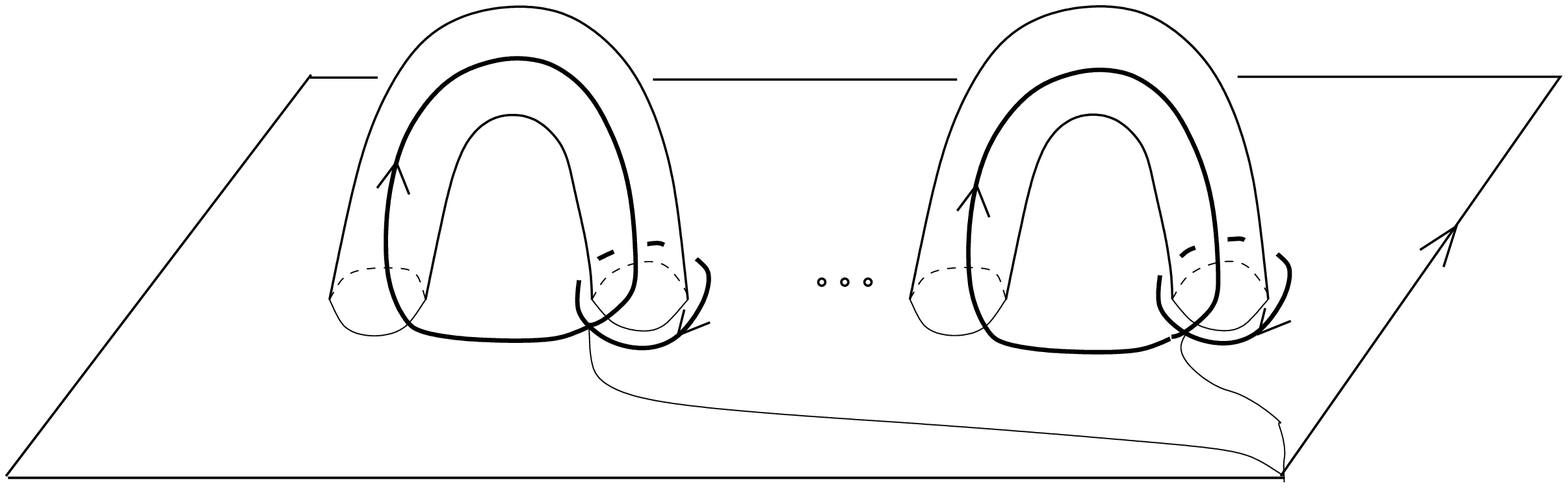}
\end{center}
\caption{The surface $\Sigma_{g,1}$ and a system of meridians and parallels $(\alpha,\alpha^\sharp)$.}
\label{fig:surface}
\end{figure}

\begin{proposition}
\label{prop:surface_bundle}
With the above notation and assumption, we have 
$$
\tau(T_f) = (t-1)^{-2} \cdot \det\Big(t \cdot \operatorname{I}_{2g} - \incl_* J(f|_*)\Big) \in \Z[H]/\pm H
$$
where $\incl_*: \pi_1(\Sigma_{g,1}) \to H_1(T_f) =H$ is induced by the inclusion.
\end{proposition}

\begin{proof}
We are asked to compute
$$
\tau(T_f) = \Delta({T_f}) = \ord  H_1(\widehat{T}_f)
$$
where  $\widehat{T}_f$ denotes the maximal (free) abelian cover of $T_f$.
If $\widehat{\Sigma}_g$ is the maximal abelian cover of $\Sigma_g$,
then $\widehat{T}_f$ can be realized as $\widehat{\Sigma}_g \times \R$, so that it decomposes into ``slices'':
$$
\widehat{T}_f = \cdots \cup \left(\widehat{\Sigma}_g \times [-1,0]\right) \cup 
\left(\widehat{\Sigma}_g \times [0,1]\right) \cup \left(\widehat{\Sigma}_g \times [1,2]\right) \cup \cdots
$$
The covering transformation corresponding to $t \in H$ shifts the slices from left to right.

We set $S:=H_1(\Sigma_g)$ so that $H = S \oplus (\Z \cdot t)$.
Observe that $\Z[H]$ is free as a $\Z[S]$-module: 
it follows that the functor $\mathcal{E} := \Z[H] \otimes_{\Z[S]} -$ is exact.
An application of the Mayer--Vietoris theorem shows that, in the category of $\Z[H]$-modules, we have 
\begin{equation}
\label{eq:cokernel}
H_1(\widehat{T}_f) = \Coker\left(\mathcal{E}\widehat{f}_* - t\cdot \Id: 
\mathcal{E} H_1(\widehat{\Sigma}_g) \longrightarrow  \mathcal{E} H_1(\widehat{\Sigma}_g)\right)
\end{equation}
where $\widehat{f}: \widehat{\Sigma}_g \to \widehat{\Sigma}_g$ is the lift of $f$ that fixes a preferred lift $\widehat{\star}$ of $\star$.
(The map $\widehat{f}_*$ is $\Z[S]$-linear by our assumption that $f$ acts trivially on $S$.)

The maximal abelian cover $\widehat{\Sigma}_{g,1}$ of $\Sigma_{g,1}$ is a surface with boundary, 
and the group of covering transformations $H_1(\Sigma_{g,1})\simeq S$ acts freely and transitively
on the set of its boundary components. 
Therefore, $\widehat{\Sigma}_{g}$ is obtained from $\widehat{\Sigma}_{g,1}$
by gluing $2$-disks $\widehat{D}_s$ indexed by $s \in S$. We deduce the short exact sequence
$$
\xymatrix{
0 \ar[r] & \Z[S] \cdot \widehat{D} \ar[r] & H_1(\widehat{\Sigma}_{g,1}) \ar[r] & H_1(\widehat{\Sigma}_{g}) \ar[r] & 0.
}
$$
The following is a commutative diagram in the category of $\Z[H]$-modules:
$$
\xymatrix{
0 \ar[r] & \Z[H] \ar[d]_{(1-t)\cdot \Id}\ar[r] & 
\mathcal{E} H_1\left(\widehat{\Sigma}_{g,1}\right) \ar[r] \ar[d]_-{\mathcal{E}\widehat{f|}_* - t\cdot \Id} & 
\mathcal{E} H_1\left(\widehat{\Sigma}_{g}\right) \ar[r]  \ar[d]_-{\mathcal{E}\widehat{f}_* - t\cdot \Id}& 0 \\
0 \ar[r] & \Z[H] \ar[r] & \mathcal{E} H_1\left(\widehat{\Sigma}_{g,1}\right) \ar[r] & 
\mathcal{E} H_1\left(\widehat{\Sigma}_{g}\right) \ar[r] & 0.
}
$$
An application of the ``snake'' lemma relates the cokernels of the three vertical maps
by a short exact sequence. Besides, the order of a module is multiplicative in short exact sequences:
this fact generalizes Exercice \ref{ex:multiplicativity} and can be found in \cite[\S 3.3]{Hillman}. We deduce that
\begin{equation}
\label{eq:order_1}
\ord \Coker\left(\mathcal{E}\widehat{f|}_* - t\cdot \Id\right)
= \ord \Coker \left(\mathcal{E}\widehat{f}_* - t\cdot \Id\right) \cdot (1-t).
\end{equation}

The surface $\Sigma_{g,1}$ deformation retracts to the union of based loops
$X_{2g}:=\alpha_1 \cup \alpha_1^\sharp \cup \dots \cup \alpha_g \cup \alpha_g^\sharp$,
which becomes a bouquet of $2g$ circles when all the basing arcs are collapsed to $\star$. 
Let $p:\widehat{\Sigma}_{g,1} \to \Sigma_{g,1}$ be the projection of the maximal abelian cover of $\Sigma_{g,1}$.
Then, $\widehat{X}_{2g} := p^{-1}(X_{2g})$ is the maximal abelian cover of $X_{2g}$. 
There is a map $\phi:\Sigma_{g,1} \to \Sigma_{g,1}$ homotopic to $f|$ 
such that $\phi(X_{2g}) = X_{2g}$ and $\phi(\star)=\star$,
and we have the following commutative square in the category of $\Z[S]$-modules:
$$
\xymatrix{
H_1\left(\widehat{X}_{2g}\right) \ar[r]_-\simeq^{\incl_*} \ar[d]_-{\widehat{\phi}_*}&
H_1\left(\widehat{\Sigma}_{g,1}\right) \ar[d]^-{\widehat{f|}_*}& \\
H_1\left(\widehat{X}_{2g}\right) \ar[r]^-\simeq_{\incl_*} & H_1\left(\widehat{\Sigma}_{g,1}\right)
}
$$
It follows that
\begin{equation}
\label{eq:order_2}
\ord \Coker\left(\mathcal{E}\widehat{f|}_* - t\cdot \Id\right)
= \ord \Coker\left(\mathcal{E}\widehat{\phi}_* - t\cdot \Id\right).
\end{equation}

Finally, the long exact sequence for the pair $(\widehat{X}_{2g}, p^{-1}(\star))$
gives the short exact sequence
$$
\xymatrix{
0 \ar[r] & H_1(\widehat{X}_{2g}) \ar[r] & H_1\left(\widehat{X}_{2g},p^{-1}(\star) \right) \ar[r]  & 
I(S) \cdot \widehat{\star} \ar[r] & 0
}
$$
where $I(S)$ is the augmentation ideal of $\Z[S]$, 
\ie the ideal generated by all the $(s-1)$ with $s \in S$.
Thus, we obtain another  commutative diagram in the category of $\Z[H]$-modules:
$$
\xymatrix{
0 \ar[r] & \mathcal{E} H_1(\widehat{X}_{2g}) \ar[d]_-{\mathcal{E} \widehat{\phi}_* - t\cdot \Id}
\ar[r] &\mathcal{E}  H_1\left(\widehat{X}_{2g},p^{-1}(\star) \right) 
\ar[r] \ar[d]_-{\mathcal{E} \widehat{\phi}^r_* - t\cdot \Id} & 
\Z[H]  \cdot I(S) \ar[r]  \ar[d]_-{(1-t)\cdot \Id}& 0\\
0 \ar[r] & \mathcal{E} H_1(\widehat{X}_{2g}) \ar[r] & 
\mathcal{E}H_1\left(\widehat{X}_{2g},p^{-1}(\star) \right) \ar[r]  & 
\Z[H]  \cdot I(S) \ar[r] & 0
}
$$
where $\widehat{\phi}^r$ denotes the relative version of $\widehat{\phi}$.
Again, by applying the ``snake'' lemma and the multiplicativity of orders, we obtain
\begin{equation}
\label{eq:order_3}
\ord \Coker(\mathcal{E} \widehat{\phi}^r_* - t\cdot \Id)
= \ord \Coker(\mathcal{E} \widehat{\phi}_* - t\cdot \Id) \cdot (1-t).
\end{equation}
The $\Z[S]$-module $H_1(\widehat{X}_{2g},p^{-1}(\star))$ is freely generated by the lifts 
$\widehat{\alpha}_1,\dots, \widehat{\alpha}_g,$ $\widehat{\alpha}^\sharp_1,\dots, \widehat{\alpha}^\sharp_g$
of the loops $\alpha_1,\dots,\alpha_g,\alpha_1^\sharp,\dots,\alpha_g^\sharp$ starting at $\widehat{\star}$.
Moreover, Proposition \ref{prop:topological_fundamental_formula} tells us that 
the matrix of $\widehat{\phi}_*^r$ in that basis is the reduction to $\Z[H_1(X_{2g})]$ 
of the jacobian matrix $J(\phi_*)=J(f|_*)$. So, we have
$$
\ord \Coker(\mathcal{E} \widehat{\phi}^r_* - t\cdot \Id) = 
\det \left(J(f|_*) - t\cdot I_{2g}\right).
$$
We conclude thanks to (\ref{eq:cokernel}), (\ref{eq:order_1}), (\ref{eq:order_2}) and (\ref{eq:order_3}).
\end{proof}

\begin{remark}
\label{rem:other_proofs}
Proposition \ref{prop:surface_bundle} is proved by Turaev in \cite{Turaev_book_dim_3} 
with no restriction on $f\in \mcg(\Sigma_g)$ using other kinds of arguments.
Observing that $T_f$ has a preferred Heegaard splitting of genus $2g+1$,
another proof would be to apply Lemma \ref{lem:torsion_Heegaard_splitting}.
Proposition \ref{prop:surface_bundle} is the analogue for $3$-manifolds 
of a classical formula expressing the Alexander polynomial of a closed braid (respectively, of a pure braid) 
from its Burau representation (respectively, from its Gassner representation) -- see Birman's book \cite{Birman}.
\end{remark}

\begin{remark}
\label{rem:Magnus_Torelli}
Just as the closed surface $\Sigma_g$, the bordered surface $\Sigma_{g,1}$ has a 
 \emph{mapping class group} defined as
the group of isotopy classes of homeomorphisms $\Sigma_{g,1} \to \Sigma_{g,1}$ which fix pointwise the boundary:
$$
\mcg(\Sigma_{g,1}) := \operatorname{Homeo}\left(\Sigma_{g,1}, \partial \Sigma_{g,1}\right)/ \cong.
$$
The natural action  on the free group $\pi_1(\Sigma_{g,1},\star) = \Free(\alpha,\alpha^\sharp)$ 
allows us to regard $\mcg(\Sigma_{g,1})$ as a subgroup of $\Aut\left(\Free(\alpha,\alpha^\sharp)\right)$.
In particular, the Magnus representation defined in \S \ref{subsec:Magnus} applies to $\mcg(\Sigma_{g,1})$:
$$
\Magnus: \mcg(\Sigma_{g,1}) \longrightarrow \operatorname{GL}\left(2g;\Z[\Free(\alpha,\alpha^\sharp)]\right),
f \longmapsto \overline{J(f_*)}.
$$
An important subgroup of $\mcg(\Sigma_{g,1})$ is the \emph{Torelli group} of $\Sigma_{g,1}$, namely
$$
\mathcal{I}(\Sigma_{g,1}) := \Ker\big(\mcg(\Sigma_{g,1}) \longrightarrow \Aut(S), \ f \longmapsto f_* \big)
$$
where $S:= H_1(\Sigma_{g,1})$. By Corollary \ref{cor:Magnus_abelian}, we obtain a group homomorphism
$$
\Magnusab:
\mathcal{I}(\Sigma_{g,1}) \longrightarrow \operatorname{GL}(2g;\Z[S])
$$
with values in a group of matrices over a commutative ring.
As illustrated by Proposition \ref{prop:surface_bundle},
this representation appears naturally in the context of abelian Reidemeister torsions,
but it is unfortunately \emph{not} faithfull \cite{Suzuki}. 
Its kernel consists of those elements of $\mathcal{I}(\Sigma_{g,1})$ which act trivially at the level 
of the second solvable quotient $\pi_1(\Sigma_{g,1},\star)/\pi_1(\Sigma_{g,1},\star)''$ of the fundamental group.
Explicit elements of the kernel are described by Suzuki in \cite{Suzuki_bis}.
\end{remark}

\section{Homotopy invariants derived from the abelian Reidemeister torsion}

In this section, we relate the maximal abelian torsion of $3$-manifolds
to some invariants which are defined with just a little bit of algebraic topology,
and which only depend on the \emph{oriented} homotopy type\footnote{
Two $3$-manifolds $M$ and $M'$ are said to have the \emph{same oriented homotopy type}
if there exists a homotopy equivalence $f:M \to M'$ such that $f_*\left(\left[M\right] \right) = \left[M'\right] \in H_3\left(M'\right)$.}.

\subsection{The triple-cup product forms}

Let $M$ be a $3$-manifold. Apart from the fundamental group,
the first homotopy invariant of $M$ which comes to one's mind is probably the cohomology ring.
Poincar\'e duality shows that all the (co)homology groups of $M$ are determined by $H_1(M;\Z)$
and that, for any integer $r\geq 0$, the cohomology ring $H^*(M;\Z_r)$ is determined by the \emph{triple-cup product form}
$$
u^{(r)}_M: H^1(M;\Z_r) \times H^1(M;\Z_r) \times H^1(M;\Z_r) \longrightarrow \Z_r.
$$
This form, defined  by 
$$
\forall x,y,z \in H^1(M;\Z_r), \quad
u^{(r)}_M(x,y,z) := \langle x \cup y \cup z, [M] \rangle \ \in \Z_r 
$$
is trilinear and skew-symmetric. 
Clearly, its isomorphism class is an invariant of the oriented homotopy type of $M$.

For instance, $u^{(0)}_M$ is a trilinear alternate form on the free abelian group $H^1(M;\Z)$.
It turns out that any algebraic object of this kind can be realized by a $3$-manifold.

\begin{theorem}[Sullivan \cite{Sullivan}]
\label{th:Sullivan}
Let $u: \Z^b \times \Z^b \times \Z^b \to \Z$ be an alternate trilinear form.
Then, there exists a $3$-manifold $M$ and an isomorphism $H^1(M;\Z)\simeq\Z^b$
through which $u^{(0)}_M$ and $u$ correspond.
\end{theorem}

\noindent
This realization result has first been proved for $\Z_2$ coefficients by Postnikov \cite{Postnikov},
and it has been generalized to the coefficient ring $\Z_r$ for any integer $r\geq 0$ by Turaev \cite{Turaev_cohomology}.

To prove Theorem \ref{th:Sullivan}, we need to introduce a way of modifying $3$-manifolds.
For this, we consider the torus $T^3 :=  S^1 \times S^1 \times S^1$ and its Heegaard splitting of genus $3$
\begin{equation}
\label{eq:Heegaard_torus}
T^3 = A \cup B
\end{equation}
where $A$ is a closed regular neighborhood of the graph 
$(S^1 \times 1 \times 1) \cup (1 \times S^1 \times 1) \cup (1 \times 1 \times S^1)$
and $B$ is the complement of $\int(A)$. 
Given a $3$-manifold $M$ and an embedding $a:A \hookrightarrow M$, 
we can ``cut'' $A$ out of $M$ and ``replace'' it by $B$. More formally, we define the $3$-manifold
$$
M_a := \left(M \setminus \int\ a(A) \right) \cup_{a_{ \partial}} (-B)
$$
where $a_\partial: \partial B =\partial A \hookrightarrow M$ is the restriction of $a$ to the boundary.

\begin{definition}
\label{def:surgery}
The move $M \leadsto M_a$ is called the \emph{$T^3$-surgery} along $a$.
\end{definition}

\noindent
This terminology is justified by the facts that  (\ref{eq:Heegaard_torus}) is the unique Heegaard splitting of $T^3$ of genus $3$ 
(according to Frohman \& Hass \cite{FH})
and there is no Heegaard splitting of $T^3$ of lower genus (for obvious homological reasons).
In this sense, the decomposition (\ref{eq:Heegaard_torus}) of $T^3$ is \emph{canonical}.
The $T^3$-surgery  is implicitly defined in Sullivan's paper \cite{Sullivan}.
Rediscovered by Matveev under the name of ``Borromean surgery'' \cite{Matveev}, 
this kind of modification has become fundamental in the theory of finite-type invariants. 
(See \S \ref{subsec:polynomial} in this connection.) 

\begin{proof}[Proof of Theorem \ref{th:Sullivan}] 
We reformulate Sullivan's proof \cite{Sullivan} in terms of $T^3$-surgeries.
Consider the basis $(e_1,e_2,e_3)$ of $H_1(T^3)$ where
$$
e_1 := [S^1 \times 1 \times 1], \ e_2 := [1 \times S^1 \times 1]
\quad \hbox{ and } \quad e_3 := [1 \times 1 \times S^1],
$$
and let $(e_1^*,e_2^*,e_3^*)$ be the dual basis of 
$H^1(T^3) \simeq \Hom \left(H_1(T^3),\Z \right)$. One can check that
\begin{equation}
\label{eq:cohomology_torus}
\langle e_1^* \cup e_2^* \cup e_3^*, [S^1 \times S^1 \times S^1]\rangle = 1 \ \in\Z.
\end{equation}

Let us now compute how the triple-cup product form changes under a $T^3$-surgery $M \leadsto M_a$.
An application of the Mayer--Vietoris theorem shows that there exists 
a unique isomorphism $\Phi_a: H_1(M) \to H_1(M_a)$ such that the following diagram is commutative:
$$
\xymatrix @!0 @R=0.8cm @C=2.5cm {
& H_1(M) \ar[dd]^-{\Phi_a}_-\simeq \\
H_1(E) \ar@{->>}[ru] \ar@{->>}[rd]& \\
& H_1(M_a)
}
\quad \hbox{where } E:= M \setminus \int\ a(A).
$$
We also denote by $\Phi_a^*: H^1(M_a) \to H^1(M)$ the dual isomorphism in cohomology.

\begin{quote}
\textbf{Claim.} For all $y_1',y_2',y_3' \in H^1(M_a)$, we have
\begin{eqnarray*}
\left\langle y_1' \cup y_2' \cup y_3', [M_a] \right\rangle
-  \left\langle \Phi_a^*(y_1') \cup \Phi_a^*(y_2') \cup \Phi_a^*(y_3'), [M] \right\rangle\\
= - 
\left| \begin{array}{ccc}
\left\langle \Phi_a^*(y_1') , a_*(e_1) \right\rangle & \left\langle \Phi_a^*(y_1') , a_*(e_2) \right\rangle 
& \left\langle \Phi_a^*(y_1') , a_*(e_3) \right\rangle\\
\left\langle \Phi_a^*(y_2') , a_*(e_1) \right\rangle & \left\langle \Phi_a^*(y_2') , a_*(e_2) \right\rangle 
& \left\langle \Phi_a^*(y_2') , a_*(e_3) \right\rangle\\
\left\langle \Phi_a^*(y_3') , a_*(e_1) \right\rangle & \left\langle \Phi_a^*(y_3') , a_*(e_2) \right\rangle 
& \left\langle \Phi_a^*(y_3') , a_*(e_3) \right\rangle\\
\end{array} \right|.
\end{eqnarray*}
\end{quote}

\noindent
To prove this, we consider  the singular $3$-manifold
$$
N := E \cup_{a_{\partial}} (A \cup B) 
$$
which contains $M$, $M_a$ as well as $T^3$.
The inclusions induce isomorphisms $\incl^*:H^1(N) \to H^1(M)$ and $\incl^*:H^1(N) \to H^1(M_a)$.
Let $z_1,z_2,z_3 \in H^1(N)$ correspond to $y'_1,y'_2,y'_3 \in  H^1(M_a)$ by $\incl^*$.
By definition of $\Phi_a$, we have $\incl^*(z_i) = \Phi_a^*(y_i')$ for $i=1,2,3$. Thus, we obtain
\begin{eqnarray*}
&&\left\langle y_1' \cup y_2' \cup y_3', [M_a] \right\rangle
-  \left\langle \Phi_a^*(y_1') \cup \Phi_a^*(y_2') \cup \Phi_a^*(y_3'), [M] \right\rangle\\
&=& \left\langle z_1\cup z_2 \cup z_3, \incl_*([M_a]) - \incl_*([M]) \right\rangle\\
&=& - \left\langle z_1\cup z_2 \cup z_3, \incl_*([T^3]) \right\rangle\\
&=& - \left\langle \incl^*(z_1)\cup \incl^*(z_2) \cup \incl^*(z_3), [T^3] \right\rangle
\end{eqnarray*}
so that the claim can be deduced from (\ref{eq:cohomology_torus}).

We conclude thanks to the claim in the following way.
Let $b\geq 0$ be an integer and let $H:= H_1(\sharp^b S^1 \times S^2)$.
Then, $H^* := H^1(\sharp^b S^1 \times S^2)$ can be identified with $\Z^b$.
There is a non-singular bilinear pairing
$$
\langle -,- \rangle: \Lambda^3 H^* \times \Lambda^3 H \longrightarrow \Z
$$
defined by
$$
\langle y_1 \wedge y_2 \wedge y_3, x_1 \wedge x_2 \wedge x_3 \rangle =
\left| \begin{array}{ccc}
\langle y_1,x_1 \rangle & \langle y_1 , x_2 \rangle & \langle y_1 , x_3 \rangle\\
\langle y_2,x_1 \rangle & \langle y_2 , x_2 \rangle & \langle y_2 , x_3 \rangle\\
\langle y_3,x_1 \rangle & \langle y_3 , x_2 \rangle & \langle y_3 , x_3 \rangle
\end{array} \right|
$$
which allows us to identify $\Hom(\Lambda^3 H^*,\Z)$ with $\Lambda^3 H$.
Thus, we can write the trilinear alternate form $u: \Z^b \times \Z^b \times \Z^b \to \Z$ as
$$
u = \sum_{j=1}^n l_1^{(j)} \wedge l_2^{(j)} \wedge l_3^{(j)} \ \in \Lambda^3 H.
$$
For each $j=1,\dots,n$, we consider an embedding $a^{(j)}: A \to \sharp^b S^1 \times S^2$ such that
${a^{(j)}}_*(e_i) = l_i^{(j)}$ for $i=1,2,3$, and we assume that the  handlebodies 
$a^{(1)}(A),\dots,a^{(n)}(A)$ are pairwise disjoint.
Thus, we can perform simultaneously the $n$ $T^3$-surgeries along $a^{(1)},\dots, a^{(n)}$
and we denote by $M'$ the resulting $3$-manifold. It follows from the above claim that
$$
u^{(0)}_{M'} - u^{(0)}_{\sharp^b S^1 \times S^2} = 
- \sum_{j=1}^n l_1^{(j)} \wedge l_2^{(j)} \wedge l_3^{(j)} = -u 
\ \in \Hom(\Lambda^3 H^*,\Z) = \Lambda^3 H
$$
where $H^1(M')$ is identified with $H^1(\sharp^b S^1 \times S^2) = H^*=\Z^b$ as explained before. 
The triple-cup product form of $S^1 \times S^2$ being trivial (since the rank of $H^1(S^1\times S^2)$
is less than $3$), we have $u^{(0)}_{\sharp^b S^1 \times S^2} =0$. 
So, $u$ is realized as the triple-cup product form of $M'$ (with the opposite orientation).
\end{proof}

\begin{remark}
The classification of trilinear alternate forms seems to be a hard problem.
So, given two $3$-manifolds $M$ and $M'$, it can be quite difficult to decide 
whether the cohomology ring of $M$ is isomorphic to the cohomology ring of $M'$.
\end{remark}

We shall now state a formula due to Turaev, which relates the torsion $\tau(M)$ to the cohomology ring of $M$. 
To simplify the exposition, we will only consider $3$-manifolds $M$ with $\beta_1(M)\geq 3$,
and we will restrict ourselves to the triple-cup product form with $\Z$ coefficients.
The reader is refered to \cite[\S III \& \S IX]{Turaev_book_dim_3} for the other cases.

Let $L$ be a free abelian group of finite rank $b$, and let 
$u:L \times L \times L \to \Z$ be a skew-symmetric trilinear form. We consider the \emph{adjoint} of $u$
$$
\widehat{u}: L \times L \longrightarrow L^*, \ (y_1,y_2) \longmapsto u(y_1,y_2,-) 
$$
where $L^*$ denotes $\Hom(L,\Z)$. 
Let $e=(e_1,\dots,e_b)$ be a basis of $L$, and let $(\widehat{u} / e)$  be the matrix of $\widehat{u}$ in the basis $e$.
This is  a $b\times b$ matrix with coefficients in the symmetric algebra $S(L^*) \supset L^*$.
Because of the skew-symmetry of $u$, its determinant happens to be zero,
and we consider instead its $(i,j)$-th minor $(\widehat{u}/e)_{i,j}$. Turaev proves that
$$
(\widehat{u}/e)_{i,j} = (-1)^{i+j} \cdot e_i^* e_j^* \cdot \operatorname{Det}(u)
\ \in S^{b-1}(L^*)
$$
for a certain $\operatorname{Det}(u) \in S^{b-3}(L^*)$ which does not depend on $e$.

\begin{example}
The matrix $(\widehat{u}/e)$ being antisymmetric, we have $\operatorname{Det}(u)=0$ whenever $b$ is even.
So, the tensor $\operatorname{Det}(u)$ is only interesting for odd $b\geq 3$. 
For instance, for $b=3$, we have $\operatorname{Det}(f) = u(e_1,e_2,e_3)^2$. 
\end{example}

Recall from \S \ref{subsec:maximal_dim_3} that, for a $3$-manifold $M$ with $\beta_1(M)\geq 2$,
$\tau(M)$ lives in $\Z[H_1(M)]$. Consider the augmentation ideal of $\Z[H_1(M)]$:
$$
I\left(H_1(M)\right) := \Ker\left(\varepsilon: \Z[H_1(M)] \longrightarrow \Z, 
\ \sum_{h} z_h \cdot h \longmapsto \sum_{h} z_h\right).
$$
The ring  $\Z[H_1(M)]$ is filtered by the powers of the augmentation ideal:
$$
\Z[H_1(M)] \supset I(H_1(M)) \supset I(H_1(M))^2 \supset I(H_1(M))^3 \supset \cdots 
$$
and one can wonder what is the ``leading term'' of $\tau(M)$ with respect to that filtration.
The triple-cup product form with $\Z$ coefficients gives a partial answer to that question.

\begin{theorem}[Turaev \cite{Turaev_book_dim_3}]
\label{th:torsion_cohomology}
Let $M$ be a $3$-manifold with first Betti number $b:=\beta_1(M)\geq 3$, for which we set $H:=H_1(M)$ and $G:= H/\Tors H$.
Expand the determinant of the trilinear skew-symmetric form $u^{(0)}_M$ as 
$$
\operatorname{Det}\left(u^{(0)}_M\right) = \sum_{g_1, \dots, g_{b-3} \in G} 
u_{g_1,\dots,g_{b-3}}\cdot g_1 \cdots g_{b-3} \ \in S^{b-3}(G).
$$
Then, we have
$$
\tau(M) = \left| \Tors H \right|\cdot \! \!  \sum_{g_1, \dots, g_{b-3} \in G} 
u_{g_1,\dots,g_{b-3}}\cdot (\widehat{g}_1-1) \cdots (\widehat{g}_{b-3}-1) \ \in \frac{I(H)^{b-3}}{I(H)^{b-2}} 
$$
where $\widehat{g_i}\in H$ denotes a lift of $g_i\in G$.
\end{theorem}

\noindent
Theorem \ref{th:torsion_cohomology} can be proved starting from Lemma \ref{lem:torsion_Heegaard_splitting}:
see \cite[\S III.2]{Turaev_book_dim_3}.
It is the analogue for $3$-manifolds of a knot-theoretic result by Traldi \cite{Traldi_Milnor,Traldi_Conway}:
the leading term of the Alexander polynomial of a link in $S^3$ (whose linking matrix is assumed to be trivial) 
is determined by its Milnor's triple linking numbers. 

\begin{exercice}
Check Theorem \ref{th:torsion_cohomology} for the trivial surface bundle $\Sigma_g \times S^1$. 
\end{exercice}

\subsection{The linking pairing}

\label{subsec:linking}

The \emph{linking pairing} of  a $3$-manifold $M$ is the map
$$
\lambda_M: \Tors H_1(M) \times \Tors H_1(M) \longrightarrow \Q/\Z
$$
defined, for all $x,y \in \Tors H_1(M)$, by
\begin{equation}
\label{eq:linking_pairing}
\lambda_M(x,y) :=   \frac{1}{m} \Sigma \centerdot Y \mod 1.
\end{equation}
Here,  $X$ and $Y$ are disjoint oriented knots which represent $x$ and $y$ respectively, 
$\Sigma$ is a compact connected oriented surface which is transverse to $Y$ and whose boundary goes $m\geq 1$ times around $X$.
We have denoted by $\Sigma \centerdot Y \in \Z$ the intersection number.

\begin{lemma}
The map $\lambda_M$ is a well-defined, bilinear, symmetric and non-degenerate pairing.
Up to isomorphism, $\lambda_M$ only depends on the oriented homotopy type of $M$.
\end{lemma}

\begin{proof}
To prove that $\lambda_M$ is well-defined and bilinear, we give it a more concise definition.
Let $B: H_2(M;\Q/\Z) \to H_1(M)$ be the Bockstein homomorphism 
for the short exact sequence of coefficients
$$
\xymatrix{
0 \ar[r] & \Z \ar[r] &  \Q \ar[r] & \Q/\Z \ar[r] & 0. 
}
$$
The image of $B$ is the kernel of the map $H_1(M;\Z) \to H_1(M;\Q)$
and, so, coincides with $\Tors H_1(M)$. 
Then, formula (\ref{eq:linking_pairing}) can be written as
$$
\lambda_M(x,y) :=  B^{-1}(x) \bullet y \ \in \Q/\Z
$$
where $\bullet: H_2(M;\Q/\Z) \times H_1(M;\Z) \to \Q/\Z$ denotes the homological intersection pairing,
and $B^{-1}(x)$ stands for an antecedent of $x$ by $B$.
Thus, we are led to check that, for any antecedents $u,u'$ of $x$ by $B$, $(u-u') \bullet y$ vanishes. 
But, this follows from the fact that $u-u' \in H_2(M;\Q/\Z)$ comes from some $v\in H_2(M;\Q)$
and $v \bullet y=0$ since the image of $y$ in $H_1(M;\Q)$ is trivial.

To check the symmetry of $\lambda_M$, \ie $\lambda_M(x,y) = \lambda_M(y,x)$,
we choose a compact connected oriented surface $\Theta$ whose boundary goes $n \geq 1$ times around $Y$.
We also assume that $\Theta$ is transverse to $\Sigma$, so that $\Sigma \cap \Theta$ is an oriented $1$-dimensional manifold.
Thus, the boundary points of $\Sigma \cap \Theta$ come by pairs. 
From this, we obtain  $\partial \Sigma \centerdot \Theta - \Sigma \centerdot \partial \Theta    = 0 \in \Z$
or, equivalently, $m \cdot X \centerdot \Theta - n \cdot \Sigma \centerdot Y = 0 \in \Z$.
We deduce that $\frac{1}{n} \Theta \centerdot  X = \frac{1}{m} \Sigma \centerdot Y \in \Q$,
hence $\lambda_M(x,y) = \lambda_M(y,x)\in \Q/\Z$.

To prove that $\lambda_M$ is non-degenerate, we use the  commutative diagram
$$
\xymatrix @!0 @R=1.5cm @C=4.5cm {
H_2(M;\Q/\Z) \ar[d]_-B \ar[r]^-{\hbox{\scriptsize Poincar\'e}}_-\simeq & H^1(M;\Q/\Z) 
\ar[r]^-{\hbox{\scriptsize evaluation}}_-\simeq & 
\Hom\left(H_1(M),\Q/\Z\right) \ar@{->>}[d]^-{\hbox{\scriptsize restriction}}\\
\Tors H_1(M)\ar[rr]_-{\widehat{\lambda}_M}&& \Hom\left(\Tors H_1(M),\Q/\Z\right).
}
$$
Here $\widehat{\lambda}_M: x \mapsto \lambda_M(x,-)$ denotes the adjoint of $\lambda_M$.
We deduce that $\widehat{\lambda}_M$ is surjective and, so, bijective.
The  diagram also shows that the isomorphism class of $\lambda_M$ only depends
on the oriented homotopy type of $M$.
\end{proof}

Any symmetric non-degenerate bilinear pairing on a finite abelian group can be realized by a $3$-manifold.

\begin{theorem}[Wall \cite{Wall}]
Let  $\lambda: G \times G \to \Q/\Z$  be a symmetric non-degenerate bilinear form on a finite abelian group $G$.
Then, there exists a $3$-manifold $M$
and an isomorphism $\Tors H_1(M) \simeq G$ through which $\lambda_M$ corresponds to $\lambda$. 
\end{theorem}

\begin{proof}[Sketch of the proof]
This is an application of an algebraic result by Wall.
Let $H$ be a finitely generated free abelian group and let $f:H \times H \to \Z$ be a symmetric bilinear form. 
Then, one can associate to $(H,f)$ the finite abelian group 
$$
G_f := \Tors \Coker(\widehat{f})
$$
where $\widehat{f}: H \to \Hom(H,\Z)$ is the adjoint of $f$. 
Furthermore, there is associated to $(H,f)$ a symmetric bilinear form
$$
\lambda_f: G_f \times G_f \longrightarrow \Q/\Z.
$$
The definition is
$$
\forall \{u\}, \{v\} \in G_f \subset \Hom(H,\Z)/\widehat{f}(H), \
\lambda_f\left(\{u\}, \{v\}\right) :=  f_\Q \left(\widehat{u},\widehat{v} \right)  \mod 1
$$
where $f_\Q: (H\otimes \Q) \times (H\otimes \Q) \to \Q$ is the extension of $f$ to rational coefficients
and where $\widehat{u},\widehat{v}$ are antecedents of $u_\Q,v_\Q: H\otimes \Q \to \Q$ by the adjoint
$\widehat{f_\Q}: H\otimes \Q \to \Hom(H\otimes \Q,\Q)$. 
It can be checked that $\lambda_f$ is non-degenerate. 
Moreover, Wall proved in \cite{Wall} 
that any symmetric non-degenerate bilinear pairing on a finite abelian group arises in this way.
Thus, there exists a  symmetric  bilinear form $f:H \times H \to \Z$
on a finitely generated free abelian group $H$ such that $(G,\lambda)\simeq (G_f,\lambda_f)$.

Next, by attaching handles of index $2$  to a $4$-dimensional ball, 
one can construct a compact oriented $4$-manifold $W$ such that $H_2(W) \simeq H$
and whose homological intersection pairing on $H_2(W)$ corresponds to $-f$.
Then, one can prove that the $3$-manifold $M:=\partial W$ is such that $H_1(M) \simeq  \Coker(\widehat{f})$ 
and that its linking pairing corresponds to the bilinear form $\lambda_f$.
\end{proof}

\begin{remark}
\label{rem:surgery}
The above proof gives a way to compute the linking pairing $\lambda_M$ of a $3$-manifold $M$,
if $M$ comes as the boundary of a $4$-manifold $W$ obtained by attaching handles of index $2$ to a $4$-ball.
This amounts to obtain $M$ by surgery in $S^3$ along an embedded framed link $L \subset S^3$.
A theorem by Lickorish \cite{Lickorish} and Wallace \cite{Wallace} asserts that any $3$-manifold $M$ has such  a ``surgery presentation''.
\end{remark}

\begin{remark}
The set of  isomorphism classes of non-degenerate symmetric bilinear pairings on finite abelian groups, 
equipped with the direct sum $\oplus$, is an abelian monoid. 
Generators and relations for this monoid are known by works of Wall \cite{Wall} and Kawauchi--Kojima \cite{KK}.
Thus, in contrast with cohomology rings, linking pairings are very well-understood from an algebraic viewpoint.
\end{remark}

In the case of rational homology $3$-spheres, 
the linking pairing is determined by the maximal abelian torsion.

\begin{theorem}[Turaev \cite{Turaev_spin_c}]
\label{th:torsion_linking_pairing}
Let $M$ be a $3$-manifold with $\beta_1(M)=0$, and let $H:=H_1(M)$.
Then, we have
$$
\forall g,g'\in H, \quad
\tau(M) \cdot (g-1) \cdot (g'-1) = -\lambda_M(g,g') \cdot \Sigma_H \ \in \Q[H]/\Z[H]
$$
where we denote $\Sigma_H := \sum_{h\in H} h$.
\end{theorem}

\noindent
This can be proved starting from Lemma \ref{lem:torsion_Heegaard_splitting} -- see \cite[\S X.2]{Turaev_book_dim_3}.

\begin{example}
For $M:=L_{p,q}$, we deduce from (\ref{eq:maximal_lens}) that
$$
- \lambda_M(T,rT) \cdot  \Sigma_H  = \varepsilon_p \cdot (1-\Sigma_H/p)^2 
= \varepsilon_p \cdot (1-\Sigma_H/p) = -\varepsilon_p \cdot  \Sigma_H/p  \ \in \frac{\Q[H]}{\Z[H]}
$$
where $r$ is the inverse of $q$ mod $p$. We deduce that
$\lambda_M(T,rT) = \varepsilon_p/p$ 
or, equivalently, that $\lambda_M(T,T) = \varepsilon_p q/p \ \in \Q/\Z$.
\end{example}

\begin{exercice}
\label{ex:homotopy_lens}
Compute the linking pairing of $L_{p,q}$ from (\ref{eq:linking_pairing}), and
prove the easy part (``$\Rightarrow$'') of Whitehead's homotopy classification\footnote{A proof of the converse
``$\Leftarrow$'' can be found  in \cite[\S VII.11]{Bredon}. 
It needs the description of lens spaces as quotients of $S^3$ proposed in Exercice \ref{ex:quotient_sphere}.}:\vspace{-0.1cm}
\begin{quote}
{\it Two lens spaces $L_{p,q}$ and $L_{p',q'}$ have the same oriented homotopy type if,
and only if, $p=p'$ and $q'q \in \Z_p$ is the square of an invertible element.}
\end{quote} 
\end{exercice}

\begin{exercice}
\label{ex:counter-examples}
Using the homotopy classification of lens spaces (Exercice \ref{ex:homotopy_lens})
and their topological classification (Theorem \ref{th:lens}), show that:
(a) $3$-manifolds with the same $\pi_1$ do not have necessarily the same homotopy type,
(b) $3$-manifolds with the same homotopy type are not necessarily homeomorphic.
\end{exercice}

\subsection{The abelian homotopy type of a $3$-manifold}

We conclude this section by telling ``how much'' of the homotopy type the
linking pairing and the triple-cup product forms do detect.
According to Exercice \ref{ex:counter-examples}, the fundamental group of a $3$-manifold $M$
is not enough to determine its homotopy type. Another oriented-homotopy invariant  is 
$$
\mu(M) := f_*([M]) \ \in H_3\left(\pi_1(M)\right)
$$
where $f:M \to \K(\pi_1(M),1)$ is a map to the Eilenberg-MacLane space of $\pi_1(M)$
which induces an isomorphism at the level of $\pi_1$.
By the general theory of Eilenberg-MacLane spaces (see \cite{Bredon} for instance),
the map $f$ is unique up to homotopy, so that the homology class $\mu(M)$ is well-defined.

\begin{theorem}[Thomas \cite{Thomas}, Swarup \cite{Swarup}]
\label{th:Thomas-Swarup}
Two $3$-manifolds $M$ and $M'$ have the same oriented homotopy type if, and only if,
there exists an isomorphism $\psi: \pi_1(M) \to \pi_1(M')$ such that 
$\psi_*:H_3\left(\pi_1(M)\right) \to H_3\left(\pi_1(M')\right)$ sends $\mu(M)$ to $\mu(M')$.
\end{theorem}

\noindent
The implication ``$\Rightarrow$'' is easily checked.
See \cite{Swarup} for the proof of  ``$\Leftarrow$''.

Theorem \ref{th:Thomas-Swarup} suggests to approximate the oriented homotopy type of a $3$-manifold
in the following way.

\begin{definition}
The \emph{abelian oriented homotopy type} of a $3$-manifold $M$ is the homology class
$$
\mu_{\operatorname{ab}}(M) := f_*([M]) \ \in H_3\left(H_1(M)\right)
$$
where $f:M \to \K(H_1(M),1)$ induces the usual map $\pi_1(M) \to H_1(M)$ at the level of $\pi_1$.
\end{definition}

It turns out that the abelian oriented homotopy type is characterized 
by the two invariants that have been presented in this section, 
namely the cohomology rings and the linking pairing.

\begin{theorem}[Cochran--Gerges--Orr \cite{CGO}]
Let $M$ and $M'$ be $3$-manifolds. An isomorphism $\psi:H_1(M) \to H_1(M')$ satisfies
$\psi_*\left(\mu_{\operatorname{ab}}(M)\right)=\mu_{\operatorname{ab}}(M')$ if, and only if, 
it makes $\lambda_{M'}$ correspond to $\lambda_{M}$ 
and $u_{M}^{(r)}$ correspond to $u_{M'}^{(r)}$ for all $r\geq 0$.
\end{theorem}

\noindent
The implication ``$\Rightarrow$'' is easily checked from the fact that the linking pairing
and the triple-cup product forms are defined by (co)homology operations, 
which also exist in the category of groups.
The converse ``$\Leftarrow$'' is proved after a careful analysis of the third homology group
of a finitely generated abelian group -- see \cite{CGO}.

\begin{remark}
There are redundancies among the invariants $\lambda_M$ and $u_M^{(0)},$ $u_M^{(2)}, u_M^{(3)},\dots$.
Obviously, there are relations among the forms $u_M^{(r)}$ 
which are induced by the various homomorphisms $\Z_r \to \Z_{r'}$.
A less obvious relation is between the pairing $\lambda_M$ 
and the form $u_M^{(r)}$ for $r$ even \cite{Turaev_cohomology}. 
\end{remark}

\section{Refinement of the abelian Reidemeister torsion}

\label{sec:structures}

The maximal abelian torsion $\tau(M)$ of a $3$-manifold $M$ has been defined in \S \ref{sec:torsion_dim_3} 
up to multiplication by some element of $H_1(M)$. 
Following Turaev \cite{Turaev_Euler},
we explain in this section how this indeterminacy can be removed by choosing an ``Euler structure'' on $M$.
We also give an algebraic description of Euler structures, 
and we use them to state  some polynomial properties of the maximal abelian torsion.

\subsection{Combinatorial Euler structures}

Let $X$ be a finite connected CW-complex whose Euler characteristic $\chi(X)$ is zero.
We denote by $E$ the set of cells of $X$.

\begin{definition}[Turaev \cite{Turaev_Euler}]
An \emph{Euler chain} in $X$ is a singular $1$-chain $p$ on $X$ 
with boundary 
$$
\partial p = \sum_{e \in E} (-1)^{\dim(e)} \cdot c_e
$$
where $c_e$ denotes the center of the cell $e$.
Two Euler chains $p$ and $p'$ are \emph{homologous} if $p-p'$ is the boundary of a singular $2$-chain.
An \emph{Euler structure} of $X$ is a homology class of Euler chains.
\end{definition}
 
\noindent
The set of Euler structures on $X$ is denoted by
$$
\Eul(X).
$$
This  is an $H_1(X)$-affine space. In other words, 
the abelian group $H_1(X)$ acts freely and transitively on the set $\Eul(X)$:
$$
\forall x =[c] \in H_1(X), \ \forall \xi =[p] \in \Eul(X), \quad
\xi + \vec{x} := [p+c] \in \Eul(X).
$$

Euler structures are used as ``instructions to lift cells'', as we shall now see. 
As before, we denote by $\widehat{X}$  the maximal abelian cover of $X$, whose group of covering transformations is identified with $H_1(X)$. 
The cell decomposition of $X$ lifts to a cell decomposition of the space $\widehat{X}$. 

\begin{definition}[Turaev \cite{Turaev_Euler}]
A family $\widehat{E}$ of cells of $\widehat{X}$ is \emph{fundamental}
when each cell $e$ of $X$ has a unique lift $\widehat{e}$ in $\widehat{E}$.
Two  fundamental families of cells $\widehat{E}$ and $\widehat{E}'$ are \emph{equivalent} when the alternate sum
$$
\sum_{e\in E}
(-1)^{\dim(e)} \cdot \overrightarrow{\widehat{e} \widehat{e}' } \ \in H_1(X)
$$
vanishes. Here $ \overrightarrow{\widehat{e} \widehat{e}' } \in H_1(X)$ 
denotes the covering transformation of $\widehat{X}$ needed to move the cell $\widehat{e}$ to $\widehat{e}'$. 
\end{definition}

\noindent
If considered up to equivalence, fundamental families of cells form a set 
$$
\Eul^{\wedge}(X)
$$ 
which, again,  is an $H_1(X)$-affine space.

There is an $H_1(X)$-equivariant bijection between $\Eul^\wedge(X)$ and $\Eul(X)$, defined as follows.
Given a fundamental family of cells $\widehat{E}$, connect by an oriented path 
the center of each cell $\widehat{e}\in \widehat{E}$ to a single point in $\widehat{X}$:
this path goes from $\widehat{e}$ to the single point if $\dim(\sigma)$ is odd, and vice-versa if $\dim(\sigma)$ is even. 
The image  of this $1$-chain in $X$ is an Euler chain (shaped like a spider).
In the sequel, this identification between $\Eul^\wedge(X)$ and $\Eul(X)$ will be implicit.

\begin{remark}
The notion of Euler structure, which we have defined for CW-complexes, is combinatorial by essence.
Nonetheless we shall see in the next subsections
that, for $3$-manifolds, Euler structures also have a \emph{topological} existence.
\end{remark}

\subsection{The Reidemeister--Turaev torsion}

Let us start with a finite connected CW-complex $X$ such that $\chi(X)=0$.
Assume that $X$ is equipped with a homological orientation $\omega$ and an Euler structure $\xi$.
Given a ring homomorphism $\varphi: \Z[H_1(X)] \to \F$ with values in a commutative field $\F$,
we refine the definition of $\tau^\varphi(X,\omega)$ given in \S \ref{subsec:torsion_CW} to
$$
\tau^\varphi(X,\xi,\omega) := 
\sgn \tau\big(C(X;\R),  oo, w \big) 
\cdot \tau\big(C^\varphi(X), 1\otimes \widehat{E}_{oo}\big) \ \in \F.
$$
Here, $w$ is a basis of $H_*(X;\R)$ representing $\omega$ 
and $\widehat{E}$ is a fundamental family of cells representing $\xi$.
Again, we agree that $\tau^\varphi(X,\xi,\omega):=0$ when $H^\varphi_*(X)\neq 0$.
This refined torsion behaves well with respect to  the affine action of $H_1(X)$ on $\Eul(X)$:
$$
\forall h \in H_1(X), \quad \tau^\varphi\left(X,\xi+\vec{h},\omega\right) = \varphi(h) \cdot \tau^\varphi(X,\xi,\omega) \ \in \F.
$$
The maximal abelian torsion introduced in \S \ref{subsec:maximal} can also be refined to 
$$
\tau(X,\xi,\omega) \in Q(\Z[H_1(X)]).
$$
This refined maximal abelian torsion is $H_1(X)$-equivariant:
$$
\forall h \in H_1(X), \quad \tau\left(X,\xi+\vec{h},\omega\right) = h \cdot \tau(X,\xi,\omega) \ \in Q(\Z[H_1(X)]).
$$

Assume now that $X'\leq X$ is a cellular subdivision. Then there is a ``subdivision operator''
$$
\Omega_{X,X'}:\Eul(X) \longrightarrow \Eul(X'), \quad [p] \longmapsto 
\left[p+ \sum_{e'\in E'} (-1)^{\dim(e')} \cdot \gamma_{e'}\right] 
$$
where $\gamma_{e'}$ is a path contained in the unique open cell  $e$ of $X$ in which
$e'$ sits, and $\gamma_{e'}$ connects the center of $e$ to the center of $e'$. 
This operator respects the hierarchy of CW-complexes with respect to subdivisions, in the sense that
$$
\forall X'' \leq X' \leq X, \quad \Omega_{X',X''} \circ \Omega_{X,X'} = \Omega_{X,X''},
$$
and it preserves the refined abelian Reidemeister torsion, in the sense that
$$
\forall \xi \in \Eul(X), \quad \tau^\varphi(X',\Omega_{X,X'}(\xi), \omega) = \tau^\varphi(X,\xi,\omega) \ \in \F.
$$
Thus, using triangulations, Turaev proves that the notions of Euler structure 
and refined abelian Reidemeister torsion extend to polyhedra \cite{Turaev_Euler}.

We now come back to $3$-manifolds. We deduce that the refined abelian Reidemeister torsion 
defines a \emph{topological} invariant of pairs ($3$-manifold, Euler structure).
To justify the topological invariance of the Reidemeister torsion,
we proceed as in \S \ref{subsec:torsion_dim_3},
\ie we use triangulations to present $3$-manifolds (Theorem \ref{th:triangulation}) 
and we appeal to the Hauptvermutung (Theorem \ref{th:hauptvermutung}).
To justify that the set of Euler structures is a topological invariant in dimension $3$,
we use the fact that any two piecewise-linear homeomorphisms between polyhedra 
act in the same way on Euler structures if they are homotopic as continuous maps \cite{Turaev_Euler}
and, again, we appeal to Theorem \ref{th:hauptvermutung}.
Thus, we obtain the following refinement of Definition \ref{def:maximal_dim_3}.

\begin{definition} \label{def:RT_3}
The set of \emph{combinatorial Euler structures} of a $3$-manifold $M$ is the $H_1(M)$-affine space
$$
\Eul_c(M) := \Eul(X)
$$
where $X$ is a cell decomposition of $M$. The \emph{Reidemeister--Turaev torsion} of $M$ equipped with $\xi\in \Eul_c(M)$ is
$$
\tau(M,\xi) := \tau(X,\xi,\omega_M) \ \in  Q(\Z[H_1(M)])
$$
where the identification between the singular homology of $M$ and the cellular homology of $X$ is implicit.
\end{definition}

\subsection{Geometric Euler structures}

Let $M$ be a $3$-manifold. 
The Euler structures that we have defined so far for $M$ have been called ``combinatorial''
because their definition makes reference to cell decompositions of $M$.
We shall now give a more geometric description of the set $\Eul_c(M)$.
For this, we need to choose a smooth structure on $M$ rather than a cell decomposition.

\begin{definition}[Turaev \cite{Turaev_Euler}]
A \emph{geometric Euler structure} on $M$ is a non-singular (\ie nowhere-vanishing) tangent vector field, 
up to homotopy on $M$ deprived of an open ball.
\end{definition}

\noindent
The Poincar\'e--Hopf theorem shows that non-singular tangent vector fields on $M$ do exist since $\chi(M)=0$.
We denote by 
$$
\Eul_g(M)
$$ 
the set of geometric Euler structures. Their  parameterization is easily obtained from obstruction theory (see, for instance, \cite{Steenrod}).
Indeed, a non-singular tangent vector field  on $M$ is a section of  $T_{\neq 0}M$,
the non-zero tangent bundle of  $M$. This is a fiber bundle with fiber $\R^3\setminus \{0\}\simeq S^2$,
whose first non-trivial homotopy group is $\pi_2(S^2) \simeq \Z$. Thus, the primary obstruction to find a homotopy
between two sections of $T_{\neq 0}M$  lives\footnote{ 
The coefficients are not twisted since the group of  $T_{\neq 0}M$ is 
$\operatorname{GL}_+(3;\R) \simeq \operatorname{SO}(3)$ which is connected,
and the isomorphism comes from Poincar\'e duality.} in
$$
H^2(M; \pi_2(S^2)) \simeq H_1(M;\Z).
$$
This is the obstruction to construct a homotopy on the $2$-skeleton (for any cell decomposition of $M$)
or, equivalently, on $M$ deprived of an open ball. We deduce that $\Eul_g(M)$ is an $H_1(M)$-affine space.

Turaev proved that two diffeomorphisms between smooth $3$-manifolds
act in the same way on geometric Euler structures if they are homotopic as continuous maps.
This fact and Theorem \ref{th:smooth} imply that the set $\Eul_g(M)$ is a topological invariant of $M$. In fact, we have the following.

\begin{theorem}[Turaev \cite{Turaev_Euler}]
For any smooth $3$-manifold $M$,
there is a canonical and $H_1(M)$-equivariant  bijection between $\Eul_c(M)$ and $\Eul_g(M)$.
\end{theorem}

\noindent
By virtue of this result, we shall now identify the sets $\Eul_c(M)$ and $\Eul_g(M)$, 
which we simply denote by $\Eul(M)$.

\begin{proof}[Sketch of the proof]
Turaev defines in  \cite{Turaev_Euler} a map  $\Eul_c(M) \to \Eul_g(M)$ by working with a smooth triangulation of $M$, 
and he shows it to be $H_1(M)$-equivariant. 
Here, we shall describe Turaev's map following the Morse-theoretic approach of Hutchings \& Lee \cite{HL}.

Thus, we consider a Morse function  $f:M \to \R$ as well as a riemannian metric on $M$.
If this metric is appropriately choosen with respect to $f$ (i.e. if $f$ satisfies the ``Smale condition''), 
then $f$ defines a cell decomposition $X_f$ of $M$ (namely the ``Thom--Smale cell decomposition'').
The $i$-dimensional cells of $X_f$ are the descending manifolds from index $i$ critical points of $f$.
Let $\xi \in \Eul_c(M)= \Eul(X_f)$ be represented by a Euler chain $p$,
which (without loss of generality) we assume to be contained in a $3$-ball $B_p \subset M$.
By definition of $\Eul(X_f)$, we have
\begin{equation}
\label{eq:Euler_chain}
\partial p = \sum_{c:\ \hbox{\scriptsize critical point of } f} (-1)^{\hbox{\scriptsize index of }c} \cdot c.
\end{equation}
Let also $\nabla f$ be the gradient field of $f$ with respect to the riemannian metric. 
It is non-singular except at each critical point $c$ of $f$, where its index is $(-1)^{\hbox{\scriptsize index of }c}$.
Thus, all critical points of $\nabla f$ are contained  in $B_p$ and since the $0$-chain (\ref{eq:Euler_chain}) augments to $\chi(M)=0$,
there is a non-singular vector field $v_p$ on $M$ which coincides with $\nabla f$ outside $B_p$. 
Then, we associate to $\xi=[p]$ the geometric Euler structure represented by $v_p$.
\end{proof}

\begin{remark}
\label{rem:_refined_torsion_Heegaard_splitting}
Let $X$ be a cell decomposition of  $M$ which comes from a Heegaard splitting 
(as explained during the proof of Lemma \ref{lem:torsion_Heegaard_splitting}).
We can find a Morse function $f:M \to \R$ (and a riemannian metric on $M$ with the Smale condition  satisfied)
such that the Thom--Smale cell decomposition $X_f$ coincides with $X$. 
Then, formula (\ref{eq:torsion_Heegaard_splitting}) can be refined to take into account
a geometric Euler structure obtained from the desingularization of  $\nabla f$ in a ball $B$ 
that contains all the critical points of $f$. We refer to \cite{Massuyeau}.
\end{remark}

Given a smooth $3$-manifold $M$, one can wonder how the group $\Diff_+(M)$
of orientation-preserving self-diffeomorphisms of $M$ acts on the set $\Eul_g(M)$ and, in particular,
one can ask for the number of orbits. For lens spaces, the Reidemeister--Turaev torsion gives the answer.

\begin{theorem}[Turaev \cite{Turaev_Euler}]
The number of orbits for the action of $\Diff_+(L_{p,q})$ on  $\Eul_g(L_{p,q})$ is
\begin{itemize}
\item[$\centerdot$] $\left[p/2\right] +1$, if $q^2 \neq 1$ or $q=\pm1$,
\item[$\centerdot$] $p/2 - b(p,q)/4+c(p,q)/2$, if $q^2=1$ and $q\neq \pm 1$.
\end{itemize}
Here, for $x\in \Q$, $[x]$ denotes the greatest integer less or equal than $x$, $b(p,q)$ 
is the number of $i\in \Z_p$ for which $i$, $q+1-i$ and $qi$ are pairwise different, and 
$c(p,q)$ is the number of $i\in \Z_p$ such that $i=q+1-i=qi$.
\end{theorem}

\begin{proof} 
Let $\zeta$ be a primitive $p$-th root of unity, and let
$\varphi: \Z[H_1(L_{p,q})]\to \C$ be the ring homomorphism defined by $\varphi(T):= \zeta$
(where $T$ is the preferred generator). 
An application of Lemma \ref{lem:torsion_lens} gives
$$
\forall \xi \in \Eul_g(L_{p,q}), \quad
\tau^\varphi(L_{p,q},\xi) = \varepsilon_p \cdot \zeta^{i(\xi)} \cdot (\zeta-1)^{-1} \cdot(\zeta^r -1)^{-1}
\ \in \C \setminus \{0\}
$$
for some unique $i(\xi) \in \Z_p$. Moreover, the map $i:\Eul_g(L_{p,q}) \to \Z_p$
is bijective, since it is affine over the homomorphism $H_1(L_{p,q}) \to \Z_p$ defined by $T \mapsto 1$.
Thus, we can study the action of $\Diff_+(L_{p,q})$ on  $\Eul_g(L_{p,q})$ by means of this refined torsion.

Let $h: L_{p,q} \to L_{p,q}$ be an orientation-preserving diffeomorphism, and let $k\in {\Z_p}^\times$
be such that $h_*(T)= k\cdot T \in H_1(L_{p,q})$. 
The identity $\tau^{\varphi h_*} (L_{p,q},h(\xi))$ $=\tau^\varphi(L_{p,q},\xi)$ writes
\begin{equation}
\label{eq:zeta_bis}
\zeta^{k\cdot i(h(\xi))} \cdot (\zeta^k-1)^{-1} \cdot(\zeta^{kr} -1)^{-1} 
= \zeta^{i(\xi)} \cdot (\zeta-1)^{-1} \cdot(\zeta^r -1)^{-1}.
\end{equation}
Similarly to the proof of Theorem \ref{th:lens} (\S \ref{subsec:torsion_lens_spaces}), we deduce from Franz's Lemma \ref{lem:Franz} that
\begin{equation}
\label{eq:quadruplets_bis}
\{1,-1,r,-r\} = \{k,-k,kr,-kr\}
\end{equation}
and, in particular, that $k=1,-1,r$ or $-r$. 

The value $k=1$ is obviously realizable by an orientation-preserving self-diffeomorphism $h$ 
(for instance, the identity). Then, (\ref{eq:zeta_bis}) implies that $i(h(\xi))=i(\xi)$.

The value $k=-1$ is also realizable: we can take the map $h$
defined by $(z_1,z_2) \mapsto (\overline{z_1},\overline{z_2})$ 
on each solid torus $D^2\times S^1$ of the Heegaard splitting of $L_{p,q}$. 
Then, (\ref{eq:zeta_bis}) implies that $i(h(\xi)) = r+1 - i(\xi)$.

Finally, the value $k=r$ is realizable if and only $k=-r$ is realizable.
Assume that $k=r$ is realizable.
Then, (\ref{eq:quadruplets_bis}) implies that $r^2=\pm 1$. 
Working with any $p$-th root of unity and passing to the group ring $\C[H_1(L_{p,q})]$
as we did in the proof of Theorem \ref{th:lens} (\S \ref{subsec:torsion_lens_spaces}), we see that the value $r^2 = -1$ is impossible for $p>2$.
So, we must have $r^2=1$ or, equivalently, $q^2=1$ 
and, indeed, there exists in this situation an orientation-preserving self-diffeomorphism $h$ of $L_{p,q}$ for which $k=r$
($h$ exchanges the two solid tori in the Heegaard splitting). 
Then, (\ref{eq:zeta_bis}) implies that $r\cdot i(h(\xi))= i(\xi)$ or, equivalently, $i(h(\xi))=  q\cdot i(\xi)$. 

The conclusion easily follows from the above analysis.
\end{proof}

\begin{remark}
As observed by Turaev,  geometric Euler structures are equivalent in dimension $3$
to $\operatorname{Spin}^c$-structures \cite{Turaev_spin_c}, 
and this is the starting point of connections between the Reidemeister--Turaev torsion and other invariants.
Thus, it follows from works of Meng \& Taubes \cite{MT}, Turaev \cite{Turaev_SW} and Nicolaescu \cite{Nicolaescu}  
that the Reidemeister--Turaev torsion  is essentially equivalent  to the Seiberg--Witten invariant 
(in the presence of the Casson--Walker invariant if $\beta_1(M)=0$).
Besides, the Reidemeister--Turaev torsion appears as the Euler characteristic of 
Ozsv\'ath \& Szab\'o's Heegaard Floer homology $HF^\pm$ when $\beta_1(M)>0$ \cite{OS}. 
\end{remark}

\begin{exercice}
\label{ex:Chern_class}
Let $M$ be a smooth $3$-manifold. Any $\xi \in \Eul_g(M)$ has an \emph{opposite}\footnote{
... also called \emph{inverse} and denoted  $\xi^{-1}$ if the abelian group $H_1(M)$ is denoted multiplicatively.}
$-\xi \in \Eul_g(M)$ defined by $-\xi:=[-v]$ if $\xi$ is represented by the non-singular vector field $v$.
Define the \emph{Chern class} of $\xi$ as the only $c(\xi) \in H^2(M)\simeq H_1(M)$ such that $(-\xi) + \overrightarrow{c(\xi)} = \xi$.
Show that $c(\xi)$ is the obstruction to find a non-singular vector field on $M$ linearly independent with $v$.
\end{exercice}

\subsection{Linking quadratic functions}

Let $M$ be a $3$-manifold. We shall now give an algebraic description of the set $\Eul(M)$,
from which it will be apparent that Euler structures of dimension $3$ exist in the topological category.
(See the discussion preceding Definition  \ref{def:RT_3}.)

For this, we need a slight modification of the linking pairing $\lambda_M$ (see \S \ref{subsec:linking}).
Let $B: H_2(M;\Q/\Z) \to \Tors H_1(M)$ be the Bockstein homomorphism 
for the short exact sequence of coefficients
\begin{equation}
\label{eq:coefficients}
\xymatrix{
0 \ar[r] & \Z \ar[r] &  \Q \ar[r] & \Q/\Z \ar[r] & 0. 
}
\end{equation}
The composition $\lambda_M \circ (B \times B)$ defines  a bilinear pairing
$$
L_M: H_2(M;\Q/\Z) \times H_2(M;\Q/\Z) \longrightarrow \Q/\Z
$$
which may be degenerate. Indeed, the long exact sequence in homology induced by (\ref{eq:coefficients})
shows that $\Ker L_M = \Ker B = H_2(M) \otimes \Q/\Z$. 

A \emph{quadratic function} of \emph{polar form} $L_M$
is a map $q: H_2(M;\Q/\Z) \to \Q/\Z$ such that
$$
\forall x,y\in H_2(M;\Q/\Z),\ q(x+y)-q(x)-q(y) = L_M(x,y) \ \in \Q/\Z.
$$
The set of quadratic functions of polar form $L_M$ is denoted by $\Quad(L_M)$.
A quadratic function $q$ is said to be \emph{homogeneous} if we have $q(nx) = n^2 q(x)$ for all $x\in H_2(M;\Q/\Z)$ and $n\in \Z$.
It is easily checked that $q$ is homogeneous if, and only if, the map $d_q:H_2(M;\Q/\Z) \to \Q/\Z$
defined by $d_q(x) := q(x) - q(-x)$ is trivial.

\begin{theorem}[See \cite{DM}]
\label{th:quadratic_functions}
For any $3$-manifold $M$, there is a canonical injection
$$
\xymatrix{
{\phi_M: \Eul(M)}\ \ar@{^{(}->}[r] & {\Quad(L_M)}, \ \xi \ar@{|->}[r] & \phi_{M,\xi}
}
$$
whose image is 
$$
\left\{ q\in \Quad(L_M) : \exists  f \in \Hom(H_2(M), \Z), f \otimes \Q/\Z = q|_{H_2(M)\otimes \Q/\Z} \right\}.
$$
Moreover, the homogeneity defect of $\phi_{M,\xi}$ is given by 
$$
\forall x\in H_2(M;\Q/\Z), \quad d_{\phi_{M,\xi}}(x) =  \langle c(\xi) , x\rangle \in \Q/\Z
$$
where $c(\xi) \in H^2(M)$ is the Chern class of $\xi$ defined in Exercice \ref{ex:Chern_class}.
\end{theorem}

\noindent
The definition of $\phi_{M,\xi}$ given in \cite{DM} is a direct generalization 
of  previous constructions by Morgan \& Sullivan \cite{MS}, Lannes \& Latour \cite{LL} and Looijenga \& Wahl \cite{LW}.
More precisely, the map $\phi_{M,\xi}$ is introduced in \cite{LW} when the Chern class $c(\xi)$ is torsion, 
in which case $\phi_{M,\xi}$ factorizes to a quadratic function $\Tors H_1(M) \to \Q/\Z$ with polar form $\lambda_M$.
The special case $c(\xi)=0$ already appears in  the works \cite{MS,LL} and 
corresponds to an Euler structure $\xi=[v]$ such that $v$ is the first vector field of a trivialization of the tangent bundle $TM$.

\begin{proof}[Sketch of  the proof of Theorem \ref{th:quadratic_functions}]
Let us define the quadratic function $\phi_{M,\xi}$ in the general case.
Our first observation is that, for every $x \in H_2(M;\Q/\Z)$, we must have
$2 \phi_{M,\xi}(x) = L_M(x,x) + d_{\phi_{M,\xi}}(x)$. Thus, we shall give a formula of the form 
$$
\phi_{M,\xi}(x) = l_x +c_x \ \in \Q/\Z
$$
where $l_x \in \Q/\Z$ is such that $2l_x= L_M(x,x)=\lambda_M\left(B(x),B(x)\right)$
and where $c_x \in \Q/\Z$ (which should be selected \emph{correlatively} to $l_x$) satisfies $2c_x = \langle c(\xi) ,x \rangle$.
To obtain such a formula,  represent the homology class $x$ in the form
$$
x=\left[S \otimes \{1/n\}\right] \in H_2(M;\Q/\Z)
$$
where $n\geq 1$ is an integer, $S\subset M$ is an oriented compact immersed surface whose boundary
goes $n$ times around an oriented knot $K \subset M$.
We can find a non-singular vector field $v$ which represents $\xi$ and is transverse to $K$.
Let $V$ be a sufficiently small regular neighborhood of $K$ in $M$, and let 
$K_v$ be the parallel of $K$ sitting on $\partial V$ and  obtained by pushing $K$ along the trajectories of $v$.
By an isotopy of $S$ (fixed on the boundary), we can ensure that $S$ is transverse to $K_v$.
Finally, let $w$ be a non-singular vector field on $V$ which is tangent to $K$ and is linearly independent with $v$, and let 
$$
c(w|v) \in H^2\left(M \setminus \int(V), \partial V\right)
$$
denote the obstruction to extend $w|_{\partial V}$ to a non-singular vector field on $M \setminus \int(V)$
linearly independent with $v|_{M \setminus \int(V)}$. With  these notations, we set
\begin{equation}
\label{eq:dim_3_formula}
\phi_{M,\xi}(x): = 
\underbrace{\left\{\frac{1}{2n}\cdot K_v \centerdot S \right\}}_{\hbox{\scriptsize our } l_x} + 
\underbrace{\left\{\frac{1}{2n}\cdot \left\langle c(w|v) , \left[S\cap (M \setminus \int(V))\right] \right\rangle 
+ \frac{1}{2}\right\}}_{\hbox{\scriptsize our } c_x}.
\end{equation}

Recall from Remark \ref{rem:surgery} that there is a formula to compute
$\lambda_M$ when $M$ is presented as the boundary of a $4$-manifold obtained 
by attaching handles of index $2$ to a $4$-ball. 
This can be refined to a $4$-dimensional formula for $\phi_{M,\xi}$, 
which can be more convenient in practice  than the $3$-dimensional formula (\ref{eq:dim_3_formula}).
In particular, it can be used to prove that the map $\phi_M$ is affine over the homomorphism
\begin{equation}
\label{eq:affine}
H^2(M) \longrightarrow \Hom\left(H_2(M;\Q/\Z),\Q/\Z\right), \ y \longmapsto \langle y, -\rangle
\end{equation}
where the affine action of $ \Hom\left(H_2(M;\Q/\Z),\Q/\Z\right)$  on $\Quad(L_M)$ is by addition of maps $H_2(M;\Q/\Z) \to \Q/\Z$.
This homomorphism corresponds by Poincar\'e duality to the map $H_1(M) \to \Hom\left(H^1(M;\Q/\Z),\Q/\Z\right)$
defined by $x\mapsto \langle -,x\rangle$ and, so, it is injective. We deduce that the map $\phi_M$ is injective.

It can also be checked that
$$
\forall x \otimes \{r\} \in H_2(M)\otimes \Q/\Z, \
\phi_{M,\xi}\left(x \otimes \{r\} \right) = \frac{\langle c(\xi), x \rangle}{2} \cdot \{r\} \ \in \Q/\Z.
$$
Using the fact that $M$ (like any $3$-manifold) has a parallelization, 
one can see that $c(\xi)$ is even. Thus, the restriction of $\phi_{M,\xi}$ to $H_2(M)\otimes \Q/\Z$
comes from $\Hom(H_2(M),\Z)$. The converse is deduced from the fact that the map $\phi_M$ is affine.
\end{proof}

We now assume that $M$ is a rational homology $3$-sphere, \ie $\beta_1(M) =0$.
We denote $H:=H_1(M)$.
In this case, the Bockstein $B: H_2(M;\Q/\Z) \to H$ is an isomorphism and Theorem \ref{th:quadratic_functions} gives a bijection
$$
\phi_M: \Eul(M) \stackrel{\simeq}{\longrightarrow} \Quad(L_M) \simeq \Quad(\lambda_M).
$$
The Reidemeister--Turaev torsion of $M$ equipped  with an Euler structure $\xi$ writes
$
\tau(M,\xi) = \sum_{h\in H} t_{M,\xi}(h) \cdot h \ \in \Q[H].
$
By Theorem \ref{th:torsion_linking_pairing}, we have
$$
-\lambda_M(h_1,h_2) =  t_{M,\xi}(h_1h_2) - t_{M,\xi}(h_1) - t_{M,\xi}(h_2) + t_{M,\xi}(1) \mod 1
$$
for all $h_1,h_2 \in H$. This equation shows that the map
$$
q_{M,\xi}: H \longrightarrow \Q/\Z, \ h \longmapsto t_{M,\xi}(1) - t_{M,\xi}(h^{-1})
$$
is a quadratic function with polar form $\lambda_M$. 

\begin{theorem}[See \cite{Nicolaescu_book, DM_torsion_mod_1}]
Let $M$ be a  $3$-manifold with $\beta_1(M)=0$. 
Then, for all $\xi \in \Eul(M)$, we have $q_{M,\xi} \circ B = \phi_{M,\xi}$.
\end{theorem}

\noindent
Nicolaescu proves this in \cite{Nicolaescu_book} (see also \cite{Nicolaescu})
using  the relation between the Reidemeister--Turaev torsion and the Seiberg--Witten invariant. 
The proof given in \cite{DM_torsion_mod_1} is purely topological:
it uses surgery presentations of $3$-manifolds and the $4$-dimensional formula for $\phi_{M,\xi}$.

\begin{exercice}
Let $M$ be a smooth $3$-manifold, 
and let $f:M \to M$ be an orientation-preserving diffeomorphism such that $f_*: H_1(M) \to H_1(M)$ is the identity.
Show that $f$ acts trivially on $\Eul_g(M)$. 
\end{exercice}

\subsection{Some polynomial properties of the Reidemeister--Turaev torsion}

\label{subsec:polynomial}
	
Let $\mathcal{M}$ be the set of $3$-manifolds, up to orientation-preserving homeomorphisms.
Recall from Definition \ref{def:surgery} that the $T^3$-surgery  is a way of modifying 
$3$-manifolds, which is modelled on the genus $3$ Heegaard splitting 
$$
T^3 = A \cup B
$$
of the $3$-dimensional torus $T^3=S^1\times S^1 \times S^1$.

\begin{definition}
\label{def:finite-type}
Let $c: \mathcal{M} \to C$ be an invariant of $3$-manifolds with values in an abelian group $C$.
We say that $c$ is a \emph{finite-type invariant} of \emph{degree} at most $d$ if, for any $3$-manifold $M \in \mathcal{M}$
and for any embeddings $a_0 :A \hookrightarrow M, \dots, a_{d}: A \hookrightarrow M$ whose images are pairwise disjoint, we have
\begin{equation}
\label{eq:finite-type}
\sum_{P\subset \{0, \dots, d \}} (-1)^{|P|} \cdot c(M_{a_P}) = 0 \ \in C.
\end{equation}
Here, $M_{a_P}$ denotes the $3$-manifold obtained from $M$ by simultaneous $T^3$-surgeries
along those embeddings $a_i: A \hookrightarrow M$ for which $i \in P$.
\end{definition}

\noindent
The notion of finite-type invariant has been introduced for homology $3$-spheres (in an equivalent way) by Ohtsuki \cite{Ohtsuki},
and it is similar to the notion of Vassiliev invariant for knots \cite{BL}: 
the $T^3$-surgery plays for $3$-manifolds the role that the crossing-change move plays for knots.
This notion is motivated by the study of quantum invariants.
The reader may consult the survey \cite{Le} for an introduction to the subject of finite-type invariants.

By analogy with a well-known characterization of polynomial functions (recalled in Exercice \ref{ex:polynomial}), 
one can think of finite-type invariants as those maps defined on $\mathcal{M}$ which behave ``polynomially'' with respect to  $T^3$-surgeries. 
We shall now see that the Reidemeister--Turaev torsion do have such polynomial properties.
But, since the Reidemeister--Turaev torsion depends both on the homology and on the Euler structures of $3$-manifolds,
we shall first refine the notion of finite-type invariant. 
Thus, we fix a finitely generated abelian group $G$ and we consider triples of the form
$$
(M,\xi, \psi),
$$
where $M$ is a $3$-manifold, $\xi \in \Eul(M)$ and $\psi: G \to H_1(M)$ is an isomorphism.
Of course, two such triples $(M_1,\xi_1,\psi_1)$ and $(M_2,\xi_2,\psi_2)$ are considered to be \emph{equivalent}
if there is an orientation-preserving homeomorphism $f:M_1 \to M_2$ which carries $\xi_1$ to $\xi_2$
and satisfies $f_* \circ \psi_1 = \psi_2$. We denote by
$$ 
\mathcal{ME}(G)
$$
the set of equivalence classes of such triples. Then, the Reidemeister--Turaev torsion can be seen as a map
$$
\tau: \mathcal{ME}(G) \longrightarrow Q(\Z[G]), \ (M,\xi, \psi) \longmapsto Q(\psi^{-1})\left( \tau(M,\xi)\right).
$$
We have seen during the proof of Theorem \ref{th:Sullivan} that a $T^3$-surgery $M \leadsto M_a$ 
induces a canonical isomorphism in homology:
$$
\Phi_a: H_1(M) \stackrel{\simeq}{\longrightarrow} H_1(M_a).
$$
Similarly, the move $M \leadsto M_a$ induces a canonical correspondence 
$$
\Omega_a: \Eul(M) \stackrel{\simeq}{\longrightarrow} \Eul(M_a)
$$
between Euler structures, which is affine over $\Phi_a$.
We refer to \cite{DM,Massuyeau} for the definition of the map $\Omega_a$ by cutting and pasting vector fields.  
Thus, the notion of $T^3$-surgery exists also
for $3$-manifolds with Euler structure and parameterized homology:
$$
(M,\xi,\psi) \leadsto (M_a,\xi_a,\psi_a) \quad
\quad \hbox{where } \xi_a := \Omega_a(\xi) \hbox{ and } \psi_a := \Phi_a \circ \psi.
$$
Therefore, we get a notion of finite-type invariant for  $3$-manifolds with Euler structure and parameterized homology:
the set $\mathcal{M}$ is replaced by $\mathcal{ME}(G)$ in Definition \ref{def:finite-type}.

\begin{theorem}[See \cite{Massuyeau}]
\label{th:polynomial}
Assume that $G$ is cyclic or has positive rank, and let $I(G)$ be the augmentation ideal  of $\Z[G]$.
Then, for any integer $d\geq 1$, the Reidemeister--Turaev torsion reduced modulo $I(G)^d$
$$
\tau(M,\xi,\psi) \ \in Q(\Z[G])/I(G)^d
$$
is a finite-type invariant of degree at most $d+1$.
\end{theorem}

\noindent
The proof uses a refinement of  Lemma \ref{lem:torsion_Heegaard_splitting} 
which takes into account Euler structures (see Remark \ref{rem:_refined_torsion_Heegaard_splitting}).
Theorem \ref{th:polynomial} generalizes the fact that the coefficients of the Alexander polynomial of knots in $S^3$
are Vassiliev invariants \cite{BL}. See \cite{Murakami} in the case of links.

\begin{exercice}
\label{ex:polynomial}
Prove that a function $c: \Q^n \to \Q$ is polynomial of degree at most $d$ if, and only if,
we have for any $m \in \Q^n$ and any $a_0,\dots,a_d \in \Q^n$
$$
\sum_{P\subset \{0, \dots, d \}} (-1)^{|P|} \cdot c(m+a_{P}) = 0 \ \in \Q
$$
where $a_P$ is the sum of those vectors $a_i$ for which $i\in P$.
\end{exercice} 

\begin{exercice}
Prove the easy part (``$\Rightarrow$'') of Matveev's theorem \cite{Matveev}:\vspace{-0.cm}
\begin{quote}
\emph{Two $3$-manifolds $M$ and $M'$  are related  by a finite sequence of $T^3$-surgeries
if, and only if, there is an isomorphism $\Phi: H_1(M) \to H_1(M')$ such that 
$\lambda_{M} = \lambda_{M'} \circ \left(\Phi |_{\Tors} \times \Phi |_{\Tors}\right)$.}
\end{quote}\vspace{-0.cm}
Show that two $3$-manifolds $M$ and $M'$ are not distinguished by finite-type invariants of degree $0$
if, and only if, there is a finite sequence of $T^3$-surgeries connecting $M$ to $M'$.
\end{exercice}

\section{Abelian Reidemeister torsion and the Thurston norm}

To conclude, we give a topological application of the abelian Reidemeister torsions:
we shall see that they can be used to define a lower bound for the ``Thurston norm.''

\subsection{The Thurston norm}

Let $M$ be a $3$-manifold with $\beta_1(M) \geq 1$. 
For each  class $ s \in H^1(M;\Z)$, one may ask for the minimal ``complexity'' of
a closed oriented surface representing the Poincar\'e dual of $s$. 
If spheres and tori are discarded, this leads to the following definition:
\begin{equation}
\label{eq:Thurston_norm}
\forall s \in H^1(M;\Z), \quad \Thurston{s} := \min_{S: \hbox{\scriptsize dual to } s } \chi_{-}(S).
\end{equation}
Here, the minimum runs over all closed oriented surfaces $S \subset M$ which are Poincar\'e dual to $s$
and which may be disconnected ($S = S_1 \sqcup \cdots \sqcup S_k$), and we denote  
$$
\chi_{-}(S) := \sum_{i=1}^k \max\{0,-\chi(S_i)\}.
$$

\begin{example}
\label{ex:norm_surface_bundle}
Let $T_f$ be the mapping torus of an orientation-preserving homeomorphism $f:\Sigma_g \to \Sigma_g$
(see \S \ref{subsec:surface_bundles}) with $g\geq 1$. We consider the cohomology class
$$
s_f \in H^1(T_f;\Z)
$$
that is Poincar\'e dual to the fiber $\Sigma_g \subset T_f$. The norm of this class is
$$
\Thurston{s_f} = 2g -2. 
$$
The inequality ``$\leq$'' is obvious since the Euler characteristic of the fiber is $2-2g\leq0$.
We shall prove the inequality ``$\geq$'' with a little bit of $3$-dimensional topology.
Let $S \subset T_f$ be a closed oriented surface which minimizes $\Thurston{s_f}$.
It can be checked -- and this is a general fact (Exercice \ref{ex:incompressible}) -- that the surface $S$ must be incompressible.
Consider the cover of $T_f$ defined by $s_f \in H^1(T_f;\Z) \simeq \Hom(\pi_1(T_f),\Z)$,
which is  $\Sigma_g \times \R$ with the obvious projection. Let $S'\subset T_f$ be the $2$-complex
obtained by connecting each component of $S$ by a path to a single point. 
The inclusion  $S' \hookrightarrow T_f$ lifts to this cover, and by composing with the  cartesian projection $\Sigma_g \times \R \to \Sigma_g$,
we get a  map  $h:S' \to \Sigma_g$. The inclusion $S' \hookrightarrow T_f$ induces an injection at the level of $\pi_1(-)$
since $S$ is incompressible. (This is an application of the loop theorem, see \cite{Jaco} for instance.)  
It follows that $h_*:\pi_1(S') \to \pi_1(\Sigma_g)$ is injective and, $h$ being of degree $1$, it is also surjective.
Therefore, $\pi_1(S')$ is isomorphic to $\pi_1(\Sigma_g)$, which implies that $S$ was connected and its  genus  is equal to $g$.
Thus, $\chi_-(S)=2g-2$ and we are done. 
\end{example}

By a ``norm'' on a real vector space $V$,
we mean a map $\Vert - \Vert : V \to \R$ which satisfies the triangle inequality, \ie
$$
\forall v_1, v_2 \in V, \quad \Vert v_1 + v_2 \Vert \leq \Vert v_1 \Vert +   \Vert v_2 \Vert,
$$
and which is homogeneous, \ie
$$
\forall r \in \R, \forall v \in V, \quad \Vert r\cdot v \Vert = |r| \cdot \Vert v \Vert.  
$$
We do not assume it to be non-degenerate: we may have $\Vert v \Vert = 0$ 
for a vector $v\neq 0$.

\begin{theorem}[Thurston \cite{Thurston_norm}]
Let $M$ be a $3$-manifold with $\beta_1(M) \geq 1$. 
The map $\Thurston{-} : H^1(M;\Z) \to \Z$ defined by (\ref{eq:Thurston_norm})
extends in a unique way to a norm $\Thurston{-} : H^1(M;\R) \to \R$. 
\end{theorem}

\begin{proof}  We follow Thurston \cite{Thurston_norm} and start with the following fact.
\begin{quote} \textbf{Claim.} Any element $s\in H_2(M;\Z)$ can be realized  by a closed
oriented surface $S \subset M$. Moreover, if $s = k \cdot s'$ is divisible by $k \in \N$,
then the surface $S$ is the union of $k$ sub-surfaces, each realizing $s'$.  
\end{quote}

\noindent
Indeed (assuming that $M$ is smooth), $s\in H_2(M) \simeq \Hom(H_1(M), \Z)$ 
can be realized by a smooth map  $f:M \to S^1=\K(\Z,1)$. Then, for any regular value $y$ of $f$,
the surface $S:= f^{-1}(y)$ with the appropriate orientation is such that $[S]=s$. If now $s=k\cdot s'$,
we can find a smooth map $f': M \to S^1$ which realizes $s'$ and satisfies $\pi_k \circ f' =f$,
where $\pi_k: S^1 \to S^1$ is the $k$-fold cover. Let $y'_1,\dots, y'_k$ be the pre-images of $y$ by $\pi_k$.
They are regular values of $f'$ so that $S$ is the union of the subsurfaces $f'^{-1}(y_1'),\dots, f'^{-1}(y'_k)$.

This fact implies that the map $\Thurston{-}: H^1(M;\Z) \to \Z$ defined by (\ref{eq:Thurston_norm}) is homogeneous.
Let us check that it also satisfies the triangle inequality. Let $s,s' \in H_2(M;\Z)$ and let $S,S' \subset M$
be closed oriented surfaces which are Poincar\'e dual to $s,s'$ respectively and satisfy $\chi_-(S) = \Thurston{s}, \chi_-(S') = \Thurston{s'}$.
We put $S$ and $S'$ in transverse positions, and we consider their intersection which consists of circles.

\begin{quote}
\textbf{Claim.} We can assume that none of the components of $S \cap S'$ bounds a disk in $S$ or in $S'$.
\end{quote}

\noindent
Otherwise,  one of these circles bounds a disk $D$, say in $S$.
We can assume that $S'$ does not meet the interior of $D$.
(If this is not the case, we jump to the innermost circle of $S' \cap D$.)
Let $S'_D$ be the surface obtained by ``decompressing'' $S'$ along $D$: 
$D$ is thickened to $D\times [-1,1] \subset M$ in such a way that 
$\left(D\times[-1,1]\right) \cap S' = \partial D \times [-1,1]$, and $S'_D$ is obtained from 
$\left(D\times[-1,1]\right) \cup S'$ by removing  $D\times ]-1,1[$.
The new surface $S'_D$ is homologous to $S'$ in $M$ and satisfies $\chi(S'_D) = \chi(S') +2$. 
Since $S'$ minimizes $\Thurston{s'}$,  the component of $S'$ which has been ``decompressed''
was either a torus (in which case, it has been transformed into a sphere) 
or was a sphere (in which case, it has been transformed into two spheres). 
Thus, we have $\chi_-(S'_D) = \chi_-(S')$ and we can replace $S'$ by $S'_D$.
This proves the claim by  induction.

Next, we consider the closed oriented surface $S''$ obtained from the desingularization of $S \cup S'$.
Here, by ``desingularization'' we mean to replace $S^1\times$(the graph of $xy=0$) by  $S^1\times$(the graph of $xy=1$)
in a way compatible with the orientations of $S$ and $S'$.
The class $[S'']=[S \cup S']\in H_2(M)$ is Poincar\'e dual to $s + s' \in H^1(M)$, and  we have $\chi(S'')=\chi(S)+\chi(S')$.
By the above claim, a sphere component of $S$ (respectively $S'$) 
can not meet $S'$ (respectively $S$) and, so, remains unchanged in $S''$.
The claim also shows that,  conversely, any sphere component of $S''$ comes from 
a sphere component of  $S$ or of $S'$. 
Thus, we have  $\chi_-(S'')=\chi_-(S)+\chi_-(S')$ and the triangle inequality is proved:
$$
\Thurston{s+s'} \leq \chi_-(S'') = \Thurston{s}+\Thurston{s'}.
$$

The map $\Thurston{-}: H^1(M;\Z) \to \Z$ extends to $H^1(M;\Q)$ in a unique way
using the rule  $\Thurston{q\cdot x}= |q| \cdot \Thurston{x}$ for all $q \in \Q$ and $x \in H^1(M;\Z)$.
The resulting map  $\Thurston{-}: H^1(M;\Q) \to \Q$ is easily checked to be a norm, 
and there is a unique way to extend it by continuity to a norm $\Thurston{-}: H^1(M;\R) \to \R$.
(Any norm in $H^1(M;\R)$ is continuous, so that we had no choice.)
\end{proof}

\begin{remark}
The kernel of the Thurston norm  is the subspace of $H^1(M;\R)$ generated 
by those elements of $H^1(M;\Z)$ which are dual to embedded tori or embedded spheres in $M$.
When  $\Thurston{-}$ is non-degenerate, its unit ball is a finite-sided polyhedron of $H^1(M;\R)$
(defined by  finitely-many linear inequalities with even integer coefficients \cite{Thurston_norm}).
\end{remark}

\begin{remark}
\label{rem:knot}
The Thurston norm generalizes the genus of a non-trivial knot $K \subset S^3$. Indeed,
the above definition of $\Thurston{-}$ also works for a $3$-manifold $M$ \emph{with boundary}. 
Instead of considering closed surfaces $S\subset M$ dual to a cohomology class $s\in H^1(M;\Z)$, 
one then considers properly embedded surfaces $S\subset M$.
In particular, we  can take $M := S^3 \setminus \hbox{N}(K)$ 
(where $\hbox{N}(K)$ is an open regular neighborhood of $K$) 
and $s \in H^1(M;\Z)$ dual to a meridian of $K$.
Then, by considering a Seifert surface which realizes the genus $g$ of $K$, we obtain $\Thurston{s} = 2g-1$.
\end{remark}

\begin{exercice}
\label{ex:incompressible}
Let $M$ be an \emph{irreducible} $3$-manifold, \ie any embedded sphere in $M$ must bound a $3$-ball.
Show that any closed oriented surface $S \subset M$ (with $S \neq S^2$)
which minimizes the Thurston norm is \emph{incompressible} --
\ie any simple closed curve in $S$ which bounds an embedded  disk in $M$ must be homotopically trivial in $S$.
\end{exercice}

\subsection{The Alexander norm and the torsion norm}

Let $M$ be a $3$-manifold with $\beta_1(M)\geq 1$. 
We define two more norms on $H^1(M;\R)$ using topological invariants.

Let $G:= H_1(M) / \Tors H_1(M)$ and let $(-)_{\R} : G \to H_1(M;\R)$ be the canonical injection.
Choose a representative in $\Z[G]$ of the Alexander polynomial of $M$,
and define a map $d_M: G \to \Z$ by the formula
$$
\Delta(M) = \sum_{g\in G} d_M(g) \cdot g \ \in \Z[G].
$$
The \emph{Alexander polytope} of $M$ is the convex hull
$$
P_A(M) := \left\langle (g_\R-g'_\R)/2\ |\ g, g' \in G, d_{M}(g) \neq 0, d_{M}(g') \neq 0 
\right\rangle_{\operatorname{convex}} 
$$
in $H_1(M;\R)$.
Observe that $P_A(M)$ is  compact (since the map $d_M$ vanishes almost everywhere) and  is symmetric in $0 \in H_1(M;\R)$.
Moreover, it does not depend on the choice of a representative for $\Delta(M) \in \Z[G]/\pm G$.

\begin{definition}[McMullen \cite{McMullen}]
The  \emph{Alexander norm}  is the map $\Alexander{-} : H^1(M;\R) \to \R$ defined by the formula
$\Alexander{s} :=  \length \big(s\left(P_A(M)\right)\big)$
or, alternatively, by the formula
$$
\Alexander{s} := \max \left\{ |\langle s,g-g' \rangle|\ \left|\ g,g' \in G, d_M(g)\neq 0, d_M(g')\neq 0\right.\right\}.
$$
\end{definition}

Let $H:= H_1(M)$ and let $(-)_\R: H \to H_1(M;\R)$ be the canonical homomorphism.
For $\xi \in \Eul(M)$, we define an  almost-everywhere zero function $t_{M,\xi}:H \to \Z$ by
$$
\sum_{h \in H} t_{M,\xi}(h) \cdot h = \left\{\begin{array}{ll}
\tau(M,\xi) & \hbox{if } \beta_1(M) \geq 2,\\
{[\tau](M,\xi)} & \hbox{if } \beta_1(M) =1,
\end{array}\right.
$$
where $[\tau](M,\xi)$ is the integral part of $\tau(M,\xi)$ which was alluded to in \S \ref{subsec:maximal_dim_3}.
The \emph{torsion polytope} of $M$ is the convex hull
$$
P_\tau(M) := \left\langle (h_\R-h'_\R)/2\ |\ h, h' \in H, t_{M,\xi}(h) \neq 0, t_{M,\xi}(h') \neq 0\right\rangle_{\operatorname{convex}}
$$
in $H_1(M;\R)$. Again, $P_\tau(M)$ is compact and symmetric in $0$. 
It is easily checked that $P_\tau(M)$ does not depend on the choice of $\xi$ (and neither on the orientation of $M$).

\begin{definition}[Turaev \cite{Turaev_book_dim_3}]
The \emph{torsion norm} is the map $\torsion{-}: H^1(M;\R) \to \R$ defined by the formula
$\torsion{s} :=  \length \big(s\left(P_\tau(M)\right)\big)$
or, alternatively, by the formula
$$
\torsion{s} := \max \left\{ |\langle s,h-h' \rangle|\ \left|\ 
h,h' \in H, t_{M,\xi}(h)\neq 0, t_{M,\xi}(h')\neq 0\right.\right\}.
$$
\end{definition}

In the case when $\beta_1(M)\geq 2$, Theorem \ref{th:Milnor-Turaev_dim_3} tells us that $P_\tau(M)$ contains $P_A(M)$, and we deduce that 
$$
\forall s \in H^1(M;\R), \ \torsion{s} \geq \Alexander{s}.
$$
In the case when $\beta_1(M)=1$,  Theorem \ref{th:Milnor-Turaev_dim_3}  shows that $P_\tau(M)$ 
contains the interval $[-p+1,p-1]$ if $P_A(M)=[-p,p]$ with $p\geq 1$, and we deduce that
$$
\forall s \in H^1(M;\R), \  \torsion{s} \geq \Alexander{s} - 2 |s|
$$
where $|-|$ is the unique norm on $H^1(M;\R)$ which takes $1$ on the generators of $H^1(M;\Z)$.
Those inequalities are strict in general\footnote{ 
See \cite[\S IV.1]{Turaev_book_dim_3} for the exact connection between $P_\tau(M)$, $P_A(M)$ 
and other polytopes associated to those twisted Alexander polynomials 
that are induced by homomorphisms $H_1(M) \to \C^*$.}.

\subsection{Comparison of norms}

\begin{theorem}[McMullen \cite{McMullen}]
\label{th:McMullen}
Let $M$ be a $3$-manifold with $\beta_1(M)\geq 1$. Then, we have the following inequality:
$$
\forall s \in H^1(M;\R), \quad \Thurston{s} \geq 
\left\{\begin{array}{ll}
\Alexander{s} & \hbox{if } \beta_1(M)\geq 2\\
\Alexander{s} - 2 |s| & \hbox{if } \beta_1(M)=1.
\end{array} \right.
$$
\end{theorem}

\noindent
Turaev proves in \cite[\S IV]{Turaev_book_dim_3} the stronger (and more concise) inequality
\begin{equation}
\label{eq:Turaev_norm}
\forall s \in H^1(M;\R), \quad \Thurston{s} \geq \torsion{s}. 
\end{equation}

\begin{proof} We mainly follow McMullen's \cite{McMullen} and Turaev's \cite{Turaev_book_dim_3} arguments.
Let us for instance assume that $\beta_1(M)\geq 2$.
We choose a representative of the Alexander polynomial of $M$
$$
\Delta(M) = \sum_{g\in G} d_M(g) \cdot g \ \in \Z[G],
$$
and we assume that it is non-trivial. 
(If $\Delta(M)=0$, then $\Alexander{s}=0$ and there is nothing to prove.)
An element $s\in H^1(M;\R)$ is  said to be ``generic'' if we have  
$$
\forall g \neq g' \in G, (d_M(g) \neq 0, d_{M}(g')\neq 0) \Longrightarrow s(g) \neq s(g').
$$
An element $s\in H^1(M;\R)$ is said to be ``primitive'' if it lives in $H^1(M;\Z)$ and if it is not divisible there.
Generic elements belonging to $H^1(M;\Q)$ are dense in $H^1(M;\R)$, 
and any element of $H^1(M;\Q)$ is a multiple of some primitive element. 
Since both $\Thurston{-}$ and $\Alexander{-}$ are continuous and homogeneous, 
we can assume that $s$ is generic and primitive.

The cohomology class $s\in H^1(M;\Z)=\Hom(H_1(M),\Z)$ extends to a ring homomorphism
$s:\Z[H_1(M)]\to \Q[t^\pm]$ defined by $s(h) :=t^{s(h)}$ for all $h\in H_1(M)$.
We are interested in the $\Q[t^\pm]$-module $H_1^s(M)$.

\begin{quote}
\textbf{Claim.} We have $\dim_\Q H_1^s(M) = \Alexander{s} +2$.
\end{quote}

\noindent
To show this, we consider the cell decomposition $X$ associated to a Heegaard splitting of $M$
(as described in the proof of Lemma \ref{lem:torsion_Heegaard_splitting}) and we denote by $Y$ its $2$-skeleton. 
Then, we have $H_1^s(M) \simeq H_1^s(X) = H_1^s(Y)$.
The same arguments as those used in the proof of Theorem \ref{th:Alexander_computation} show that
\begin{equation}
\label{eq:Deltas}
\ord \left(H_1^s(M)\right) =
\Delta_0 \left(H_1^s(M)\right) \simeq \Delta_0 \left(H_1^s(Y)\right) = \Delta_1 \left(H_1^s(Y,e^0)\right)
\end{equation}
where $e^0$ is the only $0$-cell of $Y$.
Let $\overline{Y}$ be the maximal free abelian cover of $Y$,
and let $D$ be the matrix of the boundary  $\partial_2: C_2(\overline{Y}) \to C_1(\overline{Y})$ in some appropriate basis:
this is a square matrix  with coefficients in $\Z[G]$ of size $g$, the genus of the Heegaard splitting. 
Since the cokernel of  $ \Q[t^\pm] \otimes_{\Z[G]} \partial_2 $ is $H_1^s(Y,e^0)$, we have
$$
\Delta_1 \left(H_1^s(Y,e^0)\right) = 
\gcd \big\{ \hbox{$(g-1)$-sized minors of $s(D)$}\big\} \ \in \Q[t^\pm]/\pm t^k.
$$
But, Lemma \ref{lem:torsion_Heegaard_splitting} now tells us two things. On the one hand, it gives 
$$
\forall i,j \in \{1,\dots,g\}, \quad
\tau^\mu(M) \cdot (\mu(\alpha_i^\sharp)-1) \cdot (\mu(\beta_j^\sharp)-1) =  D_{ij} \ \in Q(\Z[G])/\pm G
$$
from which we  deduced in Theorem \ref{th:Milnor-Turaev_dim_3} that $\Delta(M) = \pm \tau^\mu(M)$.
(Here $\mu: \Z[H_1(M)] \to Q(\Z[G])$ is the canonical map.) 
On the other hand, Lemma \ref{lem:torsion_Heegaard_splitting} tells us that
$$
\tau^s(M) \cdot (s(\alpha_i^\sharp)-1) \cdot (s(\beta_j^\sharp)-1)
=  s(D)_{ij} \ \in \Q(t)/\pm t^k
$$
from which we deduce that $\Delta_1\left(H_1^s(Y,e^0)\right) = \tau^s(M)\cdot(t-1)^2$.
Since $\tau^s(M) = s\left(\tau^\mu(M) \right)$, we obtain with (\ref{eq:Deltas}) that
\begin{equation}
\label{eq:order-span}
\ord\left(H_1^s(M)\right) = (t-1)^2 \cdot s(\Delta(M)).
\end{equation}
Since $\Q[t^\pm]$ is a principal ideal domain, we can compute the order of $H_1^s(M)$
from a decomposition into cyclic modules (see Example \ref{ex:PID}). We deduce that
$$ 
\dim_\Q H_1^s(M) = \span\big(\ord\left(H_1^s(M)\right)\big)
$$
where the \emph{span} of a Laurent polynomial $\sum_{i\in\Z} q_i \cdot t^i \in \Q[t^\pm]$
is the difference between $\max\{i\in \Z: q_i \neq 0\}$ and $\min\{i\in \Z: q_i \neq 0\}$.
It then follows from (\ref{eq:order-span}) and the genericity of $s$ that
$$
\dim_\Q H_1^s(M)  = \span s(\Delta(M)) +2 = \Alexander{s} +2.
$$
This proves the previous claim.

\begin{quote} \textbf{Another claim.}
We can find a \emph{connected} closed oriented surface $S$ Poincar\'e dual to $s$,  
such that $\chi_-(S) = \Thurston{s}$ and $\beta_1(S) \geq \dim_\Q H_1^s(M)$.
\end{quote}

\noindent
This will be enough to conclude. Indeed, if that surface $S$ was a sphere, 
then the $\Q$-vector space $H_1^s(M)$ would be trivial, which would contradict the first claim. So,
$$
\Thurston{s} = -\chi(S) =  \beta_1(S) -2 \geq  \dim_\Q H_1^s(M) - 2 = \Alexander{s}.
$$

To prove the second claim, we consider a closed oriented surface $S \subset M$,
Poincar\'e dual to $s$, which satisfies $\chi_-(S) = \Thurston{s}$ and with minimal $\beta_0(S)$.

We shall first prove that $\beta_0(S)=1$.
We choose one point $v_c$ in each connected component $c$ of $M\setminus S$ 
and, if a component $c$ touches a component $c'$ along a component of $S$, 
we connect $v_c$ to $v_{c'}$ by a simple path which meets that component of $S$ in a single point.
Thus, we obtain a graph $C \subset M$ which encodes the ``connectivity'' of $M \setminus S$.
Since the surface $S$ is oriented, the edges of the graph $C$ have a natural orientation. 
Using an open regular neighborhood $\operatorname{N}(S)$ of $S$,
we can construct a retraction $\pi: M \to C$ which collapses every component of $S$ 
to the center of the corresponding edge of $C$,
and every component of $M\setminus \operatorname{N}(S)$ to the corresponding vertex of $C$.
Let also $w:H_1(C) \to \Z$ be the group homomorphism which maps each oriented edge of $C$ to $1$.
Then, the composition $w \circ \pi_*: H_1(M) \to \Z$ corresponds to $s \in H^1(M;\Z) \simeq \Hom(H_1(M),\Z)$.
Let $C_w \to C$ and $M_s \to M$ be the infinite cyclic covers defined by those homomorphisms.
The map $\pi:M\to C$ lifts to $\pi_s: M_s \to C_w$, and the inclusion $C \to M$ does to.
So, $\pi_s$ induces a surjection at the level of homology, and we deduce that
$$
\dim_\Q H_1^s(M) =  \rank\ H_1(M_s) \geq \rank H_1(C_w).
$$
If $C$ had more than one loop, then $\beta_1(C_w)$ would be infinite, which would contradict the first claim.
Thus, $\beta_1(C)=1$ and $C$ is a circle to which trees may be grafted. 
Using the minimality of $\beta_0(S)$, one easily sees that $C$ does not have univalent vertices, 
nor bivalent vertices with two incident edges pointing in opposite directions.
Therefore, $C$ is an oriented circle, and because $s=w \circ \pi_*$ is assumed to be primitive, it can have only one edge.
We conclude that $S$ is connected.

The surface $S$ being connected and $s$ being not trivial, $M\setminus S$ must be connected.
Thus, the infinite cyclic cover $M_s$ is an infinite chain of copies of $M\setminus \operatorname{N}(S)$
which are glued ``top to bottom'':
$$
M_s = \cdots \cup \left(M\setminus \operatorname{N}(S)\right)_{-1} \cup
\left(M\setminus \operatorname{N}(S)\right)_0 \cup 
\left(M\setminus \operatorname{N}(S)\right)_1 \cup \cdots
$$ 
Since $H_1(M_s;\Q)$ is finite-dimensional, an application of the Mayer--Vietoris theorem shows that
it must be generated by the image of one copy of $H_1(S;\Q)$. 
Thus, we have $\dim_\Q H_1(M_s;\Q) \leq \beta_1(S)$ and the second claim is proved.
\end{proof}

\begin{example}
Let $T_f$ be the mapping torus of an orientation-preserving homeomorphism $f:\Sigma_g \to \Sigma_g$
(as defined in \S \ref{subsec:surface_bundles} with $g\geq 1$).
Let $s_f \in H^1(T_f;\Z)$ be the class Poincar\'e dual to the fiber $\Sigma_g \subset T_f$.
The above proof shows that
$$
\Thurston{s_f} = \Alexander{s_f}.
$$
The value of $ \Alexander{s_f}$ is easily deduced from Proposition \ref{prop:surface_bundle}:
$$
\Alexander{s_f} =2g-2
$$
which agrees with the discussion of Example \ref{ex:norm_surface_bundle}.
\end{example}

\begin{remark}
Theorem \ref{th:McMullen} has an analogue for $3$-manifolds \emph{with boundary} \cite{McMullen,Turaev_book_dim_3}. 
In this sense, it generalizes  a classical fact about the Alexander polynomial $\Delta(K)$
of a knot $K \subset S^3$: twice the genus of $K$ is bounded below by the span of $\Delta(K)$ --
see \cite{Rolfsen} for instance.
\end{remark}

\begin{remark}
Theorem \ref{th:McMullen} has known several recent developments.
Friedl and Kim generalize inequality (\ref{eq:Turaev_norm}) 
to \emph{twisted} Reidemeister torsions defined by group homomorphisms $\pi_1(M) \to \operatorname{GL}(\C;d)$  \cite{FK}.
Cochran \cite{Cochran} and Harvey \cite{Harvey} prove generalizations 
of Theorem \ref{th:McMullen} for non-commutative analogues of the Alexander polynomial.
(Their results are further generalized in \cite{Turaev_norm} and \cite{Friedl}.)
Another spectacular development is the work by Ozsv\'ath and Szab\'o, who proved
that the Heegaard Floer homology \emph{determines} the Thurston norm \cite{OS_genus,OS_Thurston_norm}.
This is an analogue of a property proved in 1997 by Kronheimer and Mrowka for the Seiberg--Witten monopole homology \cite{KM}.
\end{remark}

\appendix

\section{Fox's free differential calculus}

\label{app:Fox}

Fox introduced in the late 40's a kind of differential calculus for free groups \cite{Fox_free}.
This calculus is an efficient tool in combinatorial group theory as well as in low-dimensional topology.
We have used it in these notes for computations related to the Alexander polynomial and to the Reidemeister torsion.
We introduce in this appendix the basics of Fox's calculus.

\subsection{Definition of the free derivatives}

Let $G$ be a group. The \emph{augmentation} of the group ring $\Z[G]$ is the ring homomorphism 
$$
\varepsilon: \Z[G] \longrightarrow \Z, \
\sum_{g\in G} a_g \cdot g \longmapsto \sum_{g\in G} a_g
$$
which sums up the coefficients.
The kernel of the augmentation is denoted by $I(G)$, 
and is called the \emph{augmentation ideal} of $\Z[G]$.

\begin{definition}
A \emph{derivative} of the group ring $\Z[G]$ is a group homomorphism 
$d:\Z[G] \to \Z[G]$ such that 
\begin{equation}
\label{eq:derivative}
\forall a,b \in \Z[G], \quad d(a\cdot b) = d(a) \cdot \varepsilon(b) + a \cdot d(b).
\end{equation}
\end{definition}

\noindent
If follows easily from this definition that d(1)=0  and that
$$
\forall g\in G \subset \Z[G], \quad d(g^{-1}) = - g^{-1}  \cdot d(g).
$$
For example, the map $\delta:\Z[G] \to \Z[G]$ defined by
\begin{equation}
\label{eq:special_derivative}
\forall a\in  \Z[G], \quad \delta(a) := a - \varepsilon(a)
\end{equation}
is easily checked to be a derivative.

We now restrict ourselves to $G:= \Free(x)$, 
the group freely generated by the set $x=\{x_1,\dots,x_n\}$. 
An  element $a$ of the group ring $\Z[\Free(x)]$ 
can be regarded as a Laurent polynomial with integer coefficients 
in $n$ indeterminates $x_1,\dots,x_n$ that do \emph{not} commute. 
In order to emphasize this interpretation, we will sometimes denote $a\in \Z[\Free(x)]$ by $a(x)$
and the augmentation $\varepsilon(a) \in \Z$ by $a(1)$. 

\begin{theorem}[Fox]
\label{th:Fox}
For all $i\in{1,\dots,n}$, there is a unique derivative
$$
\frac{\partial \ }{\partial x_i}: \Z[\Free(x)] \longrightarrow \Z[\Free(x)]
$$
such that 
$$
\forall k=1,\dots,n,\ \frac{\partial x_k}{\partial x_i} = \delta_{i,k}.
$$
More generally, given some $h_1,\dots, h_n \in \Z[\Free(x)]$, 
there is a unique derivative $d:\Z[\Free(x)] \to \Z[\Free(x)]$ such that $d(x_k)= h_k$
for all $k=1,\dots,n$. This derivative is given by the formula
\begin{equation}
\label{eq:derivative_generators}
\forall a \in \Z[\Free(x)], \
d(a) = \sum_{k=1}^n \frac{\partial a}{\partial x_k} \cdot h_k.
\end{equation}
\end{theorem}

\noindent
The derivatives $\frac{\partial \ }{\partial x_1}, \dots, \frac{\partial \ }{\partial x_n}$
are called the \emph{free derivatives} of the free group $\Free(x)$ relative to the basis $x=\{x_1,\dots,x_n\}$.

\begin{proof}[Proof of Theorem \ref{th:Fox}]
Any element $f$ of the free group $\Free(x)$ can be written uniquely as a word  
$$
f= x_{\mu_1}^{\epsilon_1} \cdots x_{\mu_r}^{\epsilon_r} \quad 
\hbox{(where $\mu_1,\dots,\mu_r \in\{1,\dots,n\}$ and $\epsilon_1,\dots,\epsilon_r \in\{-1,+1\}$)}
$$
which is \emph{reduced} in the sense that $\mu_j=\mu_{j+1} \Rightarrow \epsilon_j \neq -\epsilon_{j+1}$.
Then, we set 
$$
\frac{\partial f}{\partial x_i} := 
\sum_{\substack{j=1,\dots,r\\ \mu_j=i}} 
\epsilon_j\cdot x_{\mu_1}^{\epsilon_1}\cdots  
x_{\mu_j}^{\frac{1}{2}(\epsilon_j-1)}.
$$
By linear extension, we get a map $\frac{\partial \ }{\partial x_i}: \Z[\Free(x)] \to \Z[\Free(x)]$ 
which is easily checked to satisfy (\ref{eq:derivative}).
Thus, $\frac{\partial \ }{\partial x_i}$ is a derivative that satisfies $\frac{\partial x_k}{\partial x_i}=\delta_{i,k}$. 

Since the ring $\Z[\Free(x)]$ is generated by $x_1, \dots, x_n$, there is at most one derivative $d$ 
such that $d(x_k)=h_k$ for all $k=1,\dots,n$. 
It is easily checked that, if $d_1$ and $d_2$ are two derivatives of $\Z[\Free(x)]$, 
then $d_1 \cdot a_1+ d_2 \cdot a_2$ is also a derivative for all $a_1,a_2 \in \Z[\Free(x)]$.
Therefore, the right-hand side of formula (\ref{eq:derivative_generators}) defines a derivative.
Since this derivative takes the values $h_1,\dots,h_n$ on $x_1,\dots, x_n$ respectively, the conclusion follows. 
\end{proof}

The derivative $\delta$ defined by (\ref{eq:special_derivative}) satisfies $\delta(x_k)=x_k-1$.
Thus, an application of (\ref{eq:derivative_generators}) to this derivative gives the identity
\begin{equation}
\label{eq:fundamental_formula}
\forall a(x) \in \Z[\Free(x)], \quad a(x) = a(1)+ 
\sum_{i=1}^n \frac{\partial a}{\partial x_i}(x) \cdot (x_i-1).
\end{equation}
This is the \emph{fundamental formula} of Fox's free differential calculus,
which can be regarded as a kind of order $1$ Taylor formula with remainder. 

Consider now a second free group,
say the group $\Free(y)$ freely generated by the set $y:=\{y_1,\dots,y_p\}$. 
For all  $b_1, \dots, b_n \in \Free(y)$, there is a unique group homomorphism
$\Free(x) \to \Free(y)$ defined by $x_i \mapsto b_i$. The image of an $a \in \Z[\Free(x)]$
by this homomorphism is denoted by $a(b_1,\dots,b_n)$.
This is compatible with our convention to denote $a \in \Z[\Free(x)]$ also by $a(x)$.

\begin{proposition}[Chain rule]
\label{prop:chain_rule}
Let $a(x) \in \Free(x)$ and let $b_1(y), \dots, b_n(y) \in \Free(y)$.
Then, for all $k=1,\dots,p$, we have 
\begin{equation}
\label{eq:chain_rule}
\frac{\partial a(b_1,\dots,b_n)}{\partial y_k} =
\sum_{j=1}^n \frac{\partial a}{\partial x_j}(b_1,\dots,b_n) \cdot 
\frac{\partial b_j}{\partial y_k} \quad \in \Z[\Free(y)].
\end{equation}
\end{proposition}

\begin{proof}
The group homomorphism $\Free(x) \to \Free(y)$ defined by $x_i \mapsto b_i$ 
transforms the fundamental formula (\ref{eq:fundamental_formula}) to
$$
a(b_1,\dots,b_n) = 1 + \sum_{j=1}^n \frac{\partial a}{\partial x_j}(b_1,\dots,b_n) \cdot (b_j-1).
$$
Next, we apply the fundamental formula to each $b_j \in \Free(y)$ to get
$$
a(b_1,\dots,b_n) = 1 + \sum_{\substack{j=1\dots,n \\ k=1,\dots,p}} 
\frac{\partial a}{\partial x_j}(b_1,\dots,b_n) \cdot \frac{\partial b_j}{\partial y_k} \cdot (y_k-1)
$$
or, equivalently,
$$
a(b_1,\dots,b_n) = 1 + \sum_{k=1}^p \left(\sum_{j=1}^n \frac{\partial a}{\partial x_j}(b_1,\dots,b_n) 
\cdot \frac{\partial b_j}{\partial y_k}  \right)  \cdot (y_k-1).
$$
By comparing this identity with the fundamental formula for $a(b_1,\dots,b_n) \in \Z[\Free(y)]$,
we exactly obtain  identity (\ref{eq:chain_rule}).
\end{proof}

\begin{exercice} Prove the following formula:
$$
\Z[\Free(x)] \ni \ \frac{\partial x_i^p}{\partial x_i}=\frac{x_i^p-1}{x_i-1} 
=\left\{\begin{array}{ll}
1+x_i+\cdots+x_i^{p-1} & \textrm{if } p\geq 1,\\
0 & \textrm{if } p=0, \\
-x_i^{p}-x_i^{p+1}-\cdots - x_i^{-1} & \textrm{if } p\leq -1.
\end{array}\right.
$$
\end{exercice}

\begin{exercice}
We can define some \emph{higher order} free derivatives by the inductive formula
$$
\forall a \in \Z[\Free(x)], \quad
\frac{\partial^{r} a}{\partial x_{i_r} \cdot \partial x_{i_{r-1}}\cdots \partial x_{i_1}}
:= \frac{\partial \ }{\partial  x_{i_r} } 
\left(\frac{\partial^{r-1} a}{\partial x_{i_{r-1}}\cdots \partial x_{i_1}}\right).
$$
Show, for all $r\geq 1$, the following order $r$ Taylor formula with remainder:
\begin{eqnarray*}
 a(x)&= & a(1) +  \sum_{i=1}^n \frac{\partial a}{\partial x_i}(1)\cdot (x_i-1) + \cdots \\
&&+  \sum_{i_1,\dots, i_{r-1}=1}^n \frac{\partial^{r-1} a}{\partial x_{i_{r-1}}\cdots \partial x_{i_1}}(1) \cdot (x_{i_{r-1}}-1)\cdots (x_{i_1}-1) \\
&& + \sum_{i_1,\dots, i_{r}=1}^n \frac{\partial^{r} a}{\partial x_{i_r} \cdots \partial x_{i_1}}(x) \cdot (x_{i_r}-1)\cdots (x_{i_1}-1). 
\end{eqnarray*} 
\end{exercice}

\subsection{The topology behind the free derivatives}

\label{subapp:topological_interpretation}

The theory of covering spaces gives a topological interpretation to free derivatives.
We consider the bouquet $X_n$ of $n$ circles
\begin{center}
{\labellist \small \hair 0pt 
\pinlabel {$x_1$} [br] at 12 86
\pinlabel {$x_2$} [lb] at 178 207
\pinlabel {$x_n$} [tl] at 386 18
\pinlabel {$\ddots$} at 253 114
\pinlabel {${\Large \star}$} at 200 24 
\endlabellist}
\includegraphics[scale=0.3]{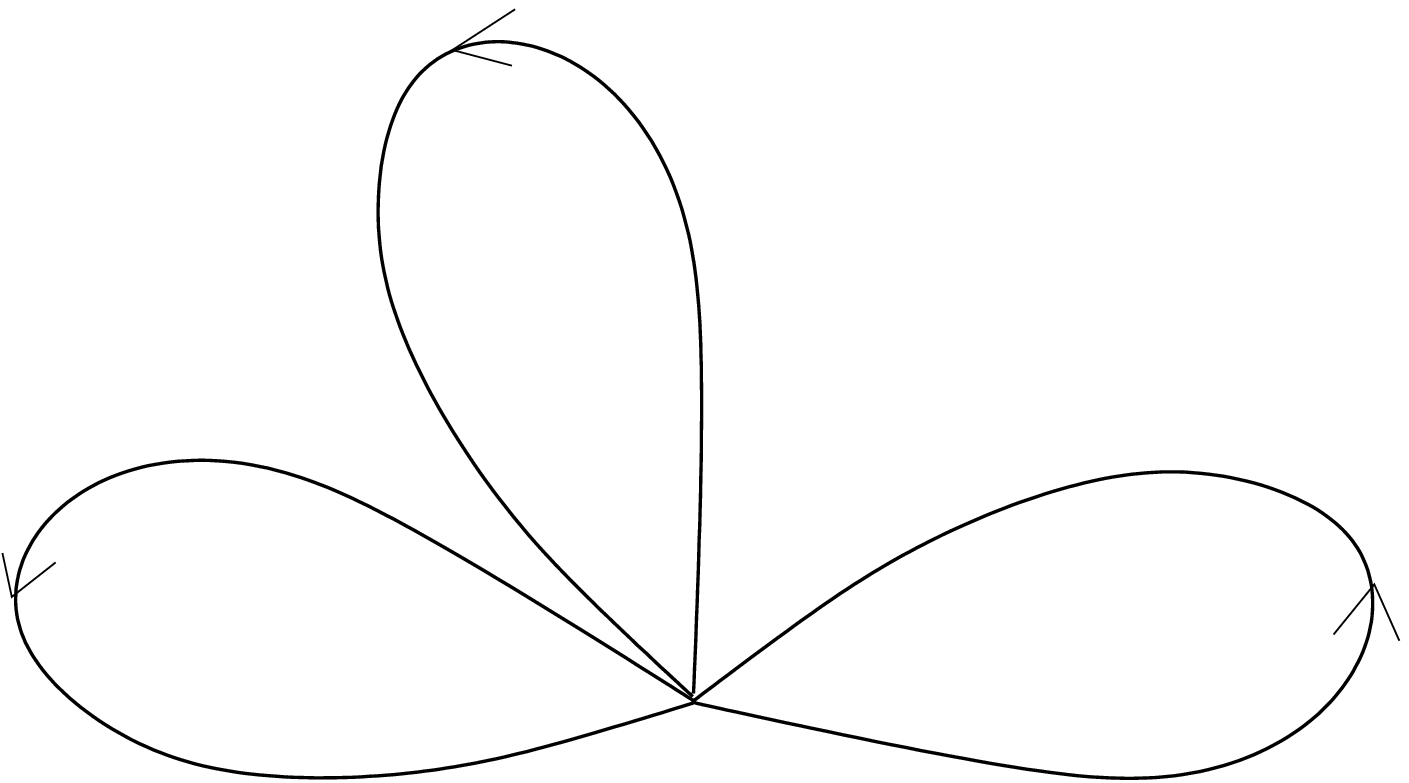}
\end{center}
Note that $\pi_1(X_n,\star)$ is the free group $\Free(x)$ on $x=\{x_1,\dots,x_n\}$.
Let also $\widetilde{X}_n$  be the infinite oriented graph with vertices indexed by $\Free(x)$
$$
V = \left\{ \left. f \cdot \widetilde{\star}\ \right|\ f \in \Free(x) \right\},
$$
with edges indexed by $n$ copies of the set $\Free(x)$
$$
E = \{f \cdot \widetilde{x}_1\ |\ f \in \Free(x) \} \cup \cdots \cup
\{f \cdot \widetilde{x}_n\ |\ f \in \Free(x) \} ,
$$
and with incidence map
$$
i=(i_0,i_1) : E \longrightarrow V \times V
$$
defined by
$$
\forall f \in \Free(x), \ \forall j\in \{1,\dots,n\}, \quad
i_0(f\cdot \widetilde{x}_j) = f \cdot \widetilde{\star}
\quad \hbox{and} \quad 
i_1(f\cdot \widetilde{x}_j) = (fx_j) \cdot \widetilde{\star}.
$$
There is a canonical map $\pi: \widetilde{X}_n \to X_n$
which sends each vertex $f \cdot \widetilde{\star}$ to $\star$
and the interior of each edge $f\cdot \widetilde{x}_j$ homeomorphically onto the interior of $x_j$.
It is easily checked that $\pi$ is the universal covering map of $X_n$.
The canonical action of $\Free(x)=\pi_1(X_n,\star) \simeq \Aut(\pi)$ on $\widetilde{X}_n$
is compatible with our notation for the vertices and edges of $\widetilde{X}_n$.
See Figure \ref{fig:infinite_tree}.

\begin{figure}[h]
\begin{center}
{\labellist \small \hair 0pt 
\pinlabel {$\widetilde{\star}$} [tr] at 206 204
\pinlabel {$\widetilde{x}_1$} [t] at 267 198
\pinlabel {$\widetilde{x}_2$} [r] at 205 269
\pinlabel {\scriptsize $x_1^2  x_2^{-1}\cdot \widetilde{\star}$} [t] at 420 180
\endlabellist}
\includegraphics[scale=0.5]{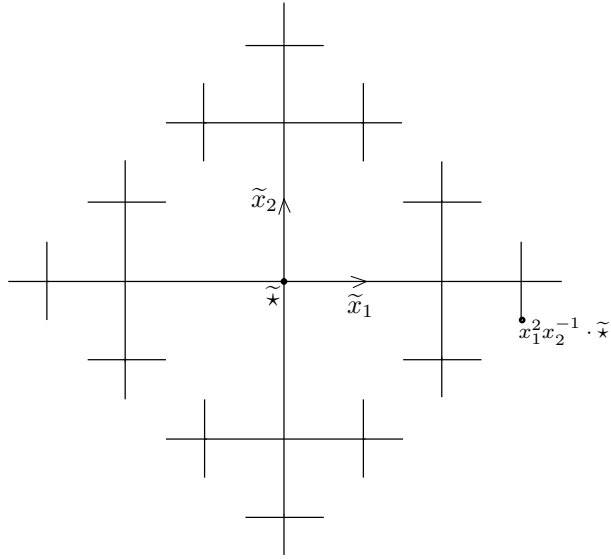}
\end{center}
\caption{The infinite tree $\widetilde{X}_2$,
where $x_1$ acts by  ``horizontal translation'' to the right and $x_2$ acts by ``vertical translation'' to the top.}
\label{fig:infinite_tree}
\end{figure}

For all $f\in \Free(x)= \pi_1(X_n,\star)$, let $\gamma$ be a closed path $\star \leadsto \star$
representing $f$ and let $\widetilde{\gamma}$ be the unique lift of $\gamma$ starting at $\widetilde{\star}$.
This is a path $\widetilde{\star} \leadsto f \cdot \widetilde{\star}$ which defines an element of $H_1(\widetilde{X}_n, \pi^{-1}(\star))$. 
Since the homotopy class of the path $\widetilde{\gamma}$ 
is determined by that of $\gamma$, this element only depends on $f$ and we denote it by
$$
\widetilde{f} \in H_1\left(\widetilde{X}_n, \pi^{-1}(\star)\right).
$$
The left action of $\Free(x)$ on $\widetilde{X}_n$ makes $H_1(\widetilde{X}_n, \pi^{-1}(\star))$ a $\Z[\Free(x)]$-module.

\begin{proposition}
\label{prop:topological_fundamental_formula}
The $\Z[\Free(x)]$-module $H_1(\widetilde{X}_n, \pi^{-1}(\star))$ is freely generated by
the lifts $\widetilde{x}_1 , \dots , \widetilde{x}_n$. Moreover, for all $f\in \Free(x)$, we have
\begin{equation}
\label{eq:topological_fundamental_formula}
\widetilde{f} = \sum_{i=1}^n \frac{\partial f}{\partial x_i}\cdot \widetilde{x_i} 
\quad \in H_1\left(\widetilde{X}_n, \pi^{-1}(\star)\right).
\end{equation}
\end{proposition}

\begin{proof}
\noindent
The abelian group $H_1(\widetilde{X}_n, \pi^{-1}(\star))$ can be described 
by means of (relative) cellular homology. Indeed, there is a canonical cell decomposition
of $\widetilde{X}_n$ induced by its graph structure: $0$-cells correspond to vertices and $1$-cells to edges. 
Thus, the first statement is obvious.

The long exact sequence for the pair $(\widetilde{X}_n, \pi^{-1}(\star))$ shows that
the connecting homomorphism $\partial_*: H_1(\widetilde{X}_n, \pi^{-1}(\star)) \to H_0(\pi^{-1}(\star))$
is injective. Thus, (\ref{eq:topological_fundamental_formula}) is equivalent to the identity
\begin{equation}
\label{eq:topological_fundamental_formula_bis}
\partial_*\left(\widetilde{f}\right) \mathop{=}^{?} 
\partial_*\left(\sum_{i=1}^n \frac{\partial f}{\partial x_i}\cdot \widetilde{x}_i\right)
\ \in H_0(\pi^{-1}(\star)).
\end{equation}
The $\Z[\Free(x)]$-module $H_0(\pi^{-1}(\star))$ is freely generated by $\widetilde{\star}$ 
and, for all $g \in \Free(x)$, we have 
$\partial_*(\widetilde{g}) = g \cdot \widetilde{\star} -  \widetilde{\star} = (g-1) \cdot  \widetilde{\star}$.
Thus, (\ref{eq:topological_fundamental_formula_bis}) is exactly the fundamental formula (\ref{eq:fundamental_formula}).
\end{proof}

\begin{exercice}
\label{ex:Schumann-Blanchfield}
Let $R$ be a normal subgroup of $\Free(x)$ and let $R':=[R,R]$ be the commutator subgroup of $R$.
Using the topological interpretation of free derivatives, show the \emph{Schumann--Blanchfield criterion}
\cite{Schumann,Blanchfield_thesis}:
$$
\forall f \in \Free(x), \quad
\left(f \in R' \right) \Longleftrightarrow 
\left(\forall i=1,\dots,n, \ \left\{\frac{\partial f}{\partial x_i}\right\}  = 0 \in \Z\left[F(x)/R\right] \right).
$$ 
\end{exercice}

\subsection{Magnus representation}

\label{subsec:Magnus}

As an application of Fox's free differential calculus, 
we  define a representation of the group of automorphisms of a free group.
For its relevance to topology, we refer to Birman's book \cite[\S 3]{Birman} and to Remark \ref{rem:Magnus_Torelli}.

Let $\phi: \Free(x) \to \Free(x)$ be a group homomorphism. 
The \emph{Jacobian matrix} of $\phi$ is the $n \times n$ matrix with coefficients in $\Z[\Free(x)]$
$$
J(\phi) := \left(\begin{array}{ccc}
\frac{\partial \phi(x_1)}{\partial x_1} & \cdots &  \frac{\partial \phi(x_n)}{\partial x_1} \\
\vdots & \ddots & \vdots \\
\frac{\partial \phi(x_1)}{\partial x_n} & \cdots &  \frac{\partial \phi(x_n)}{\partial x_n}  
\end{array}\right).
$$
The fundamental formula (\ref{eq:fundamental_formula}) shows that  a group endomorphism of $\Free(x)$
is determined by its Jacobian matrix.
For any group endomorphisms $\phi$ and $\psi$ of $\Free(x)$, 
let us see how $J(\psi \phi)$ can be derived from $J(\psi)$ and $J(\phi)$.
Proposition \ref{prop:chain_rule} shows that
$$
\frac{\partial \psi \phi (x_j)}{\partial x_i}
= \sum_{k=1}^n \psi\left(\frac{\partial \phi(x_j)}{\partial x_k}\right) \cdot \frac{\partial \psi(x_k)}{\partial x_i}.
$$
By applying the map $\overline{-} :\Z[\Free(x)] \to \Z[\Free(x)]$ defined additively
by $\overline{f}:=f^{-1}$ for all $f\in \Free(x)$, we obtain
$$
\overline{\frac{\partial \psi \phi (x_j)}{\partial x_i}}
= \sum_{k=1}^n \overline{ \frac{\partial \psi(x_k)}{\partial x_i} \cdot}\psi\left(\overline{\frac{\partial \phi(x_j)}{\partial x_k}}\right) 
$$
so that 
$$
\overline{J(\psi \phi)} = \overline{J(\psi)} \cdot \psi\left(\overline{J(\phi)}\right).
$$
It follows from this identity that, if $\phi$ is an automorphism, then $\overline{J(\phi)}$ is an invertible matrix. 
Thus, we have proved the following.

\begin{proposition}
\label{prop:Magnus}
The map $\Magnus: \Aut(\Free(x)) \to \operatorname{GL}(n;\Z[\Free(x)])$ defined by
$\Magnus(\phi):= \overline{J(\phi)}$
is a ``crossed homomorphism'' in the sense that
$$
\forall \psi, \phi \in \Aut(\Free(x)), \quad
\Magnus(\psi \phi) = \Magnus(\psi) \cdot \psi\big(\Magnus(\phi)\big).
$$
Moreover, $\Magnus$ is injective.
\end{proposition}

\noindent
The map $\Magnus$ is usually refered to as the \emph{Magnus representation} of $\Aut(\Free(x))$.

We emphasize that the map $\Magnus$ is not a ``true'' homomorphism, and that 
it takes values in a group of matrices whose ground ring is not commutative.
Those two defects can be fixed by restricting its source and by reducing its target. 
For example, we  can consider the subgroup
$$
\IAut\left(\Free(x)\right) :=
\left\{\phi \in \Aut(\Free(x)): \forall f \in \Free(x), \ \phi(f) = f \! \! \! \mod \Free(x)' \right\}
$$
where $\Free(x)' := [\Free(x),\Free(x)]$ is the commutator subgroup of $\Free(x)$.

\begin{corollary}
\label{cor:Magnus_abelian}
Let $H:= \Free(x)/\Free(x)'$ denote the abelianization of $\Free(x)$.
The map $\Magnusab: \IAut(\Free(x)) \to \operatorname{GL}(n;\Z[H])$ defined by
$$
\Magnusab(\phi):=  \overline{J(\phi)} \hbox{ reduced mod } \Free(x)'
$$
is a group homomorphism, whose kernel 
consists of those $\phi \in \Aut(\Free(x))$ such that $\phi(f)= f$ mod $\Free(x)''$ for all $f\in \Free(x)$.
\end{corollary}

\begin{proof}The second statement is an application of Exercice \ref{ex:Schumann-Blanchfield}. 
\end{proof}

\bibliographystyle{plain}
\bibliography{Pau}

\def\cprime{$'$}
\begin{thebibliography}{10}

\bibitem{Alexander}
James~W. Alexander.
\newblock Topological invariants of knots and links.
\newblock {\em Trans. Amer. Math. Soc.}, 30(2):275--306, 1928.

\bibitem{Baer1}
Reinhold Baer.
\newblock {Kurventypen auf Fl{\"a}chen.}
\newblock {\em J. f. M.}, 156:231--246, 1927.

\bibitem{Baer2}
Reinhold Baer.
\newblock {Isotopie von Kurven auf orientierbaren, geschlossenen Fl{\"a}chen
  und ihr Zusamenhang mit der topologischen Deformation der Fl{\"a}chen.}
\newblock {\em J. f. M.}, 159:101--116, 1928.

\bibitem{Birman}
Joan~S. Birman.
\newblock {\em Braids, links, and mapping class groups}.
\newblock Princeton University Press, Princeton, N.J., 1974.
\newblock Annals of Mathematics Studies, No. 82.

\bibitem{BL}
Joan~S. Birman and Xiao-Song Lin.
\newblock Knot polynomials and {V}assiliev's invariants.
\newblock {\em Invent. Math.}, 111(2):225--270, 1993.

\bibitem{Blanchfield_thesis}
Richard~C. Blanchfield.
\newblock Applications of free differential calculus to the theory of groups.
\newblock Senior thesis, Princeton University, 1949.

\bibitem{Blanchfield_intersection}
Richard~C. Blanchfield.
\newblock Intersection theory of manifolds with operators with applications to
  knot theory.
\newblock {\em Ann. of Math. (2)}, 65:340--356, 1957.

\bibitem{Bonahon}
Francis Bonahon.
\newblock Diff\'eotopies des espaces lenticulaires.
\newblock {\em Topology}, 22(3):305--314, 1983.

\bibitem{Bredon}
Glen~E. Bredon.
\newblock {\em Topology and geometry}, volume 139 of {\em Graduate Texts in
  Mathematics}.
\newblock Springer-Verlag, New York, 1993.

\bibitem{Chapman}
Thomas~A. Chapman.
\newblock Topological invariance of {W}hitehead torsion.
\newblock {\em Amer. J. Math.}, 96:488--497, 1974.

\bibitem{Cochran}
Tim~D. Cochran.
\newblock Noncommutative knot theory.
\newblock {\em Algebr. Geom. Topol.}, 4:347--398, 2004.

\bibitem{CGO}
Tim~D. Cochran, Amir Gerges, and Kent Orr.
\newblock Dehn surgery equivalence relations on 3-manifolds.
\newblock {\em Math. Proc. Cambridge Philos. Soc.}, 131(1):97--127, 2001.

\bibitem{Cohen}
Marshall~M. Cohen.
\newblock {\em A course in simple-homotopy theory}.
\newblock Springer-Verlag, New York, 1973.
\newblock Graduate Texts in Mathematics, Vol. 10.

\bibitem{deRham}
George de~Rham, Serge Maumary, and Michel~A. Kervaire.
\newblock {\em Torsion et type simple d'homotopie}.
\newblock Expos\'es faits au s\'eminaire de Topologie de l'Universit\'e de
  Lausanne. Lecture Notes in Mathematics, No. 48. Springer-Verlag, Berlin,
  1967.

\bibitem{DM_torsion_mod_1}
Florian Deloup and Gw{\'e}na{\"e}l Massuyeau.
\newblock Reidemeister-{T}uraev torsion modulo one of rational homology
  three-spheres.
\newblock {\em Geom. Topol.}, 7:773--787 (electronic), 2003.

\bibitem{DM}
Florian Deloup and Gw{\'e}na{\"e}l Massuyeau.
\newblock Quadratic functions and complex spin structures on three-manifolds.
\newblock {\em Topology}, 44(3):509--555, 2005.

\bibitem{Fox_free}
Ralph~H. Fox.
\newblock Free differential calculus. {I}. {D}erivation in the free group ring.
\newblock {\em Ann. of Math. (2)}, 57:547--560, 1953.

\bibitem{Fox_Alexander}
Ralph~H. Fox.
\newblock Free differential calculus. {V}. {T}he {A}lexander matrices
  re-examined.
\newblock {\em Ann. of Math. (2)}, 71:408--422, 1960.

\bibitem{Franz}
Wolfgang Franz.
\newblock {{\"U}ber die Torsion einer {\"U}berdeckung.}
\newblock {\em J. Reine Angew. Math.}, 173:245--254, 1935.

\bibitem{Franz_duality}
Wolfgang Franz.
\newblock {Torsionsideale, Torsionsklassen und Torsion.}
\newblock {\em J. Reine Angew. Math.}, 176:113--124, 1936.

\bibitem{Friedl}
Stefan Friedl.
\newblock Reidemeister torsion, the {T}hurston norm and {H}arvey's invariants.
\newblock {\em Pacific J. Math.}, 230(2):271--296, 2007.

\bibitem{FK}
Stefan Friedl and Taehee Kim.
\newblock The {T}hurston norm, fibered manifolds and twisted {A}lexander
  polynomials.
\newblock {\em Topology}, 45(6):929--953, 2006.

\bibitem{FV}
Stefan Friedl and Stefano Vidussi.
\newblock A survey of twisted {A}lexander polynomials.
\newblock Preprint \texttt{arXiv:0905$.$0591}, 2009.

\bibitem{FH}
Charles Frohman and Joel Hass.
\newblock Unstable minimal surfaces and {H}eegaard splittings.
\newblock {\em Invent. Math.}, 95(3):529--540, 1989.

\bibitem{Harvey}
Shelly~L. Harvey.
\newblock Higher-order polynomial invariants of 3-manifolds giving lower bounds
  for the {T}hurston norm.
\newblock {\em Topology}, 44(5):895--945, 2005.

\bibitem{Hillman}
Jonathan Hillman.
\newblock {\em Algebraic invariants of links}, volume~32 of {\em Series on
  Knots and Everything}.
\newblock World Scientific Publishing Co. Inc., River Edge, NJ, 2002.

\bibitem{HL}
Michael Hutchings and Yi-Jen Lee.
\newblock Circle-valued {M}orse theory and {R}eidemeister torsion.
\newblock {\em Geom. Topol.}, 3:369--396 (electronic), 1999.

\bibitem{Jaco}
William Jaco.
\newblock {\em Lectures on three-manifold topology}, volume~43 of {\em CBMS
  Regional Conference Series in Mathematics}.
\newblock American Mathematical Society, Providence, R.I., 1980.

\bibitem{KK}
Akio Kawauchi and Sadayoshi Kojima.
\newblock Algebraic classification of linking pairings on {$3$}-manifolds.
\newblock {\em Math. Ann.}, 253(1):29--42, 1980.

\bibitem{KM}
Peter~B. Kronheimer and Tomasz~S. Mrowka.
\newblock Scalar curvature and the {T}hurston norm.
\newblock {\em Math. Res. Lett.}, 4(6):931--937, 1997.

\bibitem{LL}
Jean Lannes and Fran{\c{c}}ois Latour.
\newblock Signature modulo {$8$} des vari\'et\'es de dimension {$4k$} dont le
  bord est stablement parall\'elis\'e.
\newblock {\em C. R. Acad. Sci. Paris S\'er. A}, 279:705--707, 1974.

\bibitem{Le}
Thang~T. Le.
\newblock Finite-type invariants of $3$-manifolds.
\newblock {\em Encyclopedia of Mathematical Physics}, pages 348--356, 2006.

\bibitem{Lickorish}
William B.~R. Lickorish.
\newblock A representation of orientable combinatorial {$3$}-manifolds.
\newblock {\em Ann. of Math. (2)}, 76:531--540, 1962.

\bibitem{LW}
Eduard Looijenga and Jonathan Wahl.
\newblock Quadratic functions and smoothing surface singularities.
\newblock {\em Topology}, 25(3):261--291, 1986.

\bibitem{Massuyeau}
Gw{\'e}na{\"e}l Massuyeau.
\newblock Some finiteness properties for the {R}eidemeister-{T}uraev torsion of
  three-manifolds.
\newblock {\em J. Knot Theory Ramifications}, 19(3):405--447, 2010.

\bibitem{Matveev}
Sergei~V. Matveev.
\newblock Generalized surgeries of three-dimensional manifolds and
  representations of homology spheres.
\newblock {\em Mat. Zametki}, 42(2):268--278, 345, 1987 (in Russian). English
  translation: {\it Math. Notes}, 42(1-2):651--656, 1987.

\bibitem{McMullen}
Curtis~T. McMullen.
\newblock The {A}lexander polynomial of a 3-manifold and the {T}hurston norm on
  cohomology.
\newblock {\em Ann. Sci. \'Ecole Norm. Sup. (4)}, 35(2):153--171, 2002.

\bibitem{MT}
Guowu Meng and Clifford~H. Taubes.
\newblock {$\underline{\rm SW}=$} {M}ilnor torsion.
\newblock {\em Math. Res. Lett.}, 3(5):661--674, 1996.

\bibitem{Milnor_duality}
John Milnor.
\newblock A duality theorem for {R}eidemeister torsion.
\newblock {\em Ann. of Math. (2)}, 76:137--147, 1962.

\bibitem{Milnor_Whitehead}
John Milnor.
\newblock Whitehead torsion.
\newblock {\em Bull. Amer. Math. Soc.}, 72:358--426, 1966.

\bibitem{Milnor_cyclic}
John Milnor.
\newblock Infinite cyclic coverings.
\newblock In {\em Conference on the {T}opology of {M}anifolds ({M}ichigan
  {S}tate {U}niv., {E}. {L}ansing, {M}ich., 1967)}, pages 115--133. Prindle,
  Weber \& Schmidt, Boston, Mass., 1968.

\bibitem{Moise}
Edwin~E. Moise.
\newblock Affine structures in {$3$}-manifolds. {V}. {T}he triangulation
  theorem and {H}auptvermutung.
\newblock {\em Ann. of Math. (2)}, 56:96--114, 1952.

\bibitem{MS}
John~W. Morgan and Dennis~P. Sullivan.
\newblock The transversality characteristic class and linking cycles in surgery
  theory.
\newblock {\em Ann. of Math. (2)}, 99:463--544, 1974.

\bibitem{Munkres}
James Munkres.
\newblock Obstructions to the smoothing of piecewise-differentiable
  homeomorphisms.
\newblock {\em Ann. of Math. (2)}, 72:521--554, 1960.

\bibitem{Murakami}
Hitoshi Murakami.
\newblock A weight system derived from the multivariable {C}onway potential
  function.
\newblock {\em J. London Math. Soc. (2)}, 59(2):698--714, 1999.

\bibitem{Nicolaescu_book}
Liviu~I. Nicolaescu.
\newblock {\em The {R}eidemeister torsion of 3-manifolds}, volume~30 of {\em de
  Gruyter Studies in Mathematics}.
\newblock Walter de Gruyter \& Co., Berlin, 2003.

\bibitem{Nicolaescu}
Liviu~I. Nicolaescu.
\newblock Seiberg-{W}itten invariants of rational homology 3-spheres.
\newblock {\em Commun. Contemp. Math.}, 6(6):833--866, 2004.

\bibitem{Ohtsuki}
Tomotada Ohtsuki.
\newblock Finite type invariants of integral homology {$3$}-spheres.
\newblock {\em J. Knot Theory Ramifications}, 5(1):101--115, 1996.

\bibitem{OS_genus}
Peter Ozsv{\'a}th and Zolt{\'a}n Szab{\'o}.
\newblock Holomorphic disks and genus bounds.
\newblock {\em Geom. Topol.}, 8:311--334 (electronic), 2004.

\bibitem{OS}
Peter Ozsv{\'a}th and Zolt{\'a}n Szab{\'o}.
\newblock Holomorphic disks and three-manifold invariants: properties and
  applications.
\newblock {\em Ann. of Math. (2)}, 159(3):1159--1245, 2004.

\bibitem{OS_Thurston_norm}
Peter Ozsv{\'a}th and Zolt{\'a}n Szab{\'o}.
\newblock Link {F}loer homology and the {T}hurston norm.
\newblock {\em J. Amer. Math. Soc.}, 21(3):671--709, 2008.

\bibitem{Porti}
Joan Porti.
\newblock Torsion de {R}eidemeister pour les vari\'et\'es hyperboliques.
\newblock {\em Mem. Amer. Math. Soc.}, 128(612):x+139, 1997.

\bibitem{Postnikov}
Mikhail~M. Postnikov.
\newblock The structure of the ring of intersections of three-dimensional
  manifolds.
\newblock {\em Doklady Akad. Nauk. SSSR (N.S.)}, 61:795--797, 1948 (in
  Russian).

\bibitem{Reidemeister}
Kurt Reidemeister.
\newblock {Homotopieringe und Linsenr{\"a}ume.}
\newblock {\em Abh. Math. Semin. Hamb. Univ.}, 11:102--109, 1935.

\bibitem{Rolfsen}
Dale Rolfsen.
\newblock {\em Knots and links}.
\newblock Publish or Perish Inc., Berkeley, Calif., 1976.
\newblock Mathematics Lecture Series, No. 7.

\bibitem{RS}
Colin~P. Rourke and Brian~J. Sanderson.
\newblock {\em Introduction to piecewise-linear topology}.
\newblock Springer-Verlag, New York, 1972.
\newblock Ergebnisse der Mathematik und ihrer Grenzgebiete, Band 69.

\bibitem{Schumann}
Hans-Georg Schumann.
\newblock \"{U}ber {M}oduln und {G}ruppenbilder.
\newblock {\em Math. Ann.}, 114(1):385--413, 1937.

\bibitem{ST}
Herbert Seifert and William Threlfall.
\newblock {\em Seifert and {T}hrelfall: a textbook of topology}, volume~89 of
  {\em Pure and Applied Mathematics}.
\newblock Academic Press Inc. [Harcourt Brace Jovanovich Publishers], New York,
  1980.

\bibitem{Steenrod}
Norman Steenrod.
\newblock {\em The {T}opology of {F}ibre {B}undles}.
\newblock Princeton Mathematical Series, vol. 14. Princeton University Press,
  Princeton, N. J., 1951.

\bibitem{Sullivan}
Dennis Sullivan.
\newblock On the intersection ring of compact three manifolds.
\newblock {\em Topology}, 14(3):275--277, 1975.

\bibitem{Suzuki}
Masaaki Suzuki.
\newblock The {M}agnus representation of the {T}orelli group {$\mathcal{I}\sb
  {g,1}$} is not faithful for {$g\geq 2$}.
\newblock {\em Proc. Amer. Math. Soc.}, 130(3):909--914 (electronic), 2002.

\bibitem{Suzuki_bis}
Masaaki Suzuki.
\newblock On the kernel of the {M}agnus representation of the {T}orelli group.
\newblock {\em Proc. Amer. Math. Soc.}, 133(6):1865--1872 (electronic), 2005.

\bibitem{Swarup}
Gadde~A. Swarup.
\newblock On a theorem of {C}. {B}. {T}homas.
\newblock {\em J. London Math. Soc. (2)}, 8:13--21, 1974.

\bibitem{Thomas}
Charles~B. Thomas.
\newblock The oriented homotopy type of compact {$3$}-manifolds.
\newblock {\em Proc. London Math. Soc. (3)}, 19:31--44, 1969.

\bibitem{Thurston_norm}
William~P. Thurston.
\newblock A norm for the homology of {$3$}-manifolds.
\newblock {\em Mem. Amer. Math. Soc.}, 59(339):i--vi and 99--130, 1986.

\bibitem{Thurston_book}
William~P. Thurston.
\newblock {\em Three-dimensional geometry and topology. {V}ol. 1}, volume~35 of
  {\em Princeton Mathematical Series}.
\newblock Princeton University Press, Princeton, NJ, 1997.
\newblock Edited by Silvio Levy.

\bibitem{Traldi_Milnor}
Lorenzo Traldi.
\newblock Milnor's invariants and the completions of link modules.
\newblock {\em Trans. Amer. Math. Soc.}, 284(1):401--424, 1984.

\bibitem{Traldi_Conway}
Lorenzo Traldi.
\newblock Conway's potential function and its {T}aylor series.
\newblock {\em Kobe J. Math.}, 5(2):233--263, 1988.

\bibitem{Turaev_Alexander}
Vladimir Turaev.
\newblock The {A}lexander polynomial of a three-dimensional manifold.
\newblock {\em Mat. Sb. (N.S.)}, 97(139):341--359, 463, 1975 (in Russian).
  English translation: {\it Math. USSR Sb.}, 26:313--329, 1975.

\bibitem{Turaev_Alexander_torsion}
Vladimir Turaev.
\newblock Reidemeister torsion and the {A}lexander polynomial.
\newblock {\em Mat. Sb. (N.S.)}, 18(66):252--270, 1976 (in Russian). English
  translation: {\it Math. USSR Sb.}, 30:221--237, 1976.

\bibitem{Turaev_cohomology}
Vladimir Turaev.
\newblock Cohomology rings, linking coefficient forms and invariants of spin
  structures in three-dimensional manifolds.
\newblock {\em Mat. Sb. (N.S.)}, 120(162):68--83, 143, 1983 (in Russian).
  English translation: {\it Math. USSR Sb.}, 48:65--79, 1984.

\bibitem{Turaev_knot}
Vladimir Turaev.
\newblock Reidemeister torsion in knot theory.
\newblock {\em Uspekhi Mat. Nauk}, 41(1):97--147, 240, 1986 (in Russian).
  English translation: {\it Russian Math. Surveys}, 41:119--182, 1986.

\bibitem{Turaev_Euler}
Vladimir Turaev.
\newblock Euler structures, nonsingular vector fields, and {R}eidemeister-type
  torsions.
\newblock {\em Izv. Akad. Nauk SSSR Ser. Mat.}, 53(3):607--643, 672, 1989 (in
  Russian). English translation: {\it Math. USSR Izvestia}, 34:627--662, 1990.

\bibitem{Turaev_spin_c}
Vladimir Turaev.
\newblock Torsion invariants of {${\rm Spin}\sp c$}-structures on
  {$3$}-manifolds.
\newblock {\em Math. Res. Lett.}, 4(5):679--695, 1997.

\bibitem{Turaev_SW}
Vladimir Turaev.
\newblock A combinatorial formulation for the {S}eiberg-{W}itten invariants of
  {$3$}-manifolds.
\newblock {\em Math. Res. Lett.}, 5(5):583--598, 1998.

\bibitem{Turaev_book}
Vladimir Turaev.
\newblock {\em Introduction to combinatorial torsions}.
\newblock Lectures in Mathematics ETH Z{\"u}rich. Birkh{\"a}user Verlag, Basel,
  2001.
\newblock Notes taken by Felix Schlenk.

\bibitem{Turaev_norm}
Vladimir Turaev.
\newblock A homological estimate for the {T}hurston norm.
\newblock Preprint \texttt{arXiv:math/0207267}, 2002.

\bibitem{Turaev_book_dim_3}
Vladimir Turaev.
\newblock {\em Torsions of {$3$}-dimensional manifolds}, volume 208 of {\em
  Progress in Mathematics}.
\newblock Birkh{\"a}user Verlag, Basel, 2002.

\bibitem{Wall}
Charles T.~C. Wall.
\newblock Quadratic forms on finite groups, and related topics.
\newblock {\em Topology}, 2:281--298, 1963.

\bibitem{Wallace}
Andrew~H. Wallace.
\newblock Modifications and cobounding manifolds.
\newblock {\em Canad. J. Math.}, 12:503--528, 1960.

\bibitem{Whitehead_smooth}
J.~H.~C. Whitehead.
\newblock Manifolds with transverse fields in euclidean space.
\newblock {\em Ann. of Math. (2)}, 73:154--212, 1961.

\end{thebibliography}

\end{document}